\documentclass[12pt]{amsart}

\usepackage[utf8]{inputenc}

\usepackage{graphicx}
\usepackage{amsmath}
\usepackage{amssymb}
\usepackage{float}
\usepackage{fullpage}
\usepackage{comment}
\usepackage{mathabx}
\usepackage{enumitem}  
\usepackage[pdftex,citecolor=green,linkcolor=red]{hyperref}
\usepackage{mathtools}
\usepackage{tikz} 
\usetikzlibrary{arrows.meta, positioning} 

\newtheorem{theorem}{Theorem}[section]
\newtheorem{lemma}[theorem]{Lemma}
\newtheorem{corollary}[theorem]{Corollary}
\newtheorem{proposition}[theorem]{Proposition}

\newtheorem{prop}[theorem]{Proposition}
\theoremstyle{definition}
\newtheorem{definition}[theorem]{Definition}

\theoremstyle{remark}
\newtheorem{remark}[theorem]{Remark}

\numberwithin{equation}{section}

\newcommand{\asum}{\sideset{}{^{\ast}}\sum}

\newcommand{\1}{\mathbf{1}}

\newcommand{\R}{{\mathbb R}}

\newcommand{\Z}{{\mathbb Z}}
\newcommand{\N}{{\mathbb N}}

\newcommand{\cond}[1]{\textnormal{cond}{#1}}

\renewcommand{\phi}{\varphi}
\renewcommand\d{\,\mathrm{d}}
\let\oldpmod\pmod
\renewcommand{\pmod}[1]{\hspace{-0.12cm}\oldpmod {#1}}

\begin{document}

\title{The exceptional set in Goldbach's problem with two Chen primes}
\newcommand{\blankbox}[2]{%
	\parbox{\columnwidth}{\centering
		\setlength{\fboxsep}{0pt}%
		\fbox{\raisebox{0pt}[#2]{\hspace{#1}}}%
	}%
}
\author{Lasse Grimmelt}
\address{Department of Pure Mathematics and Mathematical Statistics, University of Cambridge, Cambridge CB3 0WB, UK}
\email{lpg31@cam.ac.uk}

\author{Joni Ter\"{a}v\"{a}inen}
\address{Department of Pure Mathematics and Mathematical Statistics, University of Cambridge, Cambridge CB3 0WB, UK}
\email{joni.p.teravainen@gmail.com}

\begin{abstract} We show that all natural numbers $n\equiv 4\pmod 6$ are the sum of two Chen primes (primes $p$ such that $p+2$ has at most two prime factors), apart from a power-saving set of exceptions. This improves on various previous results and is optimal, barring substantial progress on the twin prime or binary Goldbach conjectures. 

The proof is based on constructing a non-negative model for the Chen primes in a suitable approximate sense. The key ingredient for this is showing that the primes are power saving Fourier close to the rough number Cram\'er model times $(1+o(1))$. Additionally, we develop an efficient sieving strategy for additive problems that for large prime factors utilises a power-saving variant of the Bombieri--Vinogradov theorem and for small prime factors a fundamental Lemma type result.

\end{abstract}

\maketitle

    \tableofcontents
	
\section{Introduction}
	
In this paper, we study the exceptional set in the binary Goldbach problem with almost twin primes. There are different ways to define what an almost twin prime should be. We shall consider here those primes $p$ for which $p+2$ has few prime factors. 
We define 
\begin{align*}
	\mathbb{P}_k\coloneqq \{m\in \mathbb{N} : m \text{ has at most } k \text{ prime factors} \},
\end{align*}
and abbreviate $\mathbb{P}_1$ as $\mathbb{P}$. Thanks to a celebrated result of Chen~\cite{chen}, we know that there are infinitely many primes $p$ such that $p+2\in \mathbb{P}_2$, and we call these primes \emph{Chen primes}. 

The binary Goldbach problem states that every even integer $n\geq 4$ is the sum of two primes. Over the years, various results have been proved showing this for all $n$ outside an exceptional set. Montgomery and Vaughan~\cite{mv} were the first to show a power-saving bound for this exceptional set; thus, there exists $\delta>0$ such that all but at most $O(N^{1-\delta})$ even integers $n\leq N$ are the sum of two primes. Apart from the precise exponent $\delta$, for which the current record is due to Pintz~\cite{pintz2018}, this is still the state of the art. 

We prove the following hybrid of the results of Chen and Montgomery--Vaughan.

\begin{theorem} \label{MT1}
	There is a constant $\delta>0$ such that all but $O(N^{1-\delta})$ natural numbers $m\leq N, m\equiv 4 \pmod{6}$ are the sum of two Chen primes.
\end{theorem}
Both $\delta$ and the implied constant in this theorem are effective and could in principle be computed. Apart from the value of $\delta$, any improvement to Theorem~\ref{MT1} would imply drastic new results for either the binary Goldbach conjecture or the twin prime conjecture, and so our result seems to be the best possible in the absence of a major advance on one of these problems. 

\subsection{Previous works}
	
Based on standard heuristics of Hardy--Littlewood type, we expect that every large integer $m\equiv 4\pmod 6$ can be written as $m=p_1+p_2$ with $p_i,p_i+2\in \mathbb{P}$ for $i\in \{1,2\}$. Needless to say, this is far out of reach as it would imply as special cases both the twin prime conjecture and the binary Goldbach conjecture (for large numbers congruent to $4\pmod 6$). Previously, various approximations to this have been proved. To state them, we define the related exceptional set as 
\begin{align*}
	E(N,k_1,k_2)\coloneqq |\{m\leq N: m\equiv 4\pmod 6,\,\, m\neq p_1+p_2\,\, \forall p_i\in \mathbb{P}\cap(\mathbb{P}_{k_i}-2)\}|.
\end{align*}
Here (and in Theorem~\ref{MT1}) the restriction to $m\equiv 4 \pmod 6$ is imposed since for any fixed $k\in \mathbb{N}$ almost all elements of $\mathbb{P}\cap (\mathbb{P}_k-2)$ are $\equiv 5 \pmod 6$.

Theorem~\ref{MT1} improves upon the following estimates. For any $A>0$,  we have
\begin{align}
    \begin{split}\label{eqn1}
	   E(N,5,7)&\ll_A N(\log N)^{-A}, \quad \quad \text{ due to Tolev~\cite{tolev}}\\
		E(N,3,\infty)&\ll_A N(\log N)^{-A}, \quad \quad \,\,\text{due to Meng~\cite{meng}}\\
		E(N,2,7)&\ll_A N(\log N)^{-A} \quad \quad \,\,\,\,\, \text{due to Matom\"aki~\cite{matomaki}},
	\end{split}
\end{align}
where $\mathbb{P}_\infty=\mathbb{N}$. All the implied constants appearing in~\eqref{eqn1} are ineffective. 

Matom\"aki and Shao in~\cite{matomaki-shao} solved the ternary Goldbach problem for the Chen primes by showing that every sufficiently large integer $m\equiv 3\pmod 6$ is the sum of three Chen primes. Even though it is not explicitly stated there, the methods of~\cite{matomaki-shao} can be adapted to the binary case with only minor modifications (see~\cite[Proposition 2.1]{tera-goldbach} for a binary version of the transference principle in~\cite{matomaki-shao}), and a preliminary calculation suggests that this leads to an estimate of the form
\begin{align*}
	E(N,2,2)\ll N(\log\log\log N)^{-\delta}
\end{align*}
for some small $\delta>0$. The expected saving of size $(\log \log \log N)^{-\delta}$ is a consequence of three factors: first, the Siegel--Walfisz theorem restricts the range of moduli to a power of $\log N$; second, the $W$-trick loses another $\log$ iteration; third, the final $\log$ iteration is lost due to the sparseness of high-rank Bohr sets. 

It should also be mentioned that this paper improves on a previous preprint of the authors~\cite{grimmelt-teravainen}, in which the weaker result $E(N,2,3)\ll N^{1-\delta}$ was shown. The approach we take here is technically simpler in addition to producing a stronger result. In particular, the current approach does not require the spectral theory of automorphic forms, which featured in~\cite{grimmelt-teravainen} due to the use of Kloosterman sum estimates of Deshouillers--Iwaniec~\cite{DI} and their variants. Apart from a similar treatment of sieves for large prime factors, the approach here is independent of~\cite{grimmelt-teravainen} and supersedes it.
 
\subsection{Background and strategy}

We now give an informal overview of the proof strategy for Theorem~\ref{MT1}.

Let $\Lambda_2(n)$ be a normalised indicator of $n$ having at most two prime divisors, defined precisely in~\eqref{eq:L2def}. We define a normalised indicator for the Chen primes as 
$$\Lambda_{\text{Chen}}(n)\coloneqq \Lambda(n)\Lambda_2(n+2).$$
Theorem~\ref{MT1} follows if we can show a reasonable lower bound for
\begin{align}\label{eq:maintarget}
    	\sum_{n_1+n_2=m}\Lambda_{\text{Chen}}(n_1)\Lambda_{\text{Chen}}(n_2)
\end{align}
for all but $O(N^{1-\delta})$ natural numbers $m\leq N, m\equiv 4 \pmod{6}$. Since we do not understand $\Lambda_{\text{Chen}}(n)$ directly (we do not even know the mean value of this function), we need to employ a suitably chosen sieve minorant $\omega_{\text{Chen}}(n)\leq \Lambda_{\text{Chen}}(n)$. Chen's famous result ensures the existence of such a minorant (sometimes called Chen's sieve) with
\begin{align}
    \sum_{n\leq N} \omega_{\text{Chen}}(n)= (c+o(1))N
\end{align}
for some constant $c>0$. However, we cannot immediately insert this minorant $\omega_{\text{Chen}}(n)$ into~\eqref{eq:maintarget}, since
\begin{align}\label{eq:sievelowerbounds}
	\sum_{n_1+n_2=m}\Lambda_{\text{Chen}}(n_1)\Lambda_{\text{Chen}}(n_2)\geq  \sum_{n_1+n_2=m}\omega_{\text{Chen}}(n_1)\omega_{\text{Chen}}(n_2),
\end{align}
is not necessarily true as $\omega_{\text{Chen}}(n)$ can also take negative values. In the existing literature, there are two ways to overcome this issue. First, one may apply a so-called vector sieve inequality. However, this weakens the sieves and is the reason that the results in~\eqref{eqn1} cannot reach two Chen primes. Second, one can apply a transference principle based approach. This approach was introduced to this problem by Matom\"aki and Shao~\cite{matomaki-shao} and only gives a relatively weak saving on the exceptional set, as mentioned above. Our strategy, inspired by work of the first author~\cite{grimmelt-nonneg} and Green's recent work~\cite{green-sarkozy}, is to construct a non-negative approximant to Chen's sieve $\omega_{\text{Chen}}$ that makes it possible to restore~\eqref{eq:sievelowerbounds}. Roughly speaking, the strategy is based on three ingredients: 
\begin{itemize}
    \item We use Fourier closeness to approximate by suitable arithmetic functions.
     \item We treat sieves efficiently by using a power-saving Bombieri--Vinogradov theorem for large moduli and by adapting the fundamental lemma of sieve theory to our additive setting.
    \item We prove that primes can be replaced by the Cram\'er model\footnote{This model, also known as the Cram\'er--Granville model and used in many other recent works on additive problems for the primes (see e.g.~\cite{tt-jems},~\cite{green-sarkozy},~\cite{MSTT},~\cite{MRSTT},~\cite{leng}), is a normalised indicator of integers having no small prime factors; see~\eqref{eq:Cramer} below.} in an approximate sense. 
\end{itemize}

Similarly to other works that consider the power-saving regime (for example,~\cite{mv}, \cite{green-roth}), we need to consider the effect of a possible exceptional zero separately. See Subsection~\ref{subsec:introexc} for more details.

We next give some details for the three aforementioned ingredients and then outline how they come together in the proof of Theorem~\ref{MT1}.

\subsubsection{A general approximation framework}

A key aspect of the transference principle is to use convolutions with Bohr sets to extract a Fourier-close approximant for any function. To create a non-negative approximant for $\omega_{\text{Chen}}(n)$, we develop this perspective further in the given additive context and write $f\approx g$ to roughly mean that
\begin{align*}
    \sum_{n_1+n_2=m}\bigl(f(n_1)-g(n_1)\bigr)\Lambda_{\text{Chen}}(n_2)
\end{align*}
is negligible for all natural numbers $m\leq N$ outside of a power-saving exceptional set. See Definition~\ref{def_approx} below for the rigorous definition. One way of obtaining approximants is by Fourier analysis. In this sense, the major arc contribution in the classical circle method can be seen as approximating the original function, if a sufficient minor arc bound is provided. However, our more general perspective allows on the one hand for different approximants than the major arcs and on the other hand also for other ways than Fourier-closeness to achieve approximation. Our main goal is to show that there exists a function $\mathcal{T}(n)\geq 0$ such that 
\begin{align}\label{eq:Chen_non-negativ}
    \omega_{\text{Chen}}(n) \approx \mathcal{T}(n),
\end{align}
which makes it possible to restore~\eqref{eq:sievelowerbounds}.
\subsubsection{Sieves}
Our construction of the approximant $\mathcal{T}$ is intricately linked to sieves. In particular, we make extensive use of sieving separately for small primes, say less than $P$; we call such sieves \emph{pre-sieves}. We denote by $r_{P}$ the Cram\'er model, given by
\begin{align}\label{eq:Cramer}
   r_P(n)=\prod_{p< P} \left(1-\frac{1}{p}\right)^{-1} \1_{p\nmid n}.
\end{align}

This function is a normalised indicator for $n$ being $P$-rough (i.e., $n$ having no prime factors less than $P$).  We further denote by $\omega, \Omega$ in this sketch the associated lower and upper bound pre-sieves (normalised to have mean comparable to $1$), so that $\omega \leq r_{P} \leq \Omega$. See Definition~\ref{def:presieve} for the precise choice. By  $\omega_\textup{M}$, $\Omega_{\textup{M}}$ we denote the lower and upper bound sieves handling the primes larger than $P$; we call these informally \emph{main-sieves}. See Definition~\ref{def:mainsieve} for their precise requirements. Importantly, to achieve a power saving we take $P=N^{\delta_0}$ to be a small power of $N$, and the obtained saving $\delta$ in Theorem~\ref{MT1} is a function of $\delta_0$. We do not specify the level of the sieves for this sketch.

We will prove two types of results for sieves. First, we show that there exists a constant $c_{\textup{M}}>0$ such that we can replace the main-sieves with it, that is
\begin{align}\label{eq:intromainsieveremoval}
    \Lambda(n)\omega(n+2) \omega_\textup{M}(n+2) \approx c_{\textup{M}}  \Lambda(n)\omega(n+2)
    \end{align}
    and similarly for $\Omega, \Omega_M$ and combinations thereof. 
    The approximation~\eqref{eq:intromainsieveremoval} is based on the fact that the Bombieri--Vinogradov theorem can produce a power saving, provided the main term includes the contribution of low conductor characters. See Drappeau's work~\cite{drappeau} for a related strategy. 
Second, as a consequence of fundamental lemma-type results, we show that the rough numbers can be exchanged with upper and lower bound sieves in the sense that
\begin{align}\label{eq:introfundamental}
    f(n)\omega(n+2)\approx  f(n)\Omega(n+2) \approx f(n)r_P(n+2),
\end{align}
for $f\in\{\Lambda,r_P\}$.

\subsubsection{Replacing primes by the Cram\'er model}
The technical centrepiece of the proof of Theorem~\ref{MT1} is replacing the primes by the Cram\'er model~\eqref{eq:Cramer}.

To be more precise, we want to show in the language of the previous subsection that
\begin{align}\label{eq:LambdatorP}
    \Lambda(n)\omega(n+2)\approx r_P(n)\omega(n+2).
\end{align}
The key ingredient for this is showing that $\Lambda$ is Fourier-close to $r_P(n)(1+o(1))$ with a power saving. Since it allows one to replace $\Lambda$ by a function that is on the one hand non-negative and on the other hand can be upper and lower bounded effectively by sieves, we expect that this result may be useful for other purposes. We state here a simplified version that assumes that there is no exceptional zero; see Theorem~\ref{thm:gallagher_application} for the full statement. 

\begin{theorem}[Fourier-approximating primes by Cram\'er model]\label{thm:F-P-C_intro}  Let $\exp((\log N)^{1/2})\leq R^{40}\leq P\leq N^{1/5}$ and assume that there is no exceptional zero of level $R^2$ and quality $\kappa$ (see Definition~\ref{def:exceptional}). Let $h_\xi$ be the multiplicative function supported on square-free integers only and given on the primes by 
 \begin{align*}
        h_{\xi}(p)=\min\Bigl\{1,10(1+|\xi|)\frac{\log p}{\log R} \Bigr\},
    \end{align*}
and let 
\begin{align*}
    H_R(n)=\tau(n)(\log R) \int_{\mathbb{R}} \frac{h_{\xi}(n)}{(1+|\xi|)^{10}}\d \xi.
\end{align*}
There exists an arithmetic function $E$ and a constant $c>0$ such that the following holds. We have
		\begin{align*}
		  \sup_{\alpha \in \R} \bigl| \sum_{N/2<n\leq N}(\Lambda(n)-r_{P}(n)-E(n))e(\alpha n) \bigr| \ll N R^{-1/3}
		\end{align*}
	    and
	   \begin{align*}
          |E(n)|&\ll H_R(n) \Bigl(\exp\bigl( -c \kappa \frac{\log N}{\log R} -c \frac{\log N}{\log P}\bigr)+\frac{\log N}{\log P}\exp\bigl(-c\frac{\log P}{\log R} \bigr)\Bigr).
	   \end{align*}
\end{theorem}

This theorem is inspired by Green's work~\cite{green-sarkozy}, where $\Lambda$ was power-saving Fourier approximated by
\begin{align}\label{eq:Lambda_r}
    \sum_{q\leq R}\frac{\mu(q)c_q(n)}{\varphi(q)}
\end{align}
plus additional terms that account for potential bad zeros of Dirichlet $L$-functions with conductor up to $R$, which in our notation are contained in $E$. Here as there, the crucial saving $\exp\bigl( -c \kappa \frac{\log N}{\log R} \bigr)$ comes from zero-density estimates, for us in the form of Gallagher's prime number theorem~\cite{Gal}.  Going one step further from~\eqref{eq:Lambda_r} to $r_P$ introduces the other components in $E$ and requires both $\frac{\log N}{\log P}$ and $\frac{\log P}{\log R}$ to be sufficiently large. To introduce $r_P$, we first bound the difference of \eqref{eq:Lambda_r} and $r_P$. Doing this with Green's explicit formula based strategy would impose considerable technical challenges. We instead extract a major arc model by convolving with a physical space major arc indicator, see $b_R$ in Definition~\ref{def:bdef}.

Another difference compared to~\cite{green-sarkozy} is that we upper bound the error $E$ in physical space by the term involving $H_R$. The estimation of $E$ involves certain functions that resemble the non-squared Selberg sieve weights and our treatment is inspired by the work of the Polymath 8b project (\cite[Proposition 4.2]{polymath}). Observe that $h_\xi(n)$ becomes negligible when $n$ has prime divisors that are much smaller than $R^{1/(1+|\xi|)}$ and one can show that
\begin{align}\label{eq:HRintro}
    \sum_{n\leq N} H_R(n)\ll N.
\end{align}
In other words, this means that $H_R$ behaves similarly to the Cram\'er model of range $R$. This justifies the intuition that $\Lambda$ is power-saving Fourier-close to $r_P(1+o(1))$ (or more precisely $r_P+o(r_R)$). We remark that under the assumption of GRH, the error term $E$ can be omitted from the statement. In contrast to~\cite{green-sarkozy}, this does not, however, considerably simplify the rest of the proof of Theorem~\ref{MT1}.

\subsubsection{Sketch of proof, unexceptional case}
\label{sec:1.2.5}
    Assume first that there is no exceptional zero for $L$-functions up to a certain conductor (see Definition~\ref{def:exceptional}). Let $\Lambda_{E_3}$ denote a weighted indicator for numbers having exactly three prime divisors in certain suitable ranges; this is a function that facilitates the sieve switching approach that Chen~\cite{chen} pioneered. In a slightly simplified form, our chain of lower bounds and approximants for Chen primes takes the following shape. There are some absolute constants $0<c_2<c_1<1$ such that for $P=N^{\delta_0}$  we have
	\begin{align}
		\omega_{\text{Chen}}(n)&\geq \Lambda(n) \omega \omega_{\textup{M}}(n+2)-\Omega\Omega_{\textup{M}}(n)\Lambda_{E_3}(n+2)\label{eq_outline1} \\
		&\approx c_1\Lambda(n) \omega (n+2)- c_2\Omega(n)\Lambda_{E_3}(n+2)\label{eq_outline2}\\
		&\approx c_1 r_{P}(n) \omega(n+2) - c_2\Omega(n)r_{P}(n+2)\label{eq_outline3}\\
		&\approx c_1 r_{P}(n) r_{P} (n+2)-c_2 r_{P}(n) r_{P} (n+2) \label{eq_outline4}\\
		&= (c_1-c_2)r_{P}(n) r_{P} (n+2). \nonumber
	\end{align}
    We remark that in the rigorous proof, for technical reasons, we split the single roughness level $P$ into $P_0, P_1$, see remark \ref{rem:roughnessplit}. This achieves the goal~\eqref{eq:Chen_non-negativ} of constructing a non-negative approximant, and Theorem~\ref{MT1} can be deduced as follows. For some error terms $\mathcal{E}_i(m)$ that are negligible outside of a power-saving exceptional set, we have
\begin{align}
    &\nonumber \,\,\sum_{n_1+n_2=m}\Lambda_{\text{Chen}}(n_1)\Lambda_{\text{Chen}}(n_2)\\
    \nonumber
    \geq& (c_1-c_2)\sum_{n_1+n_2=m}\Lambda(n_1)\Lambda_2(n_1+2)r_{P}(n_2)r_{P}(n_2+2)+\mathcal{E}_1(m)\\
	\nonumber
    \geq& (c_1-c_2)^2\sum_{n_1+n_2=m}r_{P}(n_1)r_{P}(n_1+2)r_{P}(n_2)r_{P}(n_2+2)+\mathcal{E}_2(m)\\
	\label{eq:outlinepresievasymp}=& (c_1-c_2)^2 m \mathfrak{S}(m)+\mathcal{E}_3(m),
\end{align}
where $\mathfrak{S}(m)$ is the singular series that encodes local solution densities and that agrees with the expected Hardy--Littlewood heuristics. The step~\eqref{eq:outlinepresievasymp} is a consequence of a fundamental lemma-type result that allows us to handle $P$-rough numbers in essentially any additive setup easily with high precision, as long as $(\log m)/(\log P)$ is large enough in absolute terms. 

We now outline what goes into each of the steps in the chain of inequalities and approximations~\eqref{eq_outline1}--\eqref{eq_outline4}. The sieve $\omega_{\text{Chen}}$ in~\eqref{eq_outline1} is a combination of the usual sieve-switching approach for Chen primes with a separate sieve for small prime factors. In reality, one needs to be slightly more careful, as the product of two lower bound sieves is not a lower bound sieve. Note that for our purposes $\Lambda_{E_3}$ behaves the same way as $\Lambda$. In step~\eqref{eq_outline2}, we remove the contribution of the main-sieves by~\eqref{eq:intromainsieveremoval}. In step~\eqref{eq_outline3}, we replace the primes (respectively, the $E_3$-numbers) with the much easier-to-handle $P$-Cram\'er model; this follows from Theorem~\ref{thm:F-P-C_intro} or more precisely from~\eqref{eq:LambdatorP}. The final step~\eqref{eq_outline4} is a consequence of~\eqref{eq:introfundamental}.

\subsubsection{Sketch of proof, exceptional case}\label{subsec:introexc}
If the exceptional zero $\widetilde{\beta}$ with exceptional character $\widetilde{\chi}$ exists, there are two additional technical complications, apart from which the general strategy is identical.

First, we have to include a correction term of the form $1-\widetilde{\chi}(n)n^{\widetilde{\beta}-1}$ when replacing $\Lambda$ or $\Lambda_{E_3}$ with $r_P$. This will, in some cases, reduce the number of expected representations by introducing a factor related to $(1-\widetilde{\beta})\log N$. We need to compensate for this possibly smaller main term with a better saving in the affected cases. 
			
    Second, following the strategy of the unexceptional case without modification would lead to roughly
			\begin{align*}
				\Lambda(n)\Lambda_2(n+2)\gtrapprox  r_{P}(n)r_{P}(n+2)\bigl(c_1(1-\widetilde{\chi}(n)n^{\widetilde{\beta}-1})-c_2( 1-\widetilde{\chi}(n+2)n^{\widetilde{\beta}-1}) \bigr).
			\end{align*}
             This function is not, and cannot be well approximated by, a non-negative model. Indeed, it has a negative average in residue classes $b$ modulo $\widetilde{r}$ with $\widetilde{\chi}(b)=1$ and $\widetilde{\chi}(b+2)=-1$. To circumvent this issue, we exclude this bad case before the application of Chen's lower bound by multiplying with a factor of $(1-\widetilde{\chi}(n))/2=\1_{\widetilde{\chi}(n)=-1}$ or $(1+\widetilde{\chi}(n))(1+\widetilde{\chi}(n+2))/4=\1_{\widetilde{\chi}(n)=\widetilde{\chi}(n+2)=1}$.

\subsection{Acknowledgements}
The first author received funding from the European Research Council (ERC) under the European Union's Horizon research and innovation programme, grant 
agreement no. 851318 and no. 101162746. The second author was supported by Academy of Finland grant no. 362303 and funding from the European Union's Horizon Europe research and innovation programme under Marie Sk\l{}odowska-Curie grant agreement no. 101058904 and ERC grant agreement no. 101162746. The authors thank James Maynard for helpful comments and discussions and Kaisa Matom\"aki for encouragement to obtain the optimal result.	

\section{Notation and structure}

\subsection{Notation}

Euler's totient function, the divisor function, and the von Mangoldt function are denoted by $\varphi(n),\tau(n), \Lambda(n)$, respectively. The $k$-fold divisor function is written as $\tau_k(n)$. By $a(q)$ in a summation index we mean that $a$ runs over a complete system of residues modulo $q$; similarly by $\chi(q)$ in a summation index we mean that $\chi$ runs over the multiplicative characters modulo $q$. By $\asum$ we mean that a summation is restricted to primitive residue classes or primitive characters. 

The symbols $*,\star$ denote the additive and multiplicative convolution of two arithmetic functions, respectively. For any functions $f,g\colon \mathbb{Z}\to \mathbb{C}$, we write
\begin{align*}
    f*g(n)&=\sum_{m\in \mathbb{Z}}f(m)g(n-m), \\
    f\star g(n)&=\sum_{d\mid n}f(d)g(n/d).
\end{align*}

We denote by $\|\cdot\|_1$, $\|\cdot\|_2$, and $\|\cdot\|_\infty$ the $L^1$, $L^2$, and $L^\infty$ norms of arithmetic sequences (we will use $ \|\cdot\|_{\infty}^{\wedge}$ in a Fourier context, see Definition~\ref{def:Fouriernorm}).

We call $n$ $P$-rough if it has no prime factor less than $P$ (equivalently every prime factor of $n$ is $\geq P$, equivalently $(n,\prod_{p<P}p)=1$).

We use $(\cdot)^\pm$ to denote the shift operator by $\pm 2$. We only use this notation in statements that benefit from not having the variable explicit, for example
\begin{align*}
    \| fg^+\|_1=\sum_{n\in \mathbb{Z}} |f(n)g(n+2)| \quad  \text{ and } \quad f*g^+(n)=\sum_{m\in \mathbb{Z}}f(m)g(n+2-m).
\end{align*}

We use the standard exponential phase notation 
\begin{align*}
    e(z)\coloneqq e^{2\pi i z},\quad \quad
    e_r(z)\coloneqq e(z/r),
\end{align*}
and also denote Ramanujan's sum by
\[c_q(n)=\asum_{b(q)}e_q(bn).\]
We write $n\sim N$ to denote the condition $N/2<n\leq N$.

\subsection{Choices of parameters}
We use $c$ exclusively in error terms to mean the existence of an absolute constant $c>0$ such that the statement holds. It can vary from line to line.

We fix a global parameter $\delta_1>0$ on which the saving in Theorem~\ref{MT1} depends. We assume that the final size of $N$ in Theorem~\ref{MT1} is chosen sufficiently large in terms of $\delta_1$, whenever necessary. 

The letter $R$ denotes major arc cutoff ranges and $P$, $D$ denote sifting ranges and sifting levels, respectively. The proof of Theorem~\ref{MT1} will involve certain parameter choices that need to be powers of $N$ that have suitable relative sizes. We state their final choices now, but will restate them in the relevant lemmas and propositions. 

\begin{definition}\label{def:para}
We introduce the following parameters depending on $N$ and $\delta_1$:
\begin{alignat*}{3}
D_{\mathrm{M},1}&=N^{1/3-\delta_1}, \quad &D_{\mathrm{M},2}&=N^{1/6-\delta_1},\quad   &D_1&=N^{\delta_1^3/100}, \\
    P_0&=N^{\delta_1}, \quad &P_1&=N^{\delta_1^4},\\
    R_0&=N^{\delta_1^3},\quad  &R_1&=N^{\delta_1^4/100},\quad   &\widetilde{R}&=N^{2\delta_1^5}.
\end{alignat*}

Then we have the hierarchy
\begin{align*}
D_{\mathrm{M},1}\ggg D_{\mathrm{M},2}\ggg P_0\ggg R_0\gg D_1\ggg P_1\gg R_1\ggg \widetilde{R}.     
\end{align*}

\end{definition}

\subsection{Sieve definitions}

As seen in the chain of inequalities following~\eqref{eq_outline1}, pre- and main-sieves play a crucial role in proving Theorem~\ref{MT1} and they will appear throughout the paper. We now introduce relevant definitions and notation and give some motivation.

Generalising the notion of the Cram\'er model, for $P,Q\geq 2$, we define a normalised indicator for integers with no prime factor in the interval $[Q,P)$ as
\begin{align*}
r_{[Q,P)}(n)\coloneqq \prod_{Q\leq p < P}\left(1-\frac{1}{p}\right)^{-1}\1_{p\nmid n}.    
\end{align*}
Thus
\begin{align*}
    r_{P}(n)=r_{[1,P)}(n).
\end{align*}
With this normalisation, by the fundamental lemma of the sieve~\cite[Lemma 6.8]{cribro}, for $N\geq P$ we have
\begin{align*}
\frac{1}{N}\sum_{n\leq N}r_P(n)=1+O\left(\exp\left(-\frac{1}{2}\frac{\log N}{\log P}\right)+N^{-1/2}\log P\right).
\end{align*}

Next we define pre-sieves that sift out primes up to $P_1$ (where $P_1$ and other parameters are defined in Definition~\ref{def:para}).

\begin{definition}[Pre-sieves] \label{def:presieve}
    Denote by $\omega$ and $\Omega$ the following normalised lower bound and upper bound sieves. Define them to handle the primes $2$ and $3$ directly and sift for the primes $5\leq p < P_1=N^{\delta_1^4}$ with a beta sieve with $\beta=200$ and level $D_1=N^{\delta_1^3/100}$. That is, 
    \begin{align*}
        \omega(n)\coloneqq \1_{(n,6)=1}\prod_{p < P_1}\left(1-\frac{1}{p}\right)^{-1}\sum_{d\mid n}\lambda^{-}(d),\quad \Omega(n)\coloneqq \1_{(n,6)=1}\prod_{p < P_1}\left(1-\frac{1}{p}\right)^{-1}\sum_{d\mid n}\lambda^{+}(d),
    \end{align*}
    where $\lambda^{\pm}(d)=\mu(d)\1_{d\mid \prod_{5\leq p < P_1}p}\1_{d\in \mathcal{D}^{\pm}}$ and $\mathcal{D}^-, \mathcal{D}^{+}$ are given by~\cite[equation (6.55)]{cribro} with $\beta=200$.
\end{definition}
Choosing pre-sieves as beta sieves gives us a convenient way of encoding a fundamental lemma-type result; see Lemma~\ref{lem_fundlem}. We choose the relatively large value $\beta=200$ to absorb local density fluctuations of our additive problem of interest; see the proof of Proposition~\ref{prop_presieveasmyp}. The normalisation is chosen so that, if $D_1$ is large compared to $P_1$, the pre-sieves have mean $1+o(1)$. The primes $2$ and $3$ are involved in the local solubility of our additive problem, and it is technically easier to handle them separately. Observe that by definition,
\begin{align*}
    \omega(n) \leq r_{P_1}(n) \leq \Omega(n).
\end{align*}

The complementary part to the pre-sieves of Definition~\ref{def:presieve} is played by main-sieves that sift for primes larger than $P_1$. It will be important that the weights of main-sieves have certain factorisation properties. Recall that a function $\lambda$ supported on $[1,D]$ is well-factorable if for any $R,S\geq 1$ with $D=RS$ we can write $\lambda=\gamma_1 \star \gamma_2$ with $\gamma_1$ and $\gamma_2$ being bounded coefficients, supported on $[1,R]$ and $[1,S]$, respectively. 

\begin{definition}[Main-sieves] \label{def:mainsieve} Let $D_{\mathrm{M},1}=N^{1/3-\delta_1}, D_{\mathrm{M},2}=N^{1/6-\delta_1}, P_1=N^{\delta_1^4}$.
    We say $f_\textup{M}$ is a \emph{main-sieve}, if for some $|C_{f_\textup{M}}|\ll \log N$ we have
    \begin{align*}
       f_\textup{M}(n)=C_{f_\textup{M}}\sum_{d\mid n}\lambda(d),
 \end{align*}
where the function $\lambda$ fulfils the following.
\begin{enumerate}
	\item $\lambda(1)=1$ and $\lambda(d)\neq 0$ implies $p\nmid d$ for all $p < P_1$.
	\item  $\lambda=\lambda_{1}\star\lambda_{2}$ with $|\lambda_{i}(d)|\ll \tau(d)^{O(1)}$, and  $\lambda_{1}$ is supported on $[1,D_{\mathrm{M},1}]$, and $\lambda_{2}$ is a sum of $O_{\delta_1}(1)$ many well-factorable functions supported on $[1,D_{\mathrm{M},2}]$. 
\end{enumerate}
\end{definition}
 Condition (1) ensures a good saving when the main-sieves are removed in Proposition~\ref{prop_MSremoval}, whereas the factorisation conditions in (2) make sure that $\Lambda(n)\omega(n+2)$ fulfils minor arc bounds, see Lemma~\ref{lem:minor}, which in particular uses work of Matom\"aki~\cite{matomaki} on twisted Bombieri--Vinogradov estimates. We remark that Definition~\ref{def:mainsieve} may be a misnomer in that it does not include any requirement that $f_\textup{M}$ should be a lower or upper bound sieve, making more general functions admissible than classical sieves. These main-sieves will come up for us in the construction of Chen's minorant in Proposition~\ref{le_chensieve}. As is well known, the linear sieve can be written as a sum of well-factorable weights (see Lemma~\ref{lem:wellfac}). Thus, by choosing $\lambda_1$ to be trivial, it is a main-sieve as long as the sifting range does not include primes less than $P_1$. When we speak of the linear sieve, we always mean the version with well-factorable weights.

\subsection{Approximations}
As usual in an application of the circle method, the expected number of representations for the counting problem in Theorem~\ref{MT1} involves a singular series. With our choice of normalisation, it is given by
\begin{align}\label{eq:singserdef}
    \mathfrak{S}(m)\coloneqq \1_{m\equiv 4\pmod 6}\frac{27}{2}\prod_{p\geq 5}\left(1-\frac{6p^2-4p+1}{(p-1)^4}\right)\prod_{\substack{p\geq 5\\ p\mid m(m+4)}}\left(1+\frac{1}{p-4}\right)\prod_{\substack{p\geq 5\\ p\mid m+2}}\left(1+\frac{2}{p-4}\right).
\end{align}

We next define an approximation relation $\approx_\epsilon$, which is a rigorous version of the $\approx$ relation used in~\eqref{eq_outline2} to~\eqref{eq_outline4}. Showing that certain functions approximate each other in this sense is one of our central tasks.

\begin{definition}\label{def_approx}
    Let $\epsilon>0$ and $f, g$ be finitely supported arithmetic functions. Let $\Omega$ denote the upper bound pre-sieve from Definition~\ref{def:presieve}. We write $f\approx_{\epsilon} g$ if for every $|h(n)|\leq \Omega(n)\Omega(n+2)$ it holds that
    \begin{align*}
		\bigl|\sum_{\substack{n_1+n_2=m\\ N/2< n_1,n_2\leq N}}(f-g)(n_1) h(n_2)\bigr| \leq  \epsilon  m \bigl(\mathfrak{S}(m)+1\bigr)
	\end{align*}
    for all $m\in [5N/4,7N/4]$ with at most $N^{1-(\delta_1/10)^{4}}$ exceptions. 
\end{definition}

\subsection{Structure}

In Section~\ref{sec:Chen Sieve}, we construct a version of Chen's lower bound sieve that includes a pre-sieving, in preparation for step~\eqref{eq_outline1}. The key result for this is Proposition~\ref{le_chensieve}. 

In Section~\ref{sec:fouapprox}, we gather some basic facts about Fourier approximation and how it relates to our approximation notion $\approx_\epsilon$. These basic facts are then applied in Section~\ref{sec:MSremoval} to show how to replace the main-sieves in our additive problem with constants with respect to the $\approx_\epsilon$ notation, thus making step~\eqref{eq_outline2} rigorous. This is achieved in Proposition~\ref{prop_MSremoval}. 

In Section~\ref{sec:SiftedSums}, we consider our additive problem in the case where all constituents are pre-sieves or $|\Lambda_{R,r}|$. In particular, we show the more general version of~\eqref{eq:HRintro} in Proposition~\ref{prop:LambdaQ} and obtain the asymptotics~\eqref{eq:outlinepresievasymp} in Proposition~\ref{prop_presieveasmyp}, from which~\eqref{eq:introfundamental} also follows. 

In Section~\ref{sec:PrimestoRough}, we state and prove analogues of Gallagher's prime number theorem for $r_P$ and $\Lambda_{E_3^{*}}$. This is the place where Theorem~\ref{thm:gallagher_application}, which is a more general version of Theorem~\ref{thm:F-P-C_intro}, is proved. From this we deduce that we can replace primes (respectively, $E_3^*$ numbers) with the Cram\'er model in our additive problem; see Propositions~\ref{prop_reducetorough} and~\ref{prop_reducetoroughexc}. They are precise forms of~\eqref{eq:LambdatorP} in the unexceptional and exceptional cases, respectively. 

Finally, in Section~\ref{sec:final} we combine the results of the previous sections to prove Theorem~\ref{MT1} in the way indicated in the steps leading to~\eqref{eq:outlinepresievasymp}. The final proof can be described by the following dependency diagram, where we combined pairs that correspond to unexceptional or exceptional variants of the same statement into one node:

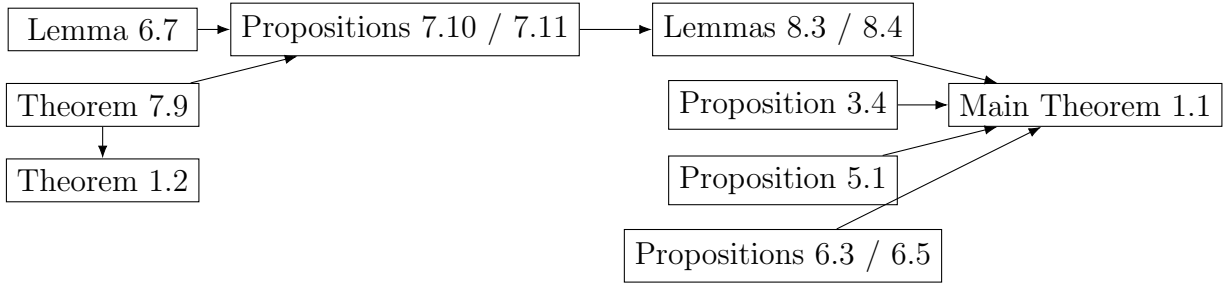
\vspace{1cm}
\begin{tikzpicture}[
  >=Latex,
  every node/.style={rectangle,draw,align=center,minimum width=2.5cm},
  column sep=3cm,
  row sep=1cm
]
  \node (L68) at (0,0)    {Lemma \ref{le:henriot}};
  \node (T79) at (0,-1)   {Theorem \ref{thm:gallagher_application}};
  \node (T12) at (0,-2)   {Theorem \ref{thm:F-P-C_intro}};

  \node (P71011) at (4,0)   {Propositions \ref{prop_reducetorough} / \ref{prop_reducetoroughexc}};

  \node (L823) at  (9,0)   {Lemmas \ref{lem_finalunex} / \ref{lem_finalex}};
  \node (P34) at  (9,-1)  {Proposition \ref{le_chensieve}};
  \node (P51) at  (9,-2)  {Proposition \ref{prop_MSremoval}};
  \node (P646) at  (9,-3)  {Propositions  \ref{prop_presieveasmyp} / \ref{prop_fundlem_exceptional}};

  \node (MT1) at (13,-1)  {Main Theorem \ref{MT1}};

  \draw[->] (T79)  -- (T12);
  \draw[->] (L68)  -- (P71011);
  \draw[->] (T79)  -- (P71011);
  \draw[->] (P71011) -- (L823);
  \draw[->] (P34)  -- (MT1);
  \draw[->] (L823)  -- (MT1);
  \draw[->] (P51)  -- (MT1);
  \draw[->] (P646)  -- (MT1);
\end{tikzpicture}

 \section{A variant of Chen's sieve} \label{sec:Chen Sieve}
 In this section, we construct a sieve minorant for the Chen primes. Our minorant is closely related to (a modern interpretation of) Chen's original construction. More precisely, for the most part we follow the construction of~\cite[Appendix A]{matomaki-shao} to ensure that the required minor arc bound can be shown. The main technical difference is that we include a pre-sieving process. 

 \subsection{Setup of the sieves}
The sieving process for Chen primes involves classical upper and lower bound sieves for rough numbers as well as numbers with precisely three prime factors in certain ranges. The latter come into play to facilitate Chen's sieve switching. 

\begin{definition}\label{def:lambdaE3}
Let 
\begin{align} \label{eq:B1def}
 B_1&=\{p_1p_2p_3: N^{1/10}\leq p_1<N^{1/3-\delta_1}\leq p_2\leq (N/p_1)^{1/2},\,p_3\geq N^{1/10}\},\\
\label{eq:B2def}
B_2&=\{p_1p_2p_3:N^{1/3-\delta_1}\leq p_1\leq p_2\leq (N/p_1)^{1/2},\,p_3\geq N^{1/10}\},
\end{align}
and define the related densities
\begin{align*}
    c_{B_1}(n)&=\int_{\substack{1/10\leq t_1\leq 1/3-\delta_1\leq t_2\leq (1-t_1)/2\\ \frac{\log n}{\log N}-t_1-t_2\geq 1/10}}\frac{\d t_1 \d t_2}{t_1t_2\log (n/N^{t_1+t_2})},\\
     c_{B_2}(n)&=\int_{\substack{1/3-\delta_1 \leq t_1\leq t_2\leq (1-t_1)/2\\ \frac{\log n}{\log N}-t_1-t_2\geq 1/10}}\frac{\d t_1 \d t_2}{t_1t_2\log (n/N^{t_1+t_2})},\\
     c_{E_3^*}(n)&=\frac{c_{B_1}(n)}{2}+c_{B_2}(n).
\end{align*}
Then we write
\begin{align}\label{eq:L3def}
\Lambda_{E_3^{*}}(n)=\frac{\1_{n\in B_1}/2+\1_{n\in B_2}}{c_{E_3^*}(n)}.
\end{align}
We also write $c_{E_3^*}=c_{E_3^*}(N) \log N$, noting that this is indeed a constant.
\end{definition}

The choices of $B_1$ and $B_2$ are identical to~\cite[equations (6.2a), (6.2b)]{matomaki-shao}, and the normalisation function appears for the choice $n=N$ in~\cite[equation (6.3)]{matomaki-shao}. It will be useful on several occasions to note that $c_{E_3^*}(t)$ is a smooth function. Furthermore, since for $N^{3/4}\leq t \leq 2N$, the integration range of $c_{B_i}(t)$ becomes independent of $t$, one obtains the estimates
\begin{align}
\label{eq:c_E}
   \begin{aligned} c_{E_3^*}(t)&\asymp \frac{1}{\log N},\\
    \frac{\d}{\d t}c_{E_3^*}(t)&\ll \frac{1}{t (\log N)^2}.
    \end{aligned}
\end{align}

One can show that a version of the Siegel--Walfisz theorem holds: if $\chi$ is any Dirichlet character of modulus $\leq (\log N)^{A}$, we have
\begin{align*}
 \sum_{n\leq N}\Lambda_{E_3^{*}}(n)\chi(n)=\1_{\chi \text{ principal}}\,N+O_A\left(\frac{N}{(\log N)^{A}}\right). 
\end{align*}
In fact, in Lemma~\ref{le:gallagherE3} we show a version of Gallagher's prime number theorem~\cite[Theorem 7]{Gal} for $\Lambda_{E_3^*}$.

 \subsection{Construction of the minorant}

 To include pre-sieves, we use a vector sieve inequality. This standard idea of constructing a lower bound sieve by composition goes back at least to Selberg~\cite{selberg}. 

	\begin{lemma}[Vector sieve inequality]\label{lem_lowerboundsievcomp}
		Let $A,B^+\geq 0$ and $A^{+},A^{-},B, B^- \in \R$ satisfy
		\begin{align*}
			B^-\leq B &\leq B^+ \\
			A^-\leq A &\leq A^+.
		\end{align*}
		Then
		\begin{align*}
			A^+B^- +(A^- -A^+)B^+  \leq AB.
		\end{align*}
	\end{lemma}
\begin{proof}
Since $A\geq 0$, we have
\begin{align*}
AB&\geq AB^-=A^+B^- +(A-A^+)B^-.
\end{align*}
As $A-A^+\leq 0$, we can bound this from below by
\begin{align*}
\geq A^+B^- +(A-A^+)B^+.
\end{align*}
Since $B^+\geq 0$, this is
\begin{align*}
\geq A^+B^- +(A^- -A^+)B^+, 
\end{align*}
as required.
\end{proof}

By Definition~\ref{def:mainsieve}, the main-sieves in the construction of Chen's minorant need to fulfil certain factorisation properties. To achieve this we now record a version of the linear sieve with well-factorable weights.

\begin{lemma}[Well-factorable linear sieve]\label{lem:wellfac}
    Let $\varepsilon>0$, $D$ be sufficiently large in terms of $\varepsilon$, and assume that $D^{\varepsilon^2}\leq P\leq z \leq D^{1/(2+2\varepsilon^{9})}$. Let $f(s)$ and $F(s)$ be the functions of the linear sieve defined by the systems~\cite[eq. (12.1), (12.2)]{cribro}. Let $J(\varepsilon)=e^{C\varepsilon^{-3}}$ for an absolute constant $C>0$. Then there exist well-factorable functions $\lambda_j^\pm(d)$ supported on $d\in [1,D]$ and $d\mid\prod_{P\leq p<z}p$ such that
    \begin{align*}
      \sum_{j\leq J(\varepsilon)}\sum_{d\mid n}  \lambda_j^-(d) \leq \1_{p\mid n \implies p\not \in [P,z)}& \leq \sum_{j\leq J(\varepsilon)}\sum_{d\mid n}  \lambda_j^+(d),\\
       \prod_{P\leq p < z}\left(1-\frac{1}{p}\right)^{-1} \sum_{j\leq J(\varepsilon)} \sum_{d}\frac{\lambda_j^-(d)}{\varphi(d)}&\geq f\left(\frac{\log D}{\log z}\right)+O(\varepsilon^5),\\
      \prod_{P\leq p < z}\left(1-\frac{1}{p}\right)^{-1}\sum_{j\leq J(\varepsilon)} \sum_{d} \frac{\lambda_j^+(d)}{\varphi(d)}&\leq F\left(\frac{\log D}{\log z}\right)+O(\varepsilon^5).
    \end{align*}
\end{lemma}
\begin{proof}
    The result follows the construction in~\cite[Section 12.7]{cribro}, with two differences: we skip the pre-sieving (which actually simplifies the proof), and we make statements about the sieve weights instead of the sifted set. We give the resulting details below.

    We only consider the case of a lower bound sieve, the upper bound being very similar. By the standard construction of the linear sieve (see~\cite[eq. (12.10)]{cribro}) we have for any $D_0$
    \begin{align*}
        \1_{p\mid n \implies p\not \in [P,z)}\geq \sum_{\substack{d\mid n\\ d\in \mathcal{D}^-(D_0)}}\mu(d)\1_{p\mid d \implies p \in [P,z)},
    \end{align*}
    where $\mathcal{D}^{-}(D_0)$ is given by~\cite[eq. (12.32)]{cribro} as
    \begin{align*}
    \mathcal{D}^{-}(D_0)=\{d=p_1\cdots p_r : p_1>\ldots > p_r, p_1\cdots p_m p_m^2 <D_0 \text{ for all } m \text{ even}\}.
    \end{align*}
    As shown in~\cite[Corollary 12.17]{cribro}, as long as $z<\sqrt{D_0}$, this set already enjoys good factorisation properties. To make use of those, we call $r$ the number of prime divisors of $d$ and split the relevant primes into short intervals. Let $\eta=\varepsilon^9$, set $D_0=D^{1/(1+\eta)}$, and let $D_1, \ldots, D_r$ run over numbers of the type $P^{(1+\eta)^j}$. Define for even $r$
    \begin{align*}
        \mathcal{D}_r=\{(D_1,\ldots,D_r)\colon D_r\leq \ldots \leq D_1 \leq \sqrt{D_0},
        D_1\ldots D_m D_m^{2}<D_0, m\leq r, m \text{ even}\} 
    \end{align*}
    and for odd $r$
    \begin{align*}
        \mathcal{D}_r=\{(D_1,\ldots,D_r)\colon D_r\leq \ldots \leq D_1 \leq D_0,
        D_1\ldots D_m D_m^{2}<D_0, m< r, m \text{ even}\}.
    \end{align*}
    Note that the condition $m\equiv r \pmod 2$ in~\cite[eq. (12.84), (12.85)]{cribro} is a typo. Similarly as in~\cite[eq. (12.86)]{cribro}, we get
    \begin{align}\label{eq:wellfacto1}
        \sum_{\substack{d\mid n\\ d\in \mathcal{D}^-(D_0)}}\mu(d)\1_{p\mid d \implies p \in [P,z)}\geq \sum_{r\leq \log D_0/ \log P} (-1)^r \sum_{(D_1,\ldots D_r)\in \mathcal{D}_r} \gamma(D_1,\ldots ,D_r)^{-1} \sum_{\substack{p_1 \cdots p_r\mid n\\ D_j\leq p_j \leq \min\{D_j^{1+\eta},z\}}}1.
    \end{align}
    Here, since we drop the condition $p_1>\ldots> p_r$, we let $ \gamma(D_1,\ldots ,D_r)=k_1!\cdots k_\ell!$ account for the multiplicity of the components if $k_i$ of the $D_j$ are equal. We follow the argument leading up to~\cite[eq. (12.91)]{cribro}. The argument in~\cite[Lemma 12.16]{cribro} shows that for each fixed $(D_1,\ldots D_r)\in \mathcal{D}_r$ the sum over $p_1\cdots p_r$ is well-factorable. Thus, again accounting for the multiplicity, we can write the right-hand side of~\eqref{eq:wellfacto1} as a sum of at most $e^{O(\varepsilon^{-3})}$ well-factorable functions
    \begin{align*}
        \sum_{d\mid n} \lambda^{-}_j(d)
    \end{align*}
    of level $D$ supported on divisors of $\prod_{P\leq p<z}p$. 
    
    To complete the proof of the lemma, it remains to lower bound
    \begin{align*}
       \sum_{j\leq J(\varepsilon)} \sum_{d} \frac{\lambda^{-}_j(d)}{\varphi(d)}.
    \end{align*}
    Observe that \begin{align*}
        \prod_{P\leq p < z}\left(1-\frac{1}{p-1}\right)= \bigl(1+O(P^{-1})\bigr)\prod_{P\leq p < z}\left(1-\frac{1}{p}\right).
    \end{align*}
    Thus, if we did not split into $\mathcal{D}_r$ and if we did not reduce the available level to $D_0=D^{1/(1+\eta)}$, the stated bound would follow directly from~\cite[eq. (12.5)]{cribro}.
    To account for the splitting, there are two terms to consider. The first comes from two prime divisors that are within a ratio of $D^\eta$ (or more precisely $P^{\eta}<D^\eta$, but this is inconsequential), the second from boundary regions not covered precisely. Using $P\geq D^{\varepsilon^2}$, both errors are shown to be $O(\varepsilon^5)$ after~\cite[eq. (12.88)]{cribro} and~\cite[eq. (12.89)]{cribro} respectively. Finally, the reduction of the level can be absorbed by the smoothness of $f$, introducing another admissible $O(\eta)$ error.
\end{proof}

We are now ready to construct a lower bound for the Chen primes and define
\begin{align}\label{eq:L2def}
    \Lambda_2(n)\coloneqq \1_{n\in\mathbb{P}_2}r_{N^{1/10}}(n),
\end{align}
so that $\Lambda(n)\Lambda_2(n+2)$ becomes a weighted indicator of Chen primes (with the additional restriction that $n+2$ is $N^{1/10}$-rough).

\begin{proposition}[Chen's lower bound sieve with pre-sieves]\label{le_chensieve}
Assume that $N$ is sufficiently large, and let $P_1=N^{\delta_1^4}$. There exist main-sieves 
\begin{align*}
    \omega_{\textnormal{M}}(n)&=\sum_{d\mid n}\lambda^\omega_{\textup{M}}(d),\\
    \Omega_{\textnormal{M}}(n)&=\sum_{d\mid n}\lambda^\Omega_{\textup{M}}(d),\\
    \Omega'_{\textnormal{M}}(n)&=\sum_{d\mid n}\lambda^{\Omega'}_{\textup{M}}(d),
\end{align*}
as in Definition~\ref{def:mainsieve}, and an error function $E(n)$ such that the following three statements hold. 
\begin{enumerate}[label=(\roman*)]
\item Let $\omega,\Omega$ be pre-sieves as in Definition~\ref{def:presieve} and denote
\begin{align*}
    g_1(n)&=\Lambda(n)\Omega(n+2)\omega_{\textnormal{M}}(n+2), \\
    g_2(n)&=(3/5+\delta_1)c_{E_3^*}\Omega(n)\Omega_{\textnormal{M}}(n)\Lambda_{E_3^*}(n+2), \\
    g_3(n)&=\Lambda(n)\bigl(\omega -\Omega\bigr)(n+2)\Omega'_{\textnormal{M}}(n+2).
\end{align*}
We have for $N/2< n\leq N$ the inequality
\begin{align}\label{eq:chenlower_minorisation}
    \Lambda(n)\Lambda_2(n+2)\geq g_1(n)-g_2(n)+g_3(n)+\Lambda(n)E(n+2).
\end{align}
    \item The sieve weights fulfil
    \begin{align}\label{eq:chentrivial_numerics}
     \prod_{P_1\leq p < N^{1/10}}\left(1-\frac{1}{p}\right)^{-1} \left(\sum_{d}\frac{\lambda^{\Omega'}_{\textup{M}}(d)}{\varphi(d)}\right)\ll 1,
    \end{align}
    and, if $\delta_1$ is sufficiently small, there exists an absolute constant $c_0>0$ such that
\begin{align}\label{eq:chenlower_numerics}
   \prod_{P_1\leq p < N^{1/10}}\left(1-\frac{1}{p}\right)^{-1} \left(\sum_{d}\frac{\lambda^\omega_{\textup{M}}(d)}{\varphi(d)}-(3/5+\delta_1) c_{E_3^*}\sum_{d}\frac{\lambda^\Omega_{\textup{M}}(d)}{\varphi(d)}\right)\geq c_0.
\end{align}

\item We have the approximation
\begin{align} \label{eq:chenlower_error}
    \Lambda(n)E(n+2)\approx_{N^{-1/100}} 0
\end{align}
in the sense of Definition~\ref{def_approx}.

\end{enumerate}

\end{proposition}

\begin{proof}
    The proof is an incorporation of pre-sieves into the construction of~\cite[Appendix A]{matomaki-shao}. In the case of lower bound sieves, Lemma~\ref{lem_lowerboundsievcomp} is used to facilitate this. 

    Recall the definition of $\Lambda_2$ in~\eqref{eq:L2def}.
    Similarly to~\cite[Appendix A.2]{matomaki-shao} with $B_1$ and $B_2$ given as in~\eqref{eq:B1def} and~\eqref{eq:B2def}, we have for $N/2 < n\leq N$ that 
    \begin{align}\label{eq:chenlower1}
        \Lambda_2(n)\geq r_{N^{1/10}}(n)\left( 1-\frac{1}{2}\sum_{\substack{p\mid n \\ N^{1/10}\leq p \leq N^{1/3-\delta_1}}}1-\frac{1}{2}\1_{n\in B_1}-\1_{n\in B_2}-E'(n)\right),
    \end{align}
    where 
    \begin{align*}
        0\leq E'(n)\ll \sum_{\substack{p^2\mid n \\ p \geq N^{1/10}}}1
    \end{align*}
    accounts for square divisors, and summing this bound gives $\sum_{n\sim N}E'(n)\ll N\sum_{p\geq N^{1/10}}p^{-2}\ll N^{9/10}$, so $E(n)\coloneqq r_{N^{1/10}}(n)E'(n)$ fulfils~\eqref{eq:chenlower_error}. Indeed, if $n\in (N/2,N]$ is square-free and $N^{1/10}$-rough, then $n$ has either two prime divisors from $[N^{1/10},N^{1/3-\delta_1}]$ or $n=p_1p_2p_3$ with $p_2,p_3>N^{1/3-\delta_1}$. But then either $n\in B_2$ or $n\in B_1$ and $p_1\leq N^{1/3-\delta_1}$. 
    Thus, by~\eqref{eq:chenlower1} and the definition of $\Lambda_{E_3^*}$ in~\eqref{eq:L3def}, we get 
    \begin{align}\label{eq:chenlower1b}
        \Lambda_2(n)\geq r_{N^{1/10}}(n)\left(1-\frac{1}{2}\sum_{\substack{p\mid n \\ N^{1/10}\leq p \leq N^{1/3-\delta_1}}}1\right)-c_{E_3^*} (n)\Lambda_{E_3^*}(n)-E(n),
    \end{align}
and then clearly~\eqref{eq:chenlower_error} holds for this error function $E$.  
    We rewrite the first component on the right-hand side of~\eqref{eq:chenlower1b} as
    \begin{align*}
         r_{N^{1/10}}(n)\left(1-\frac{1}{2}\sum_{\substack{p\mid n \\ N^{1/10}\leq p \leq N^{1/3-\delta_1}}}1\right)&=r_{P_1}(n)\left(r_{[P_1,N^{1/10})}(n)-\frac{1}{2}\sum_{\substack{p\mid n \\ N^{1/10}\leq p \leq N^{1/3-\delta_1}}}r_{[P_1,N^{1/10})}(n) \right),
    \end{align*}
    and next minorise this with the help of Lemma~\ref{lem_lowerboundsievcomp}. Let
    \begin{align*}
        A&=r_{P_1}(n),\\
        B&=r_{[P_1,N^{1/10})}(n)-\frac{1}{2}\sum_{\substack{p\mid n \\ N^{1/10}\leq p \leq N^{1/3-\delta_1}}}r_{[P_1,N^{1/10})}(n)
    \end{align*}
    with $\omega, \Omega$ as in Definition~\ref{def:presieve}. We choose
    \begin{align*}
        A^-&=\omega(n),\\
        A^+&=\Omega(n),
    \end{align*}
    so that 
    \begin{align*}
        A^-\leq A \leq A^+.
    \end{align*}
    We lower bound $B$ as in~\cite[Appendix A.3]{matomaki-shao}, using the well-factorable sieve from Lemma~\ref{lem:wellfac} with $\varepsilon=\delta_1^3$. We apply a lower bound linear sieve of range $[P_1,N^{1/10})$ and level $N^{1/2-2\delta_1}$ to the first component. Separately for each $p$, we use an upper bound linear sieve with range $[P_1,N^{1/10})$ and level $N^{1/2-2\delta_1}/p$ on $r_{[P_1,N^{1/10})}$. The sum of both sieves (including summation over $p$) is admissible for Definition~\ref{def:mainsieve}, and so we get the existence of some main-sieve $\omega_{\textup{M}}$ such that for
    \begin{align*}
        B^-=\omega_{\textup{M}}(n),
    \end{align*}
 we have
    \begin{align*}
        B^-\leq B.
    \end{align*}
    Clearly
    \begin{align*}
        B\leq r_{[P_1,N^{1/10})}(n)\leq \Omega'_{\textnormal{M}}(n)\coloneqq B^+,
    \end{align*}
    where $\Omega'_{\textnormal{M}}$ is the (well-factorable) upper bound linear sieve with range $[P_1,N^{1/10})$ and level $N^{1/2-2\delta_1}$. By the basic theory of the beta-sieve~\cite[Theorem 11.12]{cribro} and Mertens' theorem, the left-hand side of~\eqref{eq:chentrivial_numerics} is bounded in absolute terms. 
    Since $A=r_{P_1}(n)\geq0$ and $B^+=\Omega'_{\textnormal{M}}(n)\geq r_{[P_1,N^{1/10})}(n)\geq0$, the hypotheses of Lemma~\ref{lem_lowerboundsievcomp} hold, and it gives
    \begin{align}\label{eq:3.9firstlower}
     r_{N^{1/10}}(n)\left(1-\frac{1}{2}\sum_{\substack{p\mid n \\ N^{1/10}\leq p \leq N^{1/3-\delta_1}}}1\right)\geq \Omega(n)\omega_{\textnormal{M}}(n)+(\omega-\Omega)(n)\Omega'_{\textnormal{M}}(n).
    \end{align}
    Plugging this into~\eqref{eq:chenlower1b}, we get for $N/2< n \leq N-2$ that 
    \begin{align}\begin{split}\label{eq:chensieveLupper}
    &\Lambda(n)\Lambda_2(n+2)\\
    \geq& g_1(n)+g_3(n)-\Lambda(n)c_{E_3^*} (n+2)\Lambda_{E_3^*}(n+2)+\Lambda(n)E(n+2).
    \end{split}
    \end{align}
    We follow~\cite[Appendix A.4]{matomaki-shao} to majorise
    \begin{align*}
        \Lambda(n)\leq (\log N)\prod_{p<N^{1/6}}\left(1-\frac{1}{p}\right)\cdot r_{P_1}(n) r_{[P_1,N^{1/6})}(n)\leq (\log N)\prod_{p<N^{1/6}}\left(1-\frac{1}{p}\right)\cdot  \Omega (n)\Omega_{\textnormal{M}}(n),
    \end{align*}
    where $\Omega_{\textnormal{M}}$ is the (well-factorable) upper bound linear sieve with range $[P_1,N^{1/6})$ and level $N^{1/2-2 \delta_1}$.
    Observe further that by~\eqref{eq:c_E}, we have $c_{E_3^*}(n+2)=c_{E_3^*}(N)(1+O((\log N)^{-1}))$ and recall that we write $c_{E_3^*}(N) \log N=c_{E_3^*} $. By Mertens' theorem,
    \begin{align*}
        \prod_{p<N^{1/6}}\Bigl(1-\frac1p\Bigr)=\prod_{N^{1/10}\leq p<N^{1/6}}\Bigl(1-\frac1p\Bigr)\prod_{p<N^{1/10}}\Bigl(1-\frac1p\Bigr)=\Bigl(\frac35+O(\delta_1)\Bigr)\prod_{p<N^{1/10}}\Bigl(1-\frac1p\Bigr),
    \end{align*}
    the middle factor tending to $(1/10)/(1/6)=3/5$. The remaining product $\prod_{p<N^{1/10}}(1-1/p)$ cancels against the normalisations built into $\Omega(n)$ and $\Omega_{\textnormal{M}}(n)$, leaving exactly the factor $3/5+O(\delta_1)$ recorded in $g_2$. We apply Mertens' theorem and this approximation so that for any $\delta_1>0$ and $N$ sufficiently large in terms of $\delta_1$, we have
    \begin{align*}
        \Lambda(n)c_{E_3^*} (n+2)\Lambda_{E_3^*}(n+2)\leq&(3/5+\delta_1)c_{E_3^*}\Omega(n)\Omega_{\textnormal{M}}(n)\Lambda_{E_3^*}(n+2) \\
        =&g_2(n).
    \end{align*}
    Together with~\eqref{eq:chensieveLupper} and since the contribution of $n\in \{N-1,N\}$ can be absorbed in the error term, this shows~\eqref{eq:chenlower_minorisation}.

    The proof of~\eqref{eq:chenlower_numerics} follows from the calculations in~\cite[Appendices A.3, A.4, A.5]{matomaki-shao}. Indeed by Lemma~\ref{lem:wellfac}, we have
    \begin{align*}
        \prod_{P_1\leq p < N^{1/10}}\left(1-\frac{1}{p}\right)^{-1}\sum_{d}\frac{\lambda^\omega_{\textup{M}}(d)}{\varphi(d)}&\geq f(5)-\frac{1}{2}\sum_{N^{1/10}\leq p <N^{1/3-\delta_1}}\frac{F(5-10 \log p/\log N)}{\varphi(p)} -O(\delta_1)\\
        &\geq f(5)-\frac{1}{2}\int_{1/10}^{1/3} F(5-10 t) \frac{\d t}{t}-O(\delta_1)
    \end{align*}
    and
    \begin{align*}
             \frac{3}{5}c_{E_3^*}\prod_{P_1\leq p < N^{1/10}}\left(1-\frac{1}{p}\right)^{-1}\sum_{d}\frac{\lambda^\Omega_{\textup{M}}(d)}{\varphi(d)}\leq \frac{1}{2}\cdot \frac{3}{5}F(3) \int_{\substack{1/10\leq t_1\leq 1/3\leq t_2\leq (1-t_1)/2\\ 1-t_1-t_2\geq 1/10}}\frac{\d t_1 \d t_2}{t_1t_2(1-t_1-t_2)}+O(\delta_1).
        \end{align*}
        That the difference of these is strictly positive for sufficiently small $\delta_1$ is stated explicitly in \cite[Appendix A.5]{matomaki-shao}.
\end{proof}

We require one more type of main-sieve that majorises $r_{[P_1,P_0)}$.

\begin{lemma}\label{lem:P_1P_0majo}
    Recall the parameter choices in Definition~\ref{def:para}: $P_1=N^{\delta_1^4}$, $P_0=N^{\delta_1}$, $D_{\mathrm{M},1}=N^{1/3-\delta_1}$. There exist main-sieves
    \begin{align*}
        \omega_{[P_1,P_0)}(n)&=\sum_{d\mid n}\lambda^\omega_{\textup{M}}(d),\\
        \Omega_{[P_1,P_0)}(n)&=\sum_{d\mid n}\lambda^\Omega_{\textup{M}}(d)
    \end{align*}
    such that
    \begin{align}\label{eq:lemaP_1P_01}
         \omega_{[P_1,P_0)}(n)\leq r_{[P_1,P_0)}(n) \leq  \Omega_{[P_1,P_0)}(n)
    \end{align}
    and
    \begin{align}
    \label{eq:P1P0majo1} 
    \begin{aligned}
    \prod_{P_1\leq p < P_0}\Bigl(1-\frac{1}{p}\Bigr)^{-1}\sum_{d}\frac{\lambda^\Omega_{\textup{M}}(d)}{\varphi(d)}&=1+O\bigl(e^{-\frac{\log D_{\mathrm{M},1}}{40 \log P_0}}+\delta_1^{15}\bigr),\\
        \prod_{P_1\leq p < P_0}\Bigl(1-\frac{1}{p}\Bigr)^{-1}\sum_{d}\frac{\lambda^\omega_{\textup{M}}(d)}{\varphi(d)}&=1+O\bigl(e^{-\frac{\log D_{\mathrm{M},1}}{40 \log P_0}}+\delta_1^{15}\bigr).
       \end{aligned}
    \end{align}
\end{lemma}
\begin{proof}
    We let $\Omega_{[P_1,P_0)}(n)$ be the sum of the $O_{\delta_1}(1)$ well-factorable upper-bound linear sieve weights of Lemma~\ref{lem:wellfac} with sifting range $[P_1,P_0)$ and level $D_{\mathrm{M},1}=N^{1/3-\delta_1}$, normalised by $\prod_{P_1\leq p < P_0}\bigl(1-1/p\bigr)^{-1}$; see Lemma~\ref{lem:wellfac}. By construction, this fulfils~\eqref{eq:lemaP_1P_01}. Since the sifting range includes no primes less than $P_1$ and the level is $D_{\mathrm{M},1}$, this fits Definition~\ref{def:mainsieve}. The condition~\eqref{eq:P1P0majo1} follows from the basic theory of the linear sieve; see for example~\cite[Theorem 11.12, eq. (11.134)]{cribro}. Note that strictly speaking only one-sided inequalities are proved there; the asymptotics follow from the fact that 
    \begin{align*}
       \sum_{d} \frac{\lambda^\omega_{\textup{M}}(d)}{\varphi(d)}\leq \sum_d \frac{\lambda^\Omega_{\textup{M}}(d)}{\varphi(d)}.
    \end{align*}
\end{proof}

\section{Fourier approximants}\label{sec:fouapprox}

The fact that additive convolution translates to multiplication in Fourier space gives an immediate way to show an approximation $f\approx_\epsilon g$ (in the language of Definition~\ref{def_approx}) by showing a power-saving estimate for 
\begin{align*}
    \sup_\alpha \left|\sum_{N/2< n\leq N}(f-g)(n)e(\alpha n)\right|.
\end{align*}
In this section, we make this precise, gather some basic results about Fourier series, and define the function $b_R(n)$ that replaces the classical major/minor arc decomposition. 

We use the following notation to abbreviate the statements involving the Fourier series. 
\begin{definition}[Fourier norm]\label{def:Fouriernorm}
    For a function $f:\mathbb{N}\to \mathbb{C}$, define its \emph{Fourier norm} as
    \begin{align*}
        \|f\|_{\infty}^{\wedge}\coloneqq \sup_{\alpha \in \mathbb{R}}\left|\sum_{N/2< n\leq N}f(n)e(\alpha n)\right|.    
    \end{align*}
\end{definition}
Note that this differs from the similar notation in~\cite[Section 8.1]{green-sarkozy}. We first show that closeness in Fourier norm implies the approximation relation of Definition~\ref{def_approx}. Recall for this that $\delta_1$ is a global parameter that also appears in that definition, as it will fix the quantitative thresholds in the lemma below.
\begin{lemma}[Fourier-closeness implies $\approx$]\label{lem_fourierimpliesapprox}
    Let $f_1,f_2:\N\to \mathbb{C}$. If
	\begin{align*}
		\|f_1-f_2\|_{\infty}^{\wedge} \ll N^{1-2(\delta_1/10)^{4}},
    \end{align*}
	then for $\epsilon\gg N^{-(\delta_1/10)^{4}}$ we have
    \begin{align*}
        f_1\approx_{\epsilon} f_2.
	\end{align*}
	\end{lemma}
\begin{proof} 
In view of Definitions~\ref{def_approx} and~\ref{def:Fouriernorm}, we can assume that all involved functions are supported on $n\in (N/2,N]$ only.

Let $\eta=N^{-2(\delta_1/10)^4}$. Write $f=f_1-f_2$ and let $g:(N/2,N]\to \mathbb{C}$ be any function. By the assumption of the lemma we have
	\begin{align}\label{eq1b}
	   \sup_{\alpha \in \mathbb{R}}\left|\sum_{N/2<n\leq  N}f(n)e(\alpha n)\right|\leq \eta N.    
	\end{align}
	We will show that for all but at most $\eta^{2/3} N$ integers $m\in [5N/4,7N/4]$, we have
	\begin{align*}
	   |f*g(m)|\leq \eta^{2/3} N^{1/2}\|g\|_2.  
	\end{align*}
    This implies the lemma: the exceptional count $\eta^{2/3}N=N^{1-\frac43(\delta_1/10)^4}\le N^{1-(\delta_1/10)^4}$ is within the allowance of Definition~\ref{def_approx}, as the functions required for Definition~\ref{def_approx} fulfil $\|g\|_2\leq N^{1/2+o(1)}$.
	
    Let $\mathcal{S}\subset [5N/4,7N/4]$ be the set of $m$ for which
	\begin{align*}
    	|f*g(m)|>  \eta^{2/3} N^{1/2}\|g\|_2.
	\end{align*}
	Pick unimodular complex numbers $c_m$ such that
	\begin{align*}
    	c_m(f*g(m))> \eta^{2/3} N^{1/2}\|g\|_2 
	\end{align*}
	for $m\in \mathcal{S}$. 
		
	Then, summing over $m\in [5N/4,7N/4]$ and applying the orthogonality of characters, we obtain
	\begin{align}\begin{split}\label{eq2}
    	\eta^{2/3} N^{1/2}\|g\|_2|\mathcal{S}|&<\sum_{n_1,n_2\in (N/2,N]}f(n_1)g(n_2)c_{n_1+n_2}\1_{\mathcal{S}}(n_1+n_2)\\
		&= \int_{0}^{1} F(\alpha)G(\alpha)S(-\alpha) \, \d \alpha, 
        \end{split}
	\end{align}
	where
	\begin{align*}
		F(\alpha)&\coloneqq \sum_{N/2<n\leq  N}f(n)e(n\alpha),\\
		G(\alpha)&\coloneqq \sum_{N/2<n\leq N}g(n)e(n\alpha),\\
		S(\alpha)&\coloneqq \sum_{n}c_n\1_{\mathcal{S}}(n)e(n\alpha).
	\end{align*}
		
	Now, by the assumption~\eqref{eq1b}, we can apply Cauchy--Schwarz and Parseval's identity to~\eqref{eq2} to conclude that
	\begin{align*}
		\eta^{2/3} N^{1/2}\|g\|_2|\mathcal{S}|&\leq \eta N\left(\int_{0}^1 |G(\alpha)|^2\, \d\alpha\right)^{1/2} \left(\int_{0}^1 |S(\alpha)|^2\, \d\alpha\right)^{1/2}\\
		&\leq \eta N \|g\|_2|\mathcal{S}|^{1/2}.
	\end{align*}
	This implies
	\begin{align*}
		|\mathcal{S}|\leq  \eta^{2/3} N,
	\end{align*}
	as desired.
\end{proof}
The next lemma allows us to estimate trivially the effect of twisting by a sieve or character when calculating the Fourier norm.
\begin{lemma}[Fourier properties of sieve and character twists]\label{lem_twistsievechar}
    Let $f\colon \mathbb{N}\to \mathbb{C}$ be of the form $f(n)=1\star \lambda (n)=\sum_{d\mid n}\lambda(d)$ for some arithmetic function $\lambda$. Let $\chi$ be a primitive Dirichlet character of conductor $r\in \mathbb{N}$. Then for any arithmetic function $g$, we have
	\begin{align*}
		\|fg\|_{\infty}^{\wedge} &\leq \|\lambda \|_1 \|g\|_{\infty}^{\wedge} \\
		\|\chi g\|_{\infty}^{\wedge} &\leq \sqrt{r} \|g\|_{\infty}^{\wedge}. 
	\end{align*}
\end{lemma}

\begin{proof}
By the definition of the Fourier norm, we can assume that $g$ is supported on $(N/2,N]$ and in particular any sum over $g(n)$ is convergent. 
We open the definition of $f$ to get
    \begin{align*}
        \sup_{\alpha\in \mathbb{R}} \left| \sum_{n} f(n)g(n)e(\alpha n)\right|&=\sup_{\alpha}\left|\sum_{d}\lambda(d) \sum_{n\equiv 0 \pmod d} g(n) e(\alpha n)\right|\\
        &=\sup_{\alpha\in \mathbb{R}} \left| \sum_{d}\frac{\lambda(d)}{d}\sum_{b(d)} \sum_{n} g(n) e((\alpha+b/d)n)\right|\\
        &\leq \sum_{d}|\lambda(d)|\sup_{\alpha\in \mathbb{R}} \left| \sum_{n} g(n) e(\alpha n)\right|\\
        &= \|\lambda \|_1 \|g\|_{\infty}^{\wedge}.
    \end{align*}
    For the second statement, we use the well known estimate for Gauss sums that gives for primitive characters 
    \begin{align*}
        \max_{b(r)}\bigl|\sum_{a(r)}\chi(a) e_r(ab)\bigr|\leq r^{1/2}.
    \end{align*}
    Indeed, this follows for $(b,r)=1$ from~\cite[Lemma 5.1]{mv} and for $(b,r)\neq 1$, the sum vanishes. Thus, 
     \begin{align*}
        \sup_{\alpha\in \mathbb{R}} \left| \sum_{n} \chi(n)g(n)e(\alpha n)\right|&=  \sup_{\alpha\in \mathbb{R}} \left| \frac{1}{r}\sum_{b(r)} \sum_{a(r)} \chi(a) e_r(-ab) \sum_{n} g(n)e((\alpha+b/r) n) \right|\\
        &\leq \frac{1}{r} \sum_{b(r)} \left|\sum_{a(r)}\chi(a)e_r(-ab) \right| \sup_{\beta \in \R} \left|  \sum_{n} g(n)e(\beta n) \right|\\
        &\leq \sqrt{r} \|g\|_{\infty}^{\wedge}.
    \end{align*}
    This completes the proof.
\end{proof}
	
We now define our notion of major arcs of order $R$ and the related function $b_R$. 
	
\begin{definition}[Major and minor arcs]\label{def_major}
    Define the \emph{major arcs} of order $R\geq 1$ by
	\begin{align*}
	   \mathfrak{M}(R)\coloneqq [0,1)\cap \bigcup_{1\leq q\leq R}\bigcup_{\substack{b\in  \mathbb{Z}\\(b,q)=1}}\left[\frac{b}{q}-\frac{R}{N},\frac{b}{q}+\frac{R}{N}\right] 
	\end{align*}
	and the \emph{minor arcs} by
	\begin{align*}
	   \mathfrak{m}(R)\coloneqq [0,1)\setminus \mathfrak{M}(R).
	\end{align*}
    We also write 
    \begin{align*}
         \|f\|_{\mathfrak{M}(R)}^{\wedge}&\coloneqq\sup_{\alpha \in \mathfrak{M}(R)}\left|\sum_{N/2<n\leq N}f(n)e(\alpha n)\right|,\\
         \|f\|_{\mathfrak{m}(R)}^{\wedge}&\coloneqq\sup_{\alpha \in \mathfrak{m}(R)}\left|\sum_{N/2<n\leq N}f(n)e(\alpha n)\right|.
    \end{align*}
\end{definition}

\begin{definition}[Physical-space major-arc kernel $b_R$]\label{def:bdef}
	Let $G$ be a fixed smooth and non-negative function that is supported in $[-2,2]$ and is equal to $1$ in $[0,1]$, and let $R\geq 2$. Define
	\begin{align*}
	b_R(n)\coloneqq\sum_{r\geq 1}\,\,\,\asum_{b\pmod r}e_r(bn)\1_{|n|\leq N/R^4}\frac{R^4}{2N}G\bigl(\frac{\log r}{\log R}\bigr).
\end{align*}
\end{definition}

The important property for us is that $|\widehat{b}_R(\alpha)-1|$ is small for $\alpha \in \mathfrak{M}(R)$.

\begin{lemma}[$b_R$ extracts the major arc contribution]\label{lem:bRextract}
	Let $f$ be any arithmetic function supported on $(N/2,N]$, and assume that $2\leq R\leq N^{1/11}$. We have
	\begin{align*}
	   \|f-f*b_R\|_{\infty}^{\wedge} \ll \|f\|_{\mathfrak{m}(R)}^{\wedge}+\frac{\|f\|_1}{R}+ \frac{\|f\|_\infty N \log R}{R^2}.
	\end{align*}
\end{lemma}
	
\begin{proof}
We first note that by the support condition of $f$ and Definition~\ref{def:bdef}, we have the trivial upper bound
\begin{align*}
    |f*b_R(n)|=&\bigl| \sum_{n_1+n_2=n}f(n_1) b_R(n_2) \bigr|\\
    \leq & \|f\|_\infty \1_{n\in (N/2-N/R^4,N+N/R^4]} \frac{R^4}{2N} \sum_{|n-n_1|\leq N/R^4} \sum_{r\leq R^2} |c_r(n-n_1)|\\
    \ll & \|f\|_\infty \1_{n\in (N/2-N/R^4,N+N/R^4]}R^2 \log R.
\end{align*}
Here we used $|c_r(n_2)|\leq (r,n_2)$ with $n_2=n-n_1$, and that the bound $\sum_{r\leq R^2}|c_r(n_2)|\ll R^2\log R$ holds on average over the range $|n_2|\leq N/R^4$ present in the display. Indeed, writing $(r,n_2)=\sum_{d\mid (r,n_2)}\varphi(d)$ and summing over $n_2$ first,
\begin{align*}
\sum_{|n_2|\leq N/R^4}\ \sum_{r\leq R^2}(r,n_2)
=\sum_{r\leq R^2}\sum_{d\mid r}\varphi(d)\!\!\sum_{\substack{|n_2|\leq N/R^4\\ d\mid n_2}}\!\!1
\ll \frac{N}{R^4}\sum_{r\leq R^2}\sum_{d\mid r}\frac{\varphi(d)}{d}
= \frac{N}{R^4}\sum_{d\leq R^2}\frac{\varphi(d)}{d}\Bigl\lfloor\frac{R^2}{d}\Bigr\rfloor
\ll \frac{N}{R^4}\,R^2\log R,
\end{align*}
using $\sum_{d\leq R^2}\varphi(d)/d^2\leq\sum_{d\leq R^2}1/d\ll\log R$. Multiplying by the prefactor $R^4/(2N)$ gives the pointwise bound $\|f\|_\infty R^2\log R$.
We use this to estimate
\begin{align*}
   &\Bigl|\sum_{n}(f*b_R)(n)e(\alpha n)-\sum_{N/2<n\leq N}(f*b_R)(n)e(\alpha n)\Bigr|\\
   \leq& \sum_{n\leq N/2}|(f*b_R)(n)|+\sum_{n>N}|(f*b_R)(n)|\\
  \ll&  \frac{\|f\|_\infty N \log R}{R^2}.
\end{align*}
Thus, up to an admissible error we can replace 
\begin{align*}
    \sum_{N/2<n\leq N} (f*b_R)(n)e(\alpha n)
\end{align*}
by 
\begin{align*}
    \sum_{n} (f*b_R)(n)e(\alpha n)=\widehat{f}(\alpha) \widehat{b}_R(\alpha).
\end{align*}
Here $\widehat{f}$ is the unrestricted Fourier transform, which agrees with the one of Definition~\ref{def:Fouriernorm} because $f$ is supported on $(N/2,N]$.
Consequently, the lemma follows if we can show that
    \begin{align}\label{eq_bresult}
	   \widehat{b}_R(\alpha)=\begin{cases}1+O(1/R) &\text{ if } \alpha \in \mathfrak{M}(R)\\
		O(1) &\text{ if } \alpha \in \mathfrak{m}(R).
	\end{cases}
	\end{align}
	
    Given $\alpha\in \mathbb{R}$, let $a\in \Z, q\in \N$ be such that $|a/q-\alpha|$ is minimal among all choices with $(a,q)=1, q\leq R^2$. Write $\beta=\alpha-a/q$ so that
	\begin{align}\label{eq_bdiss}
	   \widehat{b}_R(\alpha)=\frac{R^4}{2N} G\bigl(\frac{\log q}{\log R}\bigr)\sum_{|n|\leq N/R^4}e(-\beta n)+\sum_{r\geq 1}\sum_{\substack{1\leq b<r\\ (b,r)=1 \\ b/r\neq a/q}} G\bigl(\frac{\log r}{\log R}\bigr) \frac{R^4}{2N}\sum_{|n|\leq N/R^4} e\bigl( (\alpha-b/r)n\bigr).
	\end{align}
	As the points $b/r$ with $(b,r)=1$ and $r\leq R^2$ are $1/R^4$ well spaced, in the second sum on the right-hand side of~\eqref{eq_bdiss} we have $|\alpha-b/r|\gg 1/R^4$ and can estimate it by
	\begin{align*}
	   \left|\sum_{r\geq 1}\sum_{\substack{1\leq b<r\\ (b,r)=1 \\ b/r\neq a/q}}  G\bigl(\frac{\log r}{\log R}\bigr) \frac{R^4}{2N}\sum_{|n|\leq N/R^4} e\bigl( (\alpha-b/r)n\bigr)\right|
	   &\ll \sum_{r\geq 1} G\bigl(\frac{\log r}{\log R}\bigr)\frac{R^4}{2N} \bigl(R^4+\sum_{1\leq b\leq r}\frac{r}{b}\bigr)\\
	   &\ll \frac{R^8}{N}\sum_{r\leq R^2} G\bigl(\frac{\log r}{\log R}\bigr)\\
	   &\ll \frac{R^{10}}{N}\\
	   &\ll \frac{1}{R},
	\end{align*}
	where we used that $R^{11}\leq N$. Furthermore, the first term on the right-hand side of~\eqref{eq_bdiss} is always $\ll 1$ in absolute value. Hence $\widehat b_R(\alpha)\ll 1$ for $\alpha\in\mathfrak m(R)$, which is the $\alpha \in \mathfrak{m}(R)$ case of~\eqref{eq_bresult}.
    
    If $\alpha\in \mathfrak{M}(R)$, then $q\leq R$ and $|\beta|\leq R/N$ in~\eqref{eq_bdiss}. In this case we have $|\beta n|\leq 1/R^3$ and a Taylor expansion gives us
	\begin{align*}
		\frac{R^4}{2N} G\bigl(\frac{\log q}{\log R}\bigr)\sum_{|n|\leq N/R^4}e(-\beta n)&=\frac{R^4}{2N}\sum_{|n|\leq N/R^4}\bigl(1+O(1/R^3)\bigr)\\
		&= 1 + O(1/R).
	\end{align*}
	Hence $\widehat b_R(\alpha)=1+O(1/R)$ for $\alpha\in\mathfrak M(R)$, which is the $\alpha \in \mathfrak{M}(R)$ case of~\eqref{eq_bresult}.
\end{proof}

The next lemma reduces major arc behaviour to short interval twists with multiplicative characters.
\begin{lemma}[Reduction to multiplicative characters]\label{lem:Mtochar} Let $1\leq R\leq N^{1/3}$ and
    let $f$ be an arithmetic function supported on $R$-rough integers in $(N/2,N]$. Then we have 
    \begin{align*}
        \|f \|_{\mathfrak{M}(R)}^{\wedge}\ll R^{5/2} \max_{\substack{N/2-N/R^2\leq n\leq N+N/R^2}} \max_{\substack{q\leq R,\\ \chi\pmod {q}\\ \chi \textnormal{ primitive}}}\left|\sum_{|n'-n|\leq N/R^2} f(n') \chi(n') \right|+\|f\|_1/R+\|f\|_\infty \frac{N}{R^2}.
    \end{align*}
\end{lemma}

\begin{proof}
Since $f$ is supported in $(N/2,N]$, we can write $\widehat{f}$ for the associated Fourier series and have
\begin{align*}
    \|f \|_{\mathfrak{M}(R)}^{\wedge}=\sup_{\alpha \in \mathfrak{M}(R)} \Bigl| \sum_{n}f(n) e(\alpha n) \Bigr|=\sup_{\alpha \in \mathfrak{M}(R)} \Bigl| \widehat{f}(\alpha)\Bigr|.
\end{align*}
We set 
\begin{align*}
    b_{q,a}(n)\coloneqq\frac{R^2}{2N} \1_{|n|\leq N/R^2}e_q(-an)
\end{align*}
and observe that 
\begin{align*}
    \widehat{b_{q,a}}(\alpha)=1+O(R^{-1})
\end{align*}
for $|\alpha-a/q|\leq R/N$, by the same arguments as leading to the major arc case of~\eqref{eq_bresult}. Recalling Definition~\ref{def_major}, we then get
\begin{align*}
   \sup_{\alpha \in \mathfrak{M}(R)} \Bigl| \widehat{f}(\alpha)\Bigr|&=\max_{\substack{q\leq R\\ (a,q)=1}} \sup_{|\alpha-a/q|\leq R/N} |\widehat{f}(\alpha)|\\
    &\ll \max_{\substack{q\leq R\\ (a,q)=1}} \sup_{|\alpha-a/q|\leq R/N} |\widehat{f*b_{q,a}}(\alpha)|+\|f\|_1/R.
\end{align*}
By discarding some terms trivially --- $f*b_{q,a}$ is supported on $(N/2-N/R^2,N+N/R^2]$ with $|f*b_{q,a}(n)|\leq\|f\|_\infty$ there, so restricting to $N/2<n\leq N$ costs only the two boundary intervals of total length $O(N/R^2)$ --- we can write
\begin{align*}
  \widehat{f*b_{q,a}}(\alpha)=\sum_{N/2-N/R^2\leq n \leq N+N/R^2} f*b_{q,a}(n)e(\alpha n)+O\left(\|f\|_\infty \frac{N}{R^2}\right)
\end{align*}
and thus
\begin{align*}
    \sup_{|\alpha-a/q|\leq R/N} \Bigg| \sum_{N/2-N/R^2\leq n \leq N+N/R^2} f*b_{q,a}(n)e(\alpha n)\Bigg|\leq 2N \max_{N/2-N/R^2\leq n \leq N+N/R^2} | f*b_{q,a}(n)|.
\end{align*}
We denote by $\tau(\chi)$ the usual Gauss sum. Since $f$ is supported on $R$-rough integers, for $f(n')\neq 0$ we have $(n',q)=1$ for $q \le R$, so that $e_q(an') = \frac{1}{\varphi(q)} \sum_{\chi(q)} \tau(\bar\chi) \chi(a) \chi(n')$. Thus
\begin{align*}
    |f*b_{q,a}(n)|&=\frac{R^2}{2N}\left|\sum_{|n'-n|\leq N/R^2}f(n')e_q(an') \right| \\
    &=\frac{R^2 }{2 N\varphi(q)} \left|\sum_{\chi(q)} \tau(\overline{\chi})\chi(a)\sum_{|n'-n|\leq N/R^2}f(n')\chi(n') \right|\\
    &\leq \frac{R^2 q^{1/2} }{N} \max_{\substack{\chi\pmod q\\ \chi \textnormal{ primitive}}}\left|\sum_{|n'-n|\leq N/R^2}f(n')\chi(n') \right|.
\end{align*}
In the last step, each character $\chi\pmod q$ agrees on the support of $f$ with the primitive character $\chi^\ast$ inducing it, and we bounded $|\tau(\overline{\chi})|\leq q^{1/2}$ before maximising over these primitive characters.
Recalling that we take the maximum over $q\leq R$, the lemma follows.
\end{proof}

The next lemma allows us to move a smooth function inside the convolution with $b_R$, provided its derivative is not too large. We require it only in the proof of Proposition~\ref{prop:rLR} to move factors of $n^{\widetilde{\beta}-1}$ (where $\widetilde\beta$ is the exceptional zero of Definition~\ref{def:exceptional}) in or out of the convolution.

\begin{lemma}\label{le:smoothconvo}
    Let $\psi$ be a smooth function, $g$ any arithmetic function, $N/2 <n\leq N$ and $R\geq 2$. Then we have
    \begin{align*}
        \psi(n)(g*b_R)(n)=(\psi g)*b_R(n)+O\left(\frac{N \log R}{R^2} \|\psi'\1_{[N/3,2N]}\|_\infty \|g\|_{\infty}\right).
    \end{align*}
\end{lemma}
\begin{proof}
    Since $R\geq 2$, Definition~\ref{def:bdef} gives that $b_R$ is supported on $|m|\leq N/R^4\leq N/16$, so $|t-n|\leq N/R^4$ implies $t\in [N/3,2N]$. By the mean value theorem we have
      \begin{align*}
        &\psi(n)(g*b_R)(n)-(\psi g)*b_R(n)\\
        =&\sum_{|m|\leq N/R^4}(\psi(n)-\psi(n-m))g(n-m)b_R(m)\\
        \ll&  \sup_{|t-n|\leq N/R^4} |\psi'(t)| \sum_{|m|\leq N/R^4}  |m g(n-m) b_R(m)| \\
        \ll&  \|\psi'\1_{[N/3,2N]}\|_\infty \|g\|_{\infty}\sum_{|m|\leq N/R^4} |m| |b_R(m)|\\
        \ll&  \|\psi'\1_{[N/3,2N]}\|_\infty \|g\|_{\infty} \frac{R^4}{N} \sum_{|m|\leq N/R^4} |m| \sum_{r\leq R^2} (r,m)\\
        \ll& \frac{N \log R}{R^2} \|\psi'\1_{[N/3,2N]}\|_\infty \|g\|_{\infty},
    \end{align*}
    as desired. Here we used $|c_r(m)|\leq (r,m)$ and, with $M=N/R^4$, the estimates $\sum_{|m|\leq M}|m|(r,m)\leq 2M^2\sum_{d\mid r}\frac{\varphi(d)}{d}\leq 2M^2\tau(r)$ and $\sum_{r\leq R^2}\tau(r)\ll R^2\log R$.
\end{proof}

For functions that are supported on rough numbers, we now rewrite their convolution with $b_R$ in terms of multiplicative characters.
\begin{definition}[Smooth Heath--Brown kernel $\Lambda_{R,r}$]\label{def:LRr}
Recall that $G$ is the fixed smooth function of Definition~\ref{def:bdef}. For $r\geq 1$ we set
\begin{align}\label{eq:LRrdef}
	\Lambda_{R,r}(n)\coloneqq \sum_{\substack{q\\ (q,r)=1}} \frac{\mu(q)c_q(n)}{\varphi(q)} G\bigl( \frac{\log rq}{\log R}\bigr).
\end{align}
\end{definition}
Then we have the following lemma; see~\cite[eq. (6.6)]{green-sarkozy} for the analogue in Green's work.
\begin{lemma}\label{lem:f*b} Let $R\geq 2$. Let $f$ be any arithmetic function supported on $R^2$-rough integers only. For $N/2<n\leq N$, we have
\begin{align}\label{eq:f*b}
    f*b_R(n)= \sum_{r\leq R^2}\,\,\,\asum_{ \chi\pmod r} \chi(n)\frac{r}{\varphi(r)} \Lambda_{R,r}(n) \frac{R^4\sum_{|n'-n|\leq N/R^4} f(n')\overline{\chi}(n')}{2N}.
\end{align}
\end{lemma}
\begin{proof}
    By definition
    \begin{align*}
        f*b_R(n)=\sum_{r}G\bigl( \frac{\log r}{\log R}\bigr)\frac{R^4}{2N} \sum_{|n'-n|\leq N/R^4} f(n')c_r(n-n').
    \end{align*}
    By the support condition of $G$, we have $r\leq R^2$ and so $(n',r)=1$. We rewrite Ramanujan's sum as $c_r(m)=\sum_{d\mid (r,m)}\mu(r/d)d$ and get 
\begin{align*}
     f(n')c_r(n-n')&=\sum_{d\mid (r,n-n')} \mu(r/d)d\cdot f(n')\\
    &=\sum_{d\mid r} \mu(r/d)d \1_{n'\equiv n \pmod d}f(n')\\
    &= \sum_{d\mid r} \frac{\mu(r/d)d }{\varphi(d)}\sum_{\chi(d)} \chi(n)\overline{\chi}(n')f(n')\\
    &=\sum_{s\mid r}\,\,\, \asum_{\chi\pmod s}\chi(n)\overline{\chi}(n') f(n') \sum_{\substack{d'\\ d'\mid r/s}} \frac{\mu(r/(sd'))sd'}{\varphi(sd')}\\
    &=\sum_{\substack{s\mid r\\ (s,r/s)=1}}\,\,\, \asum_{\chi\pmod s}\chi(n)\overline{\chi}(n') f(n') \frac{s}{\varphi(s)}\sum_{d' \mid r/s} \frac{\mu((r/s)/d')d'}{\varphi(d')}\1_{(n,d')=1} \\
    &=\sum_{\substack{s\mid r\\ (s,r/s)=1}}\,\,\,  \asum_{\chi\pmod s}\chi(n) \frac{s}{\varphi(s)}\frac{\mu(r/s)c_{r/s}(n)}{\varphi(r/s)} f(n')\overline{\chi}(n'),
\end{align*}
where we sorted the $d\mid r$ sum by conductor: each $\chi\pmod d$ is induced by a unique primitive character $\chi^\ast\pmod s$ with $s\mid d$, and $\chi(n)=\chi^\ast(n)\1_{(n,d)=1}$ while $\chi(n')=\chi^\ast(n')$ since $(n',r)=1$, so writing $d=sd'$ with $d'\mid r/s$ leaves the inner sum over $d'$, in which only squarefree $r/s$ coprime to $s$ survive. This inner sum $\sum_{d' \mid r/s} \frac{\mu((r/s)/d')d'}{\varphi(d')}\1_{(n,d')=1}$ is multiplicative in $m=r/s$. At $m=p$ it equals $\mu(p)+\frac{p}{\varphi(p)}\1_{p\nmid n}$, which is $\frac{1}{p-1}$ if $p\nmid n$ and $-1$ if $p\mid n$, both equal to $\frac{\mu(p)c_p(n)}{\varphi(p)}$, while at prime powers $p^k$ with $k\geq 2$ it vanishes, as does $\mu(p^k)$. Multiplying over $p\mid m$ gives $\frac{\mu(r/s)c_{r/s}(n)}{\varphi(r/s)}$, using also that $(n,s)=1$ in the support of the sum. Summing over $n'$ and $r$ now gives the result.
\end{proof}

In the proof of Theorem~\ref{thm:F-P-C_intro} or the more general Theorem~\ref{thm:gallagher_application}, we will show that $|\Lambda_{R,r}|$ is bounded by $H_R$ and so behaves Cram\'er-like.

We next provide a suitable multiplicative upper bound for $|\Lambda_{R,r}|$ to be used together with Lemma~\ref{le:henriot}. It involves the function $H_R$ defined in Theorem~\ref{thm:F-P-C_intro}.
\begin{lemma}\label{lem:LRupper}
  For $n\in \mathbb{N}$ and $r\le R^2$, we have
    \begin{align*}
        \1_{(n,r)=1}\frac{r}{\varphi(r)}|\Lambda_{R,r}(n)|\leq H_R(n).
    \end{align*}
\end{lemma}
\begin{proof}
We may assume that $(n,r)=1$, as otherwise the claim is trivial.     Using the formula
	\begin{align*}
	   c_q(n)=\sum_{d\mid (q,n)}d\mu\left(\frac{q}{d}\right)  
	\end{align*}
	for Ramanujan sums (which is easily verified by noting that both sides are multiplicative in $q$), we can write
	\begin{align}
		\Lambda_{R,r}(n)&=\sum_{d\mid n}d\mu(d)\sum_{\substack{q\equiv 0\pmod d\\ (q,r)=1}}\frac{\mu(q)^2}{\varphi(q)}G\left(\frac{\log(qr)}{\log R}\right)\nonumber\\
		&=\sum_{\substack{d\mid n\\ (d,r)=1}}\frac{d\mu(d)}{\varphi(d)}\sum_{\substack{\ell \\ (\ell,dr)=1}}\frac{\mu(\ell)^2}{\varphi(\ell)}G\left(\frac{\log (\ell dr)}{\log R}\right),\label{eq:Lambda_Q}
	\end{align}
    where, writing the square-free $q$ with $d\mid q$ as $q=d\ell$ with $(d,\ell)=1$, we used $\mu(q)\mu(q/d)=\mu(d)\mu(q)^2$ to extract $d\mu(d)$ and the multiplicativity $\varphi(d\ell)=\varphi(d)\varphi(\ell)$, $\mu(d\ell)^2=\mu(\ell)^2$. To proceed further, we follow a Fourier-analytic approach similar to~\cite[equation (37)]{polymath}. Let $\psi$ be the Fourier transform of $t\mapsto e^t G(t)$. Then by Fourier inversion for $x\in \mathbb{R}$ we have
	\begin{align}\label{eq:Gx}
		e^xG(x) =\int_{\mathbb{R}}\psi(\xi)e^{-i\xi x}\d \xi.   
	\end{align}
	By the smoothness of $G$, for any $B>0$ we have $\psi(t)\ll_B (1+|t|)^{-B}$. Plugging~\eqref{eq:Gx} into~\eqref{eq:Lambda_Q}, we get
    \begin{align*}
        \Lambda_{R,r}(n)= \int_{\R}\psi(\xi) r^{-(1-i\xi)/\log R
        } \sum_{\substack{d\mid n\\ (d,r)=1}}\frac{d\mu(d)}{\varphi(d)d^{(1-i\xi)/\log R
        }}\sum_{\substack{\ell \\ (\ell,dr)=1}}\frac{\mu(\ell)^2}{\varphi(\ell)\ell^{(1-i\xi)/\log R
        }}\d \xi.
    \end{align*}
    Thus, the lemma follows if we can show that the double sum over $d,\ell$ is bounded by
    \begin{align}\label{eq:dellbound}
        \Bigl|\sum_{\substack{d\mid n\\ (d,r)=1}}\frac{d\mu(d)}{\varphi(d)d^{(1-i\xi)/\log R
        }}\sum_{\substack{\ell \\(\ell,dr)=1}}\frac{\mu(\ell)^2}{\varphi(\ell)\ell^{(1-i\xi)/\log R
        }} \Bigr| \ll \frac{\varphi(r)}{r}\tau(n)  (\log R) h_\xi(n),
    \end{align}
    with a multiplicative function $h_\xi$ given on the primes by
    \begin{align*}
        h_{\xi}(p)=\min\Bigl\{1,10(1+|\xi|)\frac{\log p}{\log R} \Bigr\}.
    \end{align*}
    We successively write the summations over $\ell$ and $d$ as products so that
    \begin{align*}
        &\sum_{\substack{d\mid n\\ (d,r)=1}}\frac{d\mu(d)}{\varphi(d)d^{(1-i\xi)/\log R
        }}\sum_{\substack{\ell \\(\ell,dr)=1}}\frac{\mu(\ell)^2}{\varphi(\ell)\ell^{(1-i\xi)/\log R
        }}\\
        =&\sum_{\substack{d\mid n\\ (d,r)=1}}\frac{d\mu(d)}{\varphi(d)d^{(1-i\xi)/\log R
        }} \prod_{p_1\nmid dr}\Bigl(1+\frac{1}{(p_1-1)p_1^{(1-i\xi)/\log R}} \Bigr)\\
        =&\prod_{p_1\nmid r}\Bigl(1+\frac{1}{(p_1-1)p_1^{(1-i\xi)/\log R}} \Bigr)\prod_{\substack{p_2\mid n\\ p_2\nmid r}}\left(1-\frac{p_2}{(p_2-1)p_2^{(1-i\xi)/\log R}} \Bigl(1+\frac{1}{(p_2-1)p_2^{(1-i\xi)/\log R}} \Bigr)^{-1}\right)\\
        =&\prod_{p_1\nmid r n}\Bigl(1+\frac{1}{(p_1-1)p_1^{(1-i\xi)/\log R}} \Bigr)\prod_{\substack{p_2\mid n\\ p_2\nmid r}}\left(1-\frac{1}{p_2^{(1-i\xi)/\log R}}\right).
    \end{align*}
    Here the last equality uses the identity $\bigl(1+\tfrac{1}{(p_2-1)p_2^{s}}\bigr)\bigl(1-\tfrac{p_2}{(p_2-1)p_2^{s}}\bigl(1+\tfrac{1}{(p_2-1)p_2^{s}}\bigr)^{-1}\bigr)=1-p_2^{-s}$ with $s=(1-i\xi)/\log R$, the factor $1+\tfrac{1}{(p_2-1)p_2^{s}}$ for each $p_2\mid n$ being drawn from the product over $p_1$.
    To estimate the product over $p_1$, we first note that since $r\leq R^2$ we have
    \begin{align*}
        \frac{r}{\varphi(r)}  \prod_{p\mid r }\Bigl(1+\frac{1}{(p-1)p^{(1-i\xi)/\log R}} \Bigr)^{-1}&\ll \prod_{p\mid r }\Bigl(1+\frac{1}{p-1}\Bigr)\Bigl(1+\frac{1}{(p-1)p^{1/\log R}}\Bigr)^{-1}\\
        &\ll  \prod_{p\mid r }\Bigl(1+\frac{1-p^{-1/\log R}}{p-1}\Bigr)\\&\ll \prod_{p\mid r} \Bigl(1+O\bigl(\frac{\log p}{p \log R}\bigr)\Bigr)\\
        &\ll \exp \Bigl(O\bigl(\sum_{p\leq \log r} \frac{\log p}{p \log R}\bigr) \Bigr)\\
        &\ll 1.
    \end{align*}
     Thus the product over $p_1$ can be estimated by
    \begin{align*}
        \Bigl|\prod_{p_1\nmid r n}\Bigl(1+\frac{1}{(p_1-1)p_1^{(1-i\xi)/\log R}} \Bigr)\Bigr|&\leq \prod_{p_1\nmid r}\Bigl(1+\frac{1}{(p_1-1)p_1^{1/\log R}} \Bigr)\\
        &\ll  \frac{\varphi(r )}{r }\exp\left(\sum_{p_1}p_1^{-1-1/(\log R)}\right)\\
        &\ll \frac{\varphi(r)}{r } \log R.
    \end{align*}
    We next consider the factor involving $p_2$. 
   By using a trivial bound in the case that $(1+|\xi|)\tfrac{\log p_2}{\log R}>1/10$ and Taylor approximation in the remaining case, we get 
   \begin{align*}
       \left|1-\frac{1}{p_2^{(1-i\xi)/\log R}}\right| \leq 2\min\Bigl\{1,10(1+|\xi|)\frac{\log p_2}{\log R}\Bigr\}.
   \end{align*}
    The bound~\eqref{eq:dellbound} and with it the lemma follows. 
    \end{proof}

\section{Main-sieve removal}\label{sec:MSremoval}

The functions $g_1,g_2,g_3$ in Proposition~\ref{le_chensieve} are products of $\Lambda$ or $\Lambda_{E_3^*}$ with a main-sieve and a pre-sieve or a difference of pre-sieves. In this section, our goal is to show that the main-sieve component can be replaced by a constant, in the approximate sense of Definition~\ref{def_approx}. We do so in the generality $\Lambda^*\in\{\Lambda,\Lambda_{E_3^*},r_{P_1}\}$, as the case $\Lambda^*=r_{P_1}$ is needed later for the Cram\'er model.

\begin{proposition}[Main-sieve removal]\label{prop_MSremoval}
Assume the following conditions.
 \begin{itemize}
        \item $\Lambda^* \in \{\Lambda,\Lambda_{E_3^*},r_{P_1} \}$ with $P_1=N^{\delta_1^4}$. 
        \item $f_\textup{M}(n)=\sum_{d\mid n}\lambda_{\mathrm{M}}(d)$ is a main-sieve as in Definition~\ref{def:mainsieve}.
        \item  $f_{\textup{Pre}} \in \{\omega,\Omega\}$ is a pre-sieve.
    \end{itemize}
    Let $\epsilon \geq N^{-(\delta_1/10)^{4}}$.
		Then we have
		\begin{align*}
			\Lambda^*(n)f_{\textup{M}}(n+2) f_{\textup{Pre}}(n+2)\approx_{\epsilon} \left(\prod_{P_1\le p < N^{1/10}} \left(1-\frac{1}{p}\right)^{-1} \sum_{d}\frac{\lambda_{\mathrm{M}}(d)}{\varphi(d)}\right) \Lambda^*(n)f_{\textup{Pre}}(n+2)
		\end{align*}
		and similarly with the roles of $n$ and $n+2$ reversed.     
		Let further $\chi$ be a primitive quadratic character with conductor at most $\widetilde{R}=N^{2\delta_1^5}$ and let $\mathfrak{c}(n)=  \1_{\chi(n)=-1} $ or $\mathfrak{c}(n)=\1_{\chi(n)=\chi(n+2)=1}$. Expanding these indicators into characters via $\1_{\chi(n)=-1}=\tfrac12(1-\chi(n))$ and $\1_{\chi(n)=\chi(n+2)=1}=\tfrac14(1+\chi(n))(1+\chi^{+}(n))$ produces the four twists $\mathfrak b\in\{1,\chi,\chi^{+},\chi\chi^{+}\}$ met below.
		We have
		\begin{align*}
			\mathfrak{c}(n) \Lambda^*(n)f_{\textup{M}}(n+2)f_{\textup{Pre}}(n+2)\approx_{\epsilon} \mathfrak{c}(n)\left(\prod_{P_1\le p < N^{1/10}} \left(1-\frac{1}{p}\right)^{-1}\sum_{d}\frac{\lambda_{\mathrm{M}}(d)}{\varphi(d)}\right) \Lambda^*(n) f_{\textup{Pre}}(n+2)
		\end{align*}
		and similarly with the roles of $n$ and $n+2$ reversed.  
  \end{proposition}
  
We make two remarks on this. First, the proof does not require the pre-sieves to be precisely as in Definition~\ref{def:presieve}, any arithmetic function that is of the form $\lambda\star 1$ with a bounded function $\lambda$ that is supported on $P_1$-smooth integers $d\leq D_1$ would do. Second, the parts of the statement involving a quadratic character $\chi$ are required to deal with a possible exceptional character.

We consider only the case of the roles of $n,n+2$ as in the displays of the proposition; the reversed role case works the same way. By Lemma~\ref{lem_fourierimpliesapprox}, Proposition~\ref{prop_MSremoval} follows if we can show 
\begin{align}
    \|\mathfrak{b}\Lambda^* (f_{\textup{M}}-c_{\textup{M}})^+f_{\textup{Pre}}^+ \|^\wedge_\infty \ll N^{1-2(\delta_1/10)^{4}},
\end{align}
where 
\begin{align}\label{eq:c_Mdef}    c_{\textup{M}}=\prod_{P_1\le p < N^{1/10}} \left(1-\frac{1}{p}\right)^{-1}\sum_{d}\frac{\lambda_{\mathrm{M}}(d)}{\varphi(d)}
\end{align}
and $\mathfrak{b}\in\{1,\chi,\chi^{+},\chi\chi^{+}\}$. Then, applying Lemma~\ref{lem_twistsievechar} (possibly twice) together with the assumed bound for the conductor of $\chi$, we see that it suffices to show  
\begin{align*}
    \|\Lambda^* (f_{\textup{M}}-c_{\textup{M}})^+f_{\textup{Pre}}^+ \|^\wedge_\infty \ll N^{1-3 (\delta_1/10)^4}.
\end{align*}
With the help of the triangle inequality, this follows from the following separate major and minor arc statements:
\begin{align}\label{eq:MSremovalFourier1}
    \|\Lambda^* (f_{\textup{M}}-c_{\textup{M}})^+f_{\textup{Pre}}^+ \|^\wedge_{\mathfrak{M}(R_1)}\ll N^{1-3 (\delta_1/10)^4},\\
     \|\Lambda^* f_{\textup{M}}^+ f_{\textup{Pre}}^+ \|^\wedge_{\mathfrak{m}(R_1)}+  |c_{\textup{M}}| \|\Lambda^* f_{\textup{Pre}}^+ \|^\wedge_{\mathfrak{m}(R_1)}\ll N^{1-3 (\delta_1/10)^4} \label{eq:MSremovalFourier2}.
\end{align}

In the next four subsections, we first import some routine combinatorial decompositions, and then show a variant of the Bombieri--Vinogradov theorem with a power saving. Afterwards we apply it to obtain the major arc case bound~\eqref{eq:MSremovalFourier1}. Finally, we prove the complementing minor arc bound~\eqref{eq:MSremovalFourier2} with the help of Bombieri--Vinogradov type estimates twisted by minor arc additive phases based on~\cite{matomaki}.

\subsection{Combinatorial decompositions}
For both major and minor arc bounds, we rely on type I and II sum decompositions for the relevant $\Lambda^*$ functions. 

To handle some of the smooth normalisations (like $\log n$ or $c_{E_3^*}(n)$), we first state a simple consequence of partial summation.
\begin{lemma}\label{lem:partsum}
    Let $(a_n)_{n\in\N}$ be a sequence of complex numbers and let $f$ be a smooth function such that for $t\in [\sqrt N,N]$ we have
    \begin{align*}
        |f(t)|&\asymp \log N,\\
        |f'(t)|&\ll 1/t.
    \end{align*}
    Then we have
    \begin{align*}
        \Bigl|\sum_{\sqrt{N}<n\leq N}a_n\Bigr|&\ll \frac{1}{\log N}\max_{\sqrt N<\xi\leq N} \Bigl|\sum_{\sqrt{N}<n\leq \xi} a_n f(n) \Bigr|, \\
         \Bigl|\sum_{\sqrt{N}<n\leq N}a_n f(n)\Bigr|&\ll \log N \max_{\sqrt N<\xi\leq N} \Bigl|\sum_{\sqrt{N}<n\leq \xi} a_n  \Bigr|.
    \end{align*}
\end{lemma}

\begin{lemma}[Combinatorial decomposition]\label{lem:combdecomp1}
Let $(\gamma_n)_{n\in\N}$ be a sequence of complex numbers. Let $\Lambda^* \in\{\Lambda,\Lambda_{E_3^*},r_{P} \}$ with $P<N^{1/3}$. Let $Q\leq N^{1/10}$ and $N/2<N'\leq N$. Assume we have the following type I and II estimates:
\begin{align*}
    \Bigl|\sum_{\substack{N_1/2 < n_1 \leq N_1\\ N/2<n_1n_2\leq \xi}} a_{n_1} b_{n_2} \gamma_{n_1 n_2}\Bigr|\ll N Q^{-1},
\end{align*}
 whenever $N_1\leq N^{1/3}, \xi \leq N'$, $|a_n|\leq  \tau(n)\log n$ and $b_n\equiv 1$ or $b_n\equiv \log n$;
 \begin{align*}
    \Bigl|\sum_{\substack{N_1/2 < n_1 \leq N_1\\ N_2/2 <  n_2 \leq N_2}} a_{n_1} b_{n_2} \gamma_{n_1 n_2}\Bigr|\ll N Q^{-1},
\end{align*}
whenever $N/2<N_1N_2<4N'$ and $N_1\in [N^{1/3},2N^{1/2}]$, $|a_n|,|b_n|\leq \tau(n)\log n$.

Then we have that
\begin{align*}
   \Bigl |\sum_{N/2<n\leq N'}\Lambda^*(n)\gamma_{n}\Bigr| \ll N^{1+o(1)} Q^{-1}.
\end{align*}
\end{lemma}
\begin{proof}

We consider separately the case of each choice for $\Lambda^*$.

\textbf{The case of $\Lambda$.}

  When $\Lambda^*=\Lambda$, we use Vaughan's identity~\cite[Cor. 17.4]{cribro} to obtain
  \begin{align*}
      \sum_{N/2<n\leq N'} \Lambda(n) \gamma_n \ll T_1+T_2,
  \end{align*}
  where
  \begin{align*}
    T_1&= \sum_{N/2<n\leq N'} \log n \,\,\gamma_n \sum_{\substack{b\mid n\\ b\leq (N/2)^{1/3}}}\mu(b)
  \end{align*}
  and
  \begin{align*}
      T_2=\sum_{N/2<n\leq N'}\gamma_n \sum_{\substack{bc=n\\ b,c > (N/2)^{1/3} }}\alpha_b \beta_c
  \end{align*} 
  for some $|\alpha_b|\leq 1,|\beta_c|\leq \log c$. Writing $\log n=\log b+\log(n/b)$, we can estimate $T_1$ with the type I bound. For $T_2$, we split $b,c$ into dyadic intervals $(N_1/2,N_1],(N_2/2,N_2]$ and observe that the sum is empty unless $N/2<N_1 N_2<4N'$. Next, we denote by $ \|\theta\|$ the distance of $\theta$ to the nearest integer and remove the cross-condition by 
  \begin{align}
    \label{eq:removcross}  \1_{N/2 < bc \leq N'}&=\int_0^1 \sum_{N/(2 c) \leq j \leq N'/c}e(\theta j) e(-\theta b) \d \theta\\
      & \nonumber =\int_0^1 \nu_{\theta,c} \min\{N'/c, \|\theta\|^{-1}\} e(-\theta b) \d \theta
  \end{align}
for some measurable function $|\nu_{\theta,c}|\leq 1$. Taking the maximum over $\theta$ and incorporating $\nu_{\theta,c}$ and $e(-\theta b)$ into the coefficients gives a valid type II sum after observing that either $N_1\leq 2\sqrt{N'}$ or $N_2\leq 2\sqrt{N'}$.

\textbf{The case of $\Lambda_{E_3^*}$.}

If $\Lambda^{*}=\Lambda_{E_3^*}$, we follow the strategy in~\cite[Sections 10.2, 10.3]{matomaki-shao}. Recalling~\eqref{eq:c_E}, we can remove the weight $1/c_{E_3^*}(n)$ with Lemma~\ref{lem:partsum}. Afterwards, we are left with 
  \begin{align*}
     \1_{n\in I}( \1_{n\in B_1}/2+\1_{n\in B_2})
  \end{align*}
 for some interval $I \subseteq (N/2,N']$, with $B_i$ given in~\eqref{eq:B1def} and~\eqref{eq:B2def}. 
  
Observe that for both $B_1$ and $B_2$, we have $N^{1/3}/2\leq p_1p_2\leq N^{2/3}$. Thus, letting $n_1=p_1p_2$ and $n_2=p_3$, dyadically splitting the variables, and removing the cross-condition $n_1n_2\in I$ by~\eqref{eq:removcross}, we obtain an admissible type II sum. 

\textbf{The case of $r_{P}$.}
This case follows from Harman's fundamental theorem; see for example~\cite[Theorem 3.1]{harman2007} and a subsequent application of~\eqref{eq:removcross} to remove the cross condition in the type II case.
\end{proof}

We remark that if we were to consider $r_P$ with $N^{1/3}<P\leq N^{1/2}$, we could get a similar decomposition with Ramar\'e's identity~\cite[eq. (17.9)]{cribro}. This introduces some additional complications with square divisors that require care to not interact badly with the trivial bound for $\mathfrak{v}_P$ in the sequel. For this reason, we do not pursue this generality here.

\subsection{Bombieri--Vinogradov with power saving}
The Bombieri--Vinogradov theorem implies that for any $A\geq 1$, the approximation
\begin{align}\label{eq_BVapprox}
\sum_{\substack{n\leq N \\ n\equiv a\pmod d}}\Lambda(n)=\frac{1}{\varphi(d)}\sum_{n\leq N}\Lambda(n)(1+O_A((\log N)^{-A}))
\end{align}
is valid for almost all $d$ up to $N^{1/2}(\log N)^{-10A}$. The quality of this saving stems from the application of the Siegel--Walfisz theorem for small conductor characters and the large sieve inequality for large conductor characters. The term on the right-hand side of~\eqref{eq_BVapprox} is the contribution of the principal character $\pmod{d}$, i.e., the unique character modulo $d$ that has conductor $1$. If we instead include the contribution of all the characters with conductor up to a value $P$ larger than any fixed power of $\log N$, the large sieve inequality will give a saving that is a fixed power of $P$.

This idea was applied by Drappeau~\cite{drappeau}, and similarly to~\cite[equation (5.1)]{drappeau} we write
\begin{align}\label{eq_uPdef}
\mathfrak{v}_P(n;q)&\coloneqq  \frac{1}{\varphi(q)}\sum_{\substack{\psi(q)\\ \text{cond}(\psi)\text{ not }P\text{-smooth}}}\psi(n)\\
&=	\frac{1}{\varphi(q)}\sum_{\substack{\psi(q)}}\psi(n)-\frac{1}{\varphi(q)}\sum_{\substack{\psi(q)\\ \text{cond}(\psi)\ P\text{-smooth}}}\psi(n). \nonumber
\end{align}
In~\cite{drappeau} the main term consists of the characters of conductor at most $P$, whereas for us it is more convenient to retain the characters whose conductor has a prime factor larger than $P$, since this is what our sieve decomposition produces. As a conductor that is not $P$-smooth exceeds $P$, the large sieve applies to these characters just as well.
For any residue class $a$ with $(a,q)=1$, we use $\overline{a}$ to denote the multiplicative inverse of $a$ modulo $q$. With this notation we can state the following type I and II estimate, see also~\cite[Lemma 5.2]{drappeau}.

\begin{lemma}[Type I and II in arithmetic progressions]\label{le_I,IIsimple}
	 Let $N_1, N_2, X\geq 1$ with $N_1N_2\asymp X$, and let $2\leq P< Q\leq X$. Then for any character $\chi$ of modulus at most $P$, we have 
	\begin{align}\label{eq_Isimple}
	\sum_{q\leq Q}\sum_{n_1 \sim N_1}  \max_{\substack{1\leq a\leq q\\ (a,q)=1}}\max_{\xi\leq X/N_1}\left|\sum_{ n_2 \leq \xi}\chi(n_2)\mathfrak{v}_{P}(n_1n_2\overline{a};q) \right|\ll N_1 Q^{3/2}P^{1/2}X^{o(1)}. 
	\end{align}
	Let further $\alpha_n,\beta_n$ be complex sequences with $|\alpha_n|, |\beta_n|\leq 1$, supported on $(N_1/2,N_1]$ and $(N_2/2,N_2]$, respectively. Then we have
	\begin{align}\label{eq_IIsimple} 
	&\sum_{q\leq Q}\max_{\substack{1\leq a\leq q\\ (a,q)=1}}\left|\sum_{n_1,n_2}\alpha_{n_1}\beta_{n_2} \mathfrak{v}_{P}(n_1n_2\overline{a};q)\right|
	\ll X^{1/2+o(1)}(Q+X^{1/2}/P+N_1^{1/2}+N_2^{1/2}).
	\end{align}
\end{lemma}

\begin{proof}We first prove~\eqref{eq_Isimple}. We have
	\begin{align*}
	\sum_{\substack{n_2\leq \xi}}\chi(n_2)\mathfrak{v}_{P}(n_1n_2\overline{a};q)=\frac{1}{\varphi(q)}\sum_{\substack{\psi(q)\\ \text{cond}(\psi)\text{ not }P\text{-smooth}}} \psi(n_1\overline{a}) \sum_{n_2\leq \xi}\chi\psi(n_2).
	\end{align*}
	Since $\chi$ has modulus at most $P$, the character $\chi \psi $ (of modulus $\leq PQ$) is never principal. Let $\chi\psi$ have modulus $s$ and let it be induced by a character $\psi'$ of modulus $1<s'\leq s$. Applying M\"obius inversion and then the P\'olya--Vinogradov inequality, we get
 \begin{align*}
     \sum_{n_2 \le \xi}\chi\psi(n_2)=\sum_{d\mid s}\mu(d)\psi'(d)\sum_{n_2'\leq \xi/ d}\psi'(n_2')\ll X^{o(1)} \sqrt{PQ},
 \end{align*}
 uniformly in $\xi\leq X/N_1$.
 Thus, trivially bounding the number of $\psi \pmod{q}$ with $\text{cond}(\psi)$ not $P$-smooth by $\varphi(q)$, we get~\eqref{eq_Isimple}.
	
	The bound~\eqref{eq_IIsimple} follows from the large sieve inequality, as applied for large conductor characters in the proof of~\cite[Proposition 17.4]{iw-kow}.
\end{proof}

From the previous lemma we can deduce the following version of the Bombieri--Vinogradov theorem with an additional character twist and improved error term.

\begin{lemma}[Bombieri--Vinogradov with large savings]\label{le_BV_rough} Let  $\Lambda^* \in \{\Lambda,\Lambda_{E_3^*},r_{P_1} \}$ with $P_1=N^{\delta_1^4}$. Assume that $\chi$ is a character of modulus at most $R_1=N^{\delta_1^4/100}$ and let $Q\leq N^{1/2-\delta_1}$. Then for any $N/2<N'\leq N$ we have 
	\begin{align*}
	\sum_{q\leq Q}\max_{\substack{1\leq a\leq q\\ (a,q)=1}} \left|\sum_{N/2<n\leq N'} \Lambda^*(n)\chi(n)\mathfrak{v}_{P_1}(n\overline{a};q) \right|\ll N^{1+o(1)}P_1^{-1}.
	\end{align*}
\end{lemma}

\begin{proof}

 By Lemma~\ref{lem:combdecomp1}, it suffices to show the following type I and II bounds for $\gamma_n=\mathfrak{v}_{P_1}(n\overline{a};q)\chi(n)$. We claim that
 \begin{align}\label{eq:bombvinoI}
    \Bigl|\sum_{\substack{N_1/2 < n_1 \leq N_1\\ N/2 < n_1 n_2\leq N'}} a_{n_1} b_{n_2} \mathfrak{v}_{P_1}(n_1n_2\overline{a};q)\chi(n_1 n_2)\Bigr|\ll N^{1+o(1)} P_1^{-1}
\end{align}
 whenever $N_1\leq N^{1/3}$, $|a_n|\leq \tau(n)\log n$ and $b_n\equiv 1$ or $b_n\equiv \log n$ and
 \begin{align}\label{eq:bombvinoII}
    \Bigl|\sum_{\substack{N_1/2 < n_1 \leq N_1\\ N_2/2 <  n_2 \leq N_2\\ N/2<n_1n_2\le N'}} a_{n_1} b_{n_2}  \mathfrak{v}_{P_1}(n_1n_2\overline{a};q)\chi(n_1 n_2)\Bigr|\ll N^{1+o(1)} P_1^{-1},
\end{align}
whenever $N/2<N_1 N_2<4N'$ with $N_1\in [N^{1/3},2N^{1/2}]$ and $|a_n|,|b_n|\leq \tau(n)\log n$.

Both type I and II cases are amenable to Lemma~\ref{le_I,IIsimple}, after absorbing $\tau(n)\log n\ll N^{o(1)}$ into the coefficients. In the type I case we distinguish two ranges of $N_1$.

For $N_1\leq N^{1/4}$ we pull the sum over $n_1$ outside with the triangle inequality and remove the logarithm with $\log n=\int_{1}^{N}\1_{t\leq n}t^{-1}\d t$. For every fixed $n_1$ the remaining range of $n_2$ is an interval contained in $[1,2N/N_1]$, so~\eqref{eq_Isimple} with $X\asymp N$ and $P=P_1$ applies to each of the two cutoffs. Summing over $q$, we bound the left-hand side of~\eqref{eq:bombvinoI} by
	\begin{align*}
	  \ll N^{1/4+o(1)}Q^{3/2}P_1^{1/2}\ll N^{1/4+3/4-3\delta_1/2+\delta_1^4/2+o(1)}\ll N^{1-\delta_1}\ll NP_1^{-1}.
	\end{align*}
	
For $N^{1/4}<N_1\leq N^{1/3}$ we instead interpret the sum as a type II sum. Splitting $n_2$ into $O(1)$ dyadic intervals $(N_2/2,N_2]$ with $N_2\asymp N/N_1<N^{3/4}$ and removing the cross-condition by~\eqref{eq:removcross}, we obtain an admissible type II sum. Since $N_1^{1/2}\leq N^{1/6}$ and $N_2^{1/2}\ll N^{3/8}$, summing over $q$ and applying~\eqref{eq_IIsimple} with $X\asymp N$ and $P=P_1$ bounds the left-hand side of~\eqref{eq:bombvinoI} by
	\begin{align*}
	  \ll N^{1/2+o(1)}\bigl(Q+N^{1/2}/P_1+N_1^{1/2}+N_2^{1/2}\bigr)\ll N^{1/2+o(1)}\bigl(N^{1/2-\delta_1}+N^{1/2}/P_1+N^{3/8}\bigr)\ll N^{1+o(1)}P_1^{-1}.
	\end{align*}

 Similarly, in the type II case, summing over $q$ and applying~\eqref{eq_IIsimple} bounds the left-hand side of~\eqref{eq:bombvinoII} by
 \begin{align*}
    \ll N^{1/2+o(1)}(N^{1/2-\delta_1}+N^{1/2}/P_1+N^{1/3})\ll N^{1-\delta_1^4+o(1)}=N^{1+o(1)}P_1^{-1}.
 \end{align*}
\end{proof}
\subsection{Major arc contribution}\label{subsec:major}

In this subsection we apply the Bombieri--Vinogradov theorem with large savings, Lemma~\ref{le_BV_rough}, to show~\eqref{eq:MSremovalFourier1}.
\begin{lemma}[Key major arc estimate]\label{lem:majorarcbound}
Assume the conditions of Proposition~\ref{prop_MSremoval}. Then we have
    \begin{align*}
        \|\Lambda^*  f_{\mathrm{{Pre}}}^+ f_\textup{M}^+ -c_{\textup{M}}\Lambda^*  f_{\mathrm{{Pre}}}^+ \|^{\wedge}_{\mathfrak{M}(R_1)}\ll N^{1+o(1)}(P_1^{-1/2} R_1^{1/2}+R_1^{-1}),
    \end{align*}
    with $c_{\textup{M}}$ given by~\eqref{eq:c_Mdef}, and similarly with $^+$ placed on $\Lambda^{*}$ instead. In particular,~\eqref{eq:MSremovalFourier1} holds.
\end{lemma}
\begin{proof}
We only consider the first case, the second case where $+$ is placed on $\Lambda^{*}$ is very similar. Since $\Lambda^*$ is supported on $P_1$-rough numbers outside a negligible set, we can appeal to Lemma~\ref{lem:Mtochar}. 
Therefore it suffices to show that
\begin{align}\label{eq:majorarcbound1}
\left|\sum_{\substack{|n'-n|\leq N/R_1^2 \\ }}\Lambda^*(n')\chi(n')\1_{n'\leq N} f_{\mathrm{{Pre}}}(n'+2)(f_{\mathrm{{M}}}(n'+2)-c_{\textup{M}})\right|\ll N^{1+o(1)} P_1^{-1/2} R_1^{-2}
\end{align}
uniformly for $N/2-N/R_1^2\leq n\leq N+N/R_1^2$ and for primitive characters $\chi$ of modulus $\leq R_1$.
By assumption we can write (incorporating the condition $\1_{(n,6)=1}$ of the pre-sieve via M\"obius inversion into the sieve weights)
  \begin{align*}
      f_{\mathrm{{Pre}}}(n)&=\prod_{p < P_1}\left(1-\frac{1}{p}\right)^{-1}\sum_{d\mid n}\lambda_{\mathrm{Pre}}(d),\\
      f_\textup{M}(n)&=\prod_{P_1\le p < N^{1/10}}\left(1-\frac{1}{p}\right)^{-1}\sum_{d\mid n}\lambda_{\mathrm{M}}(d),
  \end{align*}
  where the weights $\lambda_{\mathrm{Pre}},  \lambda_{\mathrm{M}}$ are supported on integers having only prime divisors up to or greater than $P_1$, respectively. Recall further that by~\eqref{eq:c_Mdef} we have
  \begin{align*}
      c_{\textup{M}}=\prod_{P_1\le p < N^{1/10}} \left(1-\frac{1}{p}\right)^{-1}\sum_{d}\frac{\lambda_{\mathrm{M}}(d)}{\varphi(d)}.
  \end{align*}
To proceed, it is convenient to include a condition $(n',d)=1$ by setting
  \begin{align*}
      c'_{\textup{M}}(n')=\prod_{P_1\le p < N^{1/10}} \left(1-\frac{1}{p}\right)^{-1}\sum_{d}\frac{\lambda_{\mathrm{M}}(d)\1_{(n',d)=1}}{\varphi(d)}.
  \end{align*}
Since $\lambda_{\mathrm{M}}$ is supported on $P_1$-rough $d$, any $j=(n',d)>1$ satisfies $j\geq P_1$ and $j\mid n'$, so that
  \begin{align*}
      |c'_{\textup{M}}(n')-c_{\textup{M}}|\ll N^{o(1)} \sum_{P_1\leq d < N} \frac{\1_{(n',d)>1}}{d} \ll N^{o(1)} \sum_{j\geq P_1} \frac{\1_{j\mid n'}}{j}  \ll \frac{N^{o(1)} }{P_1}.
  \end{align*}
As $|\Lambda^*(n')f_{\mathrm{{Pre}}}(n'+2)|\ll N^{o(1)}$, replacing $c_{\textup{M}}$ by $c'_{\textup{M}}(n')$ changes the left-hand side of~\eqref{eq:majorarcbound1} by $\ll N^{1+o(1)}P_1^{-1}R_1^{-2}$, which is admissible as $P_1$ is a power of $N$. Thus it suffices to prove~\eqref{eq:majorarcbound1} with $c_{\textup{M}}$ replaced by $c'_{\textup{M}}(n')$.

  Let $d= d_{\le P_1}d_{>P_1}$ be the decomposition of $d$ into its $P_1$-smooth and $P_1$-rough parts. Let
  \begin{align*}
    \lambda(d)=\lambda_{\mathrm{Pre}}(d_{\leq P_1})\lambda_{\mathrm{M}}(d_{>P_1}).
  \end{align*}
Then we can write
  \begin{align*}
      f_{\mathrm{{Pre}}}(n'+2)(f_{\mathrm{{M}}}(n'+2)-c'_{\textup{M}}(n'))=\prod_{p<N^{1/10}}\left(1-\frac{1}{p}\right)^{-1}\sum_{d}\lambda(d)\left(\1_{n'\equiv -2\pmod d}-\frac{\1_{(n',d_{>P_1})=1}\1_{n'\equiv -2 \pmod{d_{\leq P_1}}}}{\varphi(d_{>P_1})}\right).
  \end{align*}
 To make Lemma~\ref{le_BV_rough} applicable, we rewrite, denoting by $\chi_{0}^{(s)}$ the principal character $\pmod{s}$ and by $\overline{-2}$ the inverse of $-2$ modulo $d$, 
\begin{align*}\nonumber 
&\1_{n'\equiv -2\pmod{d}}-\frac{\1_{(n',d_{>P_1})=1}\1_{n'\equiv -2 \pmod{d_{\leq P_1}}}}{\varphi(d_{>P_1})}\\
&=\nonumber\frac{1}{\varphi(d)}\sum_{\psi_1(d)}\psi_1(\overline{-2}n')-\frac{1}{\varphi(d)}\sum_{\psi_2(d_{\leq P_1})}\psi_2(\overline{-2}n')\chi_{0}^{(d_{>P_1})}(\overline{-2}n')\\
&=\nonumber \frac{1}{\varphi(d)}\sum_{\substack{\psi(d)\\ \text{cond}(\psi)\text{ not }P_1\text{-smooth}}}\psi(\overline{-2}n')\\
&=\nonumber \mathfrak{v}_{P_1}(\overline{-2}n';d).
\end{align*}
Coming back to the left-hand side of~\eqref{eq:majorarcbound1}, we have reduced it to showing
\begin{align}
\left|\sum_{(d,2)=1}\lambda(d)\sum_{\substack{|n'-n|\leq N/R_1^2\\}}\Lambda^*(n')\chi(n')\1_{n'\leq N}\mathfrak{v}_{P_1}(\overline{-2}n';d)\right| \ll N P_1^{-1/2} R_1^{-2} (\log N)^{-2}.
\end{align}
Recall that by the definition of pre- and main-sieve we have $d\leq D_1 D_{\mathrm{M},1}D_{\mathrm{M},2}=:Q\leq N^{1/2-\delta_1}$ and trivially estimate the weight $|\lambda(d)|\ll \tau(d)^{O(1)}\ll n^{o(1)}$, so that it suffices to prove 
\begin{align}\label{eq_prop1final_2}
\sum_{\substack{d\leq Q \\(d,2)=1}}\left|\sum_{\substack{|n'-n|\leq N/R_1^2\\}}\Lambda^*(n')\chi(n')\1_{n'\leq N}\mathfrak{v}_{P_1}(\overline{-2}n';d)\right| \ll NR_1^{-2} P_1^{-1/2-1/1000},
\end{align}
where we have used that $P_1$ is a power of $N$ to absorb $N^{o(1)}$ losses. After subtracting the upper and lower bound of the interval and applying the triangle inequality, we can apply Lemma~\ref{le_BV_rough} to estimate the left-hand side of~\eqref{eq_prop1final_2} by
\begin{align*}
    N^{1+o(1)}P_1^{-1}=N^{1+o(1)-\delta_1^4}\leq N^{1-\delta_1^4(2/100+1/2+1/1000)}= NR_1^{-2}P_1^{-1/2-1/1000}.
\end{align*}
This shows the main statement of the lemma. Since $P_1^{-1/2}R_1^{1/2}\leq R_1^{-2}=N^{-2\delta_1^4/100}\leq N^{-3(\delta_1/10)^4}$, the estimate~\eqref{eq:MSremovalFourier1} is an immediate consequence.
\end{proof}

\subsection{Minor arcs}

In this subsection, we show~\eqref{eq:MSremovalFourier2}. The required ingredients can essentially be found in~\cite{matomaki} and~\cite{matomaki-shao}, but since in the former our type of weights is not treated explicitly and the latter has different aims that do not easily adapt to power saving, we give the details. The central ingredients are the following type I and II estimates that can be found in~\cite{matomaki}. 

 \begin{lemma}[Type I and II minor arc estimate] \label{lem:t1t2minor}
 Let $a,q\in \mathbb{N}$ with $(a,q)=1$. Let $|a_n|,|b_n|\leq \tau(n)\log n$, and assume that $|\alpha-a/q|\leq q^{-2}$ with $q\leq X$. Then for any fixed $c\in \Z_{\neq 0}$, we have the type I estimate 
 \begin{align}
     &\sum_{r\leq D} \bigg| \sum_{\substack{N_1/2 < n_1 \leq N_1\\ X/(2n_1)<n_2\leq X/n_1\\ n_1n_2\equiv c \pmod{r}}} a_{n_1} b_{n_2} e(\alpha n_1n_2) \bigg|\ll X^{1+o(1)} \bigg(\frac{1}{q}+\frac{DN_1}{X}+\frac{q}{X} \bigg)^{1/2}. \label{eq:typIminor}
\end{align}
We have the type II estimate 
     \begin{align}
    &\sum_{\substack{S_1/2 < s_1 \leq  S_1\\ S_2/2 < s_2 \leq  S_2}}\bigg| \sum_{\substack{N_1/2 < n_1 \leq N_1 \\ X/(2N_1) < n_2\leq X/N_1\\ n_1n_2\equiv c \pmod{s_1s_2}}} a_{n_1} b_{n_2} e(\alpha n_1n_2) \bigg|   \nonumber\\
      \ll&  X^{1/2+o(1)}\bigg(X\left(\frac{1}{q}+\frac{q}{X}\right)^{1/3}+\min\bigg\{\frac{X^{11/10} S_1^2 S_2^2}{N_1}+N_1,S_1^2 S_2^2+N_1 S_2+\frac{XS_1^2S_2}{N_1}+\frac{X}{S_2^{1/3}}\bigg\}\bigg)^{1/2}.\label{eq:typIIminor}
 \end{align}
 \end{lemma}

\begin{proof}
The type I bound~\eqref{eq:typIminor} follows from~\cite[Lemma 6]{matomaki} in the same way as~\cite[eq. (6)]{matomaki} is obtained.  The type II bound~\eqref{eq:typIIminor} is a combination of~\cite[equation (7)]{matomaki} with $\eta=1/10$ and the estimate at the end of the proof of~\cite[Proposition 9]{matomaki}. 
\end{proof}

We can now provide the minor arc bound that complements Lemma~\ref{lem:majorarcbound}.

\begin{lemma}[Minor arc estimate]\label{lem:minor}
Let $\Lambda^{*}\in \{\Lambda, \Lambda_{E_3^{*}},r_{P_1}\}$ with $P_1=N^{\delta_1^4}$. Let $f_{\textup{M}},f_{\textup{Pre}}$ be as in Proposition~\ref{prop_MSremoval} and assume that $R\leq N^{\delta_1^2}$. Then we have
\begin{align}\label{eq:minorarcbound1}
    \|\Lambda^* f_{\textup{M}}^+ f_{\textup{Pre}}^+ \|^{\wedge}_{\mathfrak{m}(R)}&\ll N^{1+o(1)} R^{-1/6},\\
    \label{eq:minorarcbound2}  \|\Lambda^*  f_{\textup{Pre}}^+ \|^{\wedge}_{\mathfrak{m}(R)}&\ll N^{1+o(1)} R^{-1/6},
\end{align}
and similarly with the shift operator $^+$ changing places to be on $\Lambda^*$. In particular,~\eqref{eq:MSremovalFourier2} holds.

\end{lemma}

\begin{proof}
Moving the shift operator does not change the proof apart from flipping $+2$ to $-2$, so we only consider the stated case. We first show~\eqref{eq:minorarcbound1}.

By Lemma~\ref{lem:combdecomp1} it suffices to show that 
\begin{align}\label{eq:minorI}
  \max_{\alpha \in \mathfrak{m}(R)}\Bigl|\sum_{\substack{N_1/2 < n_1 \leq N_1\\ N/2 < n_1 n_2\leq N}} a_{n_1} b_{n_2}  f_{\textup{M}}(n_1n_2+2) f_{\textup{Pre}}(n_1n_2+2) e(\alpha n_1n_2)\Bigr|\ll N^{1+o(1)} R^{-1/6},
\end{align}
 whenever $N_1\leq N^{1/3}$, $|a_n|\leq \tau(n)\log n$ and $b_n\equiv 1$ or $b_n\equiv \log n$; and
 \begin{align}\label{eq:minorII}
     \max_{\alpha \in \mathfrak{m}(R)}\Bigl|\sum_{\substack{N_1/2 < n_1 \leq N_1\\ N_2/2 <  n_2 \leq N_2\\ N/2 < n_1n_2 \le N}} a_{n_1} b_{n_2}  \ f_{\textup{M}}(n_1n_2+2) f_{\textup{Pre}}(n_1n_2+2) e(\alpha n_1n_2)\Bigr|\ll N^{1+o(1)} R^{-1/6},
\end{align}
whenever $N/2<N_1 N_2<4N$ and $N_1\in [N^{1/3},2N^{2/3}]$, $|a_n|,|b_n|\leq \tau(n)\log n$.

We recall that by Definitions~\ref{def:presieve} and~\ref{def:mainsieve}, we have
\begin{align*}
      f_{\textup{M}}(n)&=C_{f_\textup{M}}\sum_{d_1d_2\mid n}\lambda_{\mathrm{M},1}(d_1)\lambda_{\mathrm{M},2}(d_2),\\
     f_{\textup{Pre}}(n)&=\1_{(n,6)=1}\prod_{p < P_1}\bigl(1-\frac{1}{p}\bigr)^{-1}\sum_{d\mid n}\lambda_{\textup{Pre}}(d),
\end{align*}
where the sieve weights $\lambda_{\textup{Pre}}$ are $1$-bounded, $\lambda_{\textup{Pre}}$ is of level $D_1$, and $|\lambda_{\mathrm{M},j}(d)|\leq \tau(d)^{O(1)}$,  $\lambda_{\mathrm{M},1}$ is of level $D_{\mathrm{M},1}=N^{1/3-\delta_1}$, and $\lambda_{\mathrm{M},2}(d_2)$ a sum of $O_{\delta_1}(1)$ well-factorable weights of level $D_{\mathrm{M},2}=N^{1/6-\delta_1}$ and $|C_{f_\textup{M}}|\ll (\log N)$.

We first consider the type I estimate~\eqref{eq:minorI}. We combine the three sieve weights into one, apply the triangle inequality and then the type I case of Lemma~\ref{lem:t1t2minor} to obtain for~\eqref{eq:minorI} the bound
\begin{align*}
    \ll& N^{o(1)} \max_{\alpha \in \mathfrak{m}(R)} \sum_{r\leq D_{1}D_{\mathrm{M},1}D_{\mathrm{M},2}} \bigg| \sum_{\substack{N/2< n_1 n_2\leq N\\ n_1n_2\equiv -2 \pmod{r}\\ N_1/2<n_1\leq N_1 }} a_{n_1}b_{n_2} e(\alpha n_1n_2) \bigg|\ll N^{1+o(1)} \bigg(\frac{1}{R}+\frac{N_1  D_{1}D_{\mathrm{M},1}D_{\mathrm{M},2}}{N}\bigg)^{1/2},
\end{align*}
where we used the fact that every $\alpha \in \mathfrak{m}(R)$ satisfies $|\alpha-a/q|\leq 1/(q(N/R))$ for some integer $R<q\leq N/R$ and some $a$ coprime to $q$. 
Since 
\begin{align*}
    N_1  D_{1}D_{\mathrm{M},1}D_{\mathrm{M},2}\leq N^{1/3+1/3-\delta_1+1/6-\delta_1+\delta_1^3/100}\leq N/R,
\end{align*}
this bound is $\ll N^{1+o(1)}R^{-1/2}$.

For the type II estimate~\eqref{eq:minorII}, we proceed similarly to~\cite[end of proof of Proposition 9]{matomaki} and split the weights dyadically into segments $d\sim D$ with 
\begin{align*}
    D\leq D_{1}D_{\mathrm{M},1}D_{\mathrm{M},2}.
\end{align*}    

If $D\leq N^{1/10}$, we apply the type II case of Lemma~\ref{lem:t1t2minor} with $S_1=D$ and $S_2=1$, using the first part of the minimum to obtain for the left-hand side of~\eqref{eq:minorII} the upper bound
\begin{align*}
    \ll N^{1/2+o(1)} \left(NR^{-1/3}+N^{9/10} \right)^{1/2}\ll N^{1+o(1)} R^{-1/6}.
\end{align*}

For the remaining case of $N^{1/10}\leq D\leq D_{1}D_{\mathrm{M},1}D_{\mathrm{M},2}$ we use the factorisation properties and choose
\begin{align*}
    S_2=\min\left\{D,\frac{N^{1-\delta_1/2}}{N_1}\right\}.
\end{align*}
This choice fulfils $S_2\geq \min\{D,N^{1/3-\delta_1/2}\}$, and thus, using the facts that $\lambda_{\mathrm{M},2}$ is a sum of $O_{\delta_1}(1)$ well-factorable weights and that $D_{1}N^{1/3-2\delta_1}\leq N^{1/3-\delta_1/2}$, we can group and dyadically decompose the sieve weights into $O((\log N)^2)$ components of the shape $\lambda' \star \lambda''$,
where $|\lambda'(d)|,|\lambda''(d)|\leq \tau(d)^{O(1)}$, $\lambda'$ is supported on $[S_2/2,S_2]$, and $\lambda''$ is supported on $[S_1/2,S_1]$ with $S_1\coloneqq D/S_2$. Observe that for these choices we have
\begin{align*}
    &N S_1^2 S_2/N_1= D^2  N / (S_2N_1)\leq D^2 N^{\delta_1/2}, \\
   & N^{1/10} \leq S_2  \leq N^{1-\delta_1/2}/N_1. 
\end{align*}
Thus, using the second part of the minimum in the type II case of Lemma~\ref{lem:t1t2minor}, we can estimate the left-hand side of~\eqref{eq:minorII}  by
\begin{align*}
    \ll N^{1/2+o(1)} \left(NR^{-1/3}+N^{1-\delta_1/2}+D^2 N^{\delta_1/2}\right)^{1/2}.
\end{align*}
 Since $D\leq  D_{1}D_{\mathrm{M},1}D_{\mathrm{M},2}=N^{1/2-2\delta_1+\delta_1^3/100}$, $R\leq N^{\delta_1^2}$ and $N_1\geq N^{1/3}$, this is bounded by
\begin{align*}
    \ll N^{1+o(1)}R^{-1/6},
\end{align*}
and~\eqref{eq:minorII} follows.

The estimate~\eqref{eq:minorarcbound2} follows directly from the same argument (taking $f_{\textup{M}}(n)=1$ and skipping the steps about the decomposition of $f_{\textup{M}}$).

Recalling $R_1=N^{\delta_1^4/100}$,~\eqref{eq:MSremovalFourier2} follows immediately.
\end{proof}

We end this section by stating a minor arc bound for the Cram\'er model twisted by a power function. It is used in the proof of Proposition~\ref{prop:rLR}.

\begin{lemma}\label{lem:roughexceptionalminorarc}
   Let $2\leq (\log N)^{100}<R^2\leq P\leq N^{1/4}$ and $1/2\leq \beta \leq 1$. We have
    \begin{align*}
        \sup_{\alpha\in \mathfrak{m}(R)} \left|\sum_{n\leq N} r_P(n)n^{\beta-1} e(\alpha n)\right|\ll N R^{-1/3}.
    \end{align*}
\end{lemma}
\begin{proof}
Since
\begin{align}\label{eq:expnbeta}
n^{\beta-1}=1+(\beta-1)\int_{1}^{n}t^{\beta-2}\d t,    
\end{align}
and $\int_{1}^{N}t^{\beta-2}\d t\ll \log N$, by the triangle inequality it suffices to prove the claim without the $n^{\beta-1}$ factor and with the summation range replaced by $n\leq \xi$ for $\xi\leq N$. We discard the contribution of $n\leq NR^{-1/2}$ trivially, dyadically decompose and apply Lemma~\ref{lem:combdecomp1}, noting that $(N R^{-1/2})^{1/3}>N^{1/4}$. The claim then follows from standard minor arc estimates, see for example~\cite[Lemmas 13.7 and 13.8]{iw-kow}.
\end{proof}

\section{The additive problem with sieves}\label{sec:SiftedSums}

In this section, we consider a sifted convolution sum related to our additive problem. We are concerned with expressions of the form 
\begin{align} \label{eq_siftedsums_main}
    \sum_{n_1+n_2=m}f_1(n_1)f_2(n_1+2)f_3(n_2)f_4(n_2+2),
\end{align}
where each of the functions $f_i$ is \emph{sifted} in the sense that the contribution of $n$ that have small prime divisors is negligible.

In Subsection~\ref{subsec:6.1} we prove some basic results related to sieves. In Subsection~\ref{subsec:6.2} we consider~\eqref{eq_siftedsums_main} in the case of all the $f_i$ being pre-sieves as in Definition~\ref{def:presieve} for which we prove an asymptotics, see Proposition~\ref{prop_presieveasmyp} and Proposition~\ref{prop_fundlem_exceptional} for additional twists by quadratic characters to handle the possible exceptional character. In Subsection~\ref{subsec:6.3} we consider~\eqref{eq_siftedsums_main} in the case where three of the $f_i$ are absolute values of pre-sieves and the last one is of the form $|\Lambda_{R,r}|$ (recall~\eqref{eq:LRrdef}), for which we provide an upper bound of correct order in Proposition~\ref{prop:LambdaQ}.

\subsection{Sieve lemmas}\label{subsec:6.1}
We require two sieve lemmas to understand~\eqref{eq_siftedsums_main}. 

The first result is an identity to relate sums of sieve weights over multiples of a given integer to the sieve.

\begin{lemma} \label{lem_sie1}
    Let $\mathcal{P}$ be a square-free positive integer, $(\lambda_d)_{d\geq 1}$ be a sequence of complex numbers supported on the divisors of $\mathcal{P}$ only and assume that $e\mid\mathcal{P}$. Let $g$ be a multiplicative function such that $0\leq g(p) <1$ for all primes $p$ and $g(p)>0$ for all $p\mid e$, and define a multiplicative function $h$ supported on square-free numbers by $h(p)=\frac{g(p)}{1-g(p)}$.  Then we have
	\begin{align*}
	   g(e)\sum_{d\mid \mathcal{P}/e}\lambda_{de}g(d)=\prod_{p\mid\mathcal{P}}\bigl(1-g(p)\bigr) \mu(e)h(e)\sum_{b\mid\mathcal{P}}\theta(b)h(b) \frac{\mu((b,e))}{h((b,e))},
	\end{align*}
	where $\theta(n)=\sum_{d\mid n}\lambda_d$.
\end{lemma}
\begin{proof}
    This is shown in the first display of the proof of~\cite[Lemma 6.1]{matoalmost} for the case $g(d)=1/d$; the general case works exactly the same. Note that this fixes a mistake in~\cite[Lemma 6.18]{cribro}.
\end{proof}

The second sieve result is of fundamental lemma type. It is a minor variant of~\cite[Lemma 3.2]{MatoMeri}, see also ~\cite[after equation (58)]{matoalmost}.
\begin{lemma}\label{lem_fundlem}
    Let $f\in \{\omega,\Omega\}$ be pre-sieves as in Definition~\ref{def:presieve}, and let $\beta=200$. Write $\mathcal{P}(y)=\prod_{\substack{ 5\leq p< y}}p$ and $\mathcal{P}_r=\mathcal{P}(P_1^{\bigl(\frac{\beta-1}{\beta+1}\bigr)^{r}})$, and assume that $s=\frac{\log D_1}{\log P_1}>\beta+2$. Then we have
	\begin{align*}
	   |f(n)-r_{P_1}(n)|\leq \prod_{p< P_1}\left(1-\frac{1}{p}\right)^{-1}\tau(n)^2\sum_{\substack{r\in \mathbb{N}\\ r>( s-\beta-1)/2}}4^{-r}\1_{(n,\mathcal{P}_r)=1}.
	\end{align*}
\end{lemma}

\begin{proof}
The beta sieve is a combinatorial sieve that is constructed by iteration with certain cutoff parameters, see~\cite[equations (6.31)--(6.34), (6.54)]{cribro}.  From this it follows immediately that for the upper bound sieve $\Omega$ we have the identity
	\begin{align*}
	   \Omega(n)=r_{P_1}(n)+ \prod_{p < P_1}\left(1-\frac{1}{p}\right)^{-1}\sum_{r \text{ odd}}V_r(n),
	\end{align*}
	and similarly for the lower bound sieve $\omega$ we have
	\begin{align*}
	   \omega(n)=r_{P_1}(n)- \prod_{p < P_1}\left(1-\frac{1}{p}\right)^{-1}\sum_{r \text{ even}}V_r(n),
	\end{align*}
	where
	\begin{align*}
	   V_r(n)=\sum_{\substack{p_1\cdots p_r \mid n\\ p_i\mid \mathcal{P}(P_1), \,\,\, p_r<p_{r-1}<\ldots <p_1 \\ p_1p_2\cdots p_r p_r^\beta\geq D_1 \\ p_1p_2\cdots p_hp_h^{\beta} <D_1 \text{ for all } h<r, h\equiv r\pmod 2 }}\1_{(n/(p_1\cdots p_r),\mathcal{P}(p_r))=1}.
	\end{align*}
	By~\cite[Corollary 6.6]{cribro}, in the last sum we have $p_r\geq P_1^{\bigl(\frac{\beta-1}{\beta+1}\bigr)^{r/2}}$. We also have $D_1\leq p_1^{r+\beta}\leq P_1^{r+\beta}$, which is impossible unless $r\geq s-\beta$. We complete the proof by repeating the argument in~\cite[after equation (58)]{matoalmost}: In the sum defining $V_r$, the distinct primes $p_i\mid\mathcal{P}(P_1)$ satisfy $p_i\ge5$, so writing $\nu(n)$ for the number of prime divisors of $n$ we have $\nu(n)\geq r$ so that $2^{\nu(n)-r}\geq 1$. Hence,
    \begin{align*}
        V_r(n)\leq \1_{(n,\mathcal{P}_{r/2})=1} 2^{\nu(n)-r} \sum_{k \mid n} 1 \leq  \1_{(n,\mathcal{P}_{r/2})=1} \tau(n)^2 2^{-r}.
    \end{align*}
     Writing $r=2r'+1$ or $r=2r'$, respectively for the case of an upper bound or lower bound sieve, the lemma follows.
\end{proof}

\subsection{Asymptotics for the sifted convolution sum}\label{subsec:6.2}
We now apply Lemmas~\ref{lem_sie1} and~\ref{lem_fundlem} to prove an asymptotic formula for the additive problem with four pre-sieves (Proposition~\ref{prop_presieveasmyp}), its variant twisted by the exceptional character (Proposition~\ref{prop_fundlem_exceptional}), and the resulting upper and lower bounds (Corollary~\ref{cor_upperlowersieve}).
\begin{prop}[The additive problem with pre-sieves]\label{prop_presieveasmyp}
    For $1\leq i \leq 4$, let the functions $f_i\in \{ \omega, \Omega \}$ be pre-sieves as in Definition~\ref{def:presieve}. For any natural number $m$ and any $1\leq X\leq m$, we have
    \begin{align}\label{eq_propsifted_asymptotics}\begin{aligned}
	   &\sum_{n\leq X} f_1(n)f_2(n+2) f_3(m-n)f_4(m-n+2)\\
       =& X \mathfrak{S}(m)\left(1+ O\left(e^{-\frac{\log D_1}{10 \log P_1}}+\frac{\log m}{P_1 }\right)\right)+O(D_1^4 (\log P_1)^4).
       \end{aligned}
	\end{align}
\end{prop}
	
\begin{proof}
We write  $\mathcal{P}=\prod_{5\leq p < P_1}p$, 
expand out the sieves $f_i$ according to their definitions, and observe that $(n,6)=1, (n+2,6)=1, (m-n,6)=1, (m-n+2,6)=1$ is equivalent to $m\equiv 4 \pmod 6, n\equiv 5\pmod 6$. This leads to
    \begin{align}
       \nonumber &\sum_{n\leq X} f_1(n)f_2(n+2) f_3(m-n)f_4(m-n+2) \\
        =& \frac{3^4}{\prod_{p<P_1}(1-1/p)^4} \1_{m\equiv 4\pmod 6} \sum_{d_1,d_2,d_3,d_4\mid\mathcal{P}}  \prod_{i=1}^4\lambda_{i}(d_i)\sum_{\substack{n\leq X\\ d_1\mid n \\ d_2\mid n+2 \\ d_3\mid m-n \\ d_4\mid m-n+2}} \1_{n\equiv 5(6)}. \label{eq_sieveopen}
	\end{align}
    There exists a solution to the divisibility conditions in the innermost sum if and only if 
	\begin{align}
\begin{aligned}
(d_1,d_2)&=1,\\
(d_3,d_4)&=1,\\
(d_1,d_3)&\mid m,\\
(d_1,d_4)&\mid m+2,\\
(d_2,d_3)&\mid m+2,\\
(d_2,d_4)&\mid m+4.
\end{aligned}
\label{eq:conditions}
\end{align}
Let us abbreviate these conditions as $(d_1,d_2,d_3,d_4)\in \mathcal{D}_m$.
	Under these conditions, the sum over $n$ in~\eqref{eq_sieveopen} equals
	\begin{align*}
		\frac{X}{6[d_1,d_2,d_3,d_4]}+O(1).
	\end{align*}
    The $O(1)$ error term contributes to~\eqref{eq_sieveopen} at most $O((\log P_1)^4D_1^4)$ by Mertens' theorem. Hence,~\eqref{eq_sieveopen} becomes 
    \begin{align}\label{eq:sievemainsum}
     \frac{3^4 X}{6\prod_{p<P_1}(1-1/p)^4} \1_{m\equiv 4 \pmod 6}  \sum_{\substack{d_1,d_2,d_3,d_4\mid\mathcal{P}\\ (d_1,d_2,d_3,d_4)\in \mathcal{D}_m}} \frac{\prod_{i=1}^4\lambda_{i}(d_i)}{[d_1,d_2,d_3,d_4]}+ O((\log P_1)^4D_1^4).
    \end{align}
    We want to evaluate the summation over each of the $d_i$ with Lemma~\ref{lem_sie1} and afterwards apply Lemma~\ref{lem_fundlem}. To make this possible we need to carefully disentangle the coprimality and divisibility conditions in order to not lose relative to the size of the expected main term in the error terms.

    In the rest of the proof, we use $\sideset{}{'}\sum$ to indicate that the summation variables are supported on divisors of $\mathcal{P}$ only. Next, write $e_1=(d_1,d_3), e_2=(d_1,d_4), e_3=(d_2,d_3), e_4=(d_2,d_4)$ and further $d_1=d_1'e_1e_2$, $d_2=d_2'e_3e_4$, $d_3=d_3'e_1e_3$, $d_4=d_4'e_2e_4$. Note that this is possible since by $(d_3,d_4)=1$ we have $(e_1,e_4)=1$ (and $(e_2,e_3)=1$ etc.), and by $(d_1,d_2)=1$ we have $(e_1,e_3)=1$ (and so on). Note further that~\eqref{eq:conditions} implies that the variables $d'_i$ and $e_j$ are mutually coprime. Thus, $[d_1,d_2,d_3,d_4]=d_1'd_2'd_3'd_4'e_1e_2e_3e_4$ and the central sum in~\eqref{eq:sievemainsum} becomes
	\begin{align*}
		&\sideset{}{'}\sum_{\substack{d_1, d_2, d_3, d_4\\ (d_1,d_2,d_3,d_4)\in \mathcal{D}_m}}\frac{\prod_{i=1}^4\lambda_{i}(d_i)}{[d_1,d_2,d_3,d_4]}\\
		=&\sideset{}{'}\sum_{\substack{e_1\mid m\\ e_2, e_3\mid  m+2 \\ e_4\mid m+4\\ e_i \text{ coprime}}}\frac{1}{e_1 e_2 e_3 e_4} \sideset{}{'}\sum_{\substack{d_1'\\ (d_1',e_3e_4)=1}} \frac{\lambda_1(d_1'e_1e_2)}{d_1'} \sideset{}{'}\sum_{\substack{d_2' \\ (d_2',d_1'e_1e_2)=1}} \frac{\lambda_2(d_2'e_3e_4)}{d_2'} \\
		&\times\sideset{}{'}\sum_{\substack{d_3' \\ (d_3',d_1'd_2'e_2 e_4)=1}} \frac{\lambda_3(d_3'e_1e_3)}{d_3'} \sideset{}{'}\sum_{\substack{d_4' \\ (d_4',d_1' d_2' d_3' e_1 e_3)=1}} \frac{\lambda_4(d_4'e_2e_4)}{d_4'}.
	\end{align*}
	To ease notation, write $\mathbf{e}=e_1e_2e_3e_4$ and remove the $'$ from the $d_i'$ variables again. The strategy to approach~\eqref{eq:sievemainsum} is to evaluate the sums over $d_4,d_3,d_2,d_1$ in the last display with Lemma~\ref{lem_sie1}, successively exchanging the order of summation. Note that since the function $\theta$ of Lemma~\ref{lem_sie1} associated with a given sequence $\lambda_i$ is $\prod_{p<P_1}(1-1/p) \cdot f_i$, each successive application will cancel one of the $(1-1/p)^{-1}$ factors in~\eqref{eq:sievemainsum}. 

    Starting with the summation over $d_4$, we apply of Lemma~\ref{lem_sie1} with $g(d)=\frac{\1_{(d,d_1d_2d_3e_1 e_3)=1}}{d}$ and $e=e_2 e_4$. This gives us
	\begin{align*}
		& \prod_{p<P_1}\left(1-\frac{1}{p}\right)^{-1}\sideset{}{'}\sum_{\substack{d_4 \\ (d_4,d_1 d_2 d_3 e_1 e_3)=1}} \frac{\lambda_4(d_4e_2e_4)}{d_4} \\
		=& \prod_{p\mid \mathcal{P}/(d_1d_2d_3e_1e_3)}\left(1-\frac{1}{p}\right) \frac{\mu(e_2 e_4) e_2 e_4}{\varphi(e_2 e_4)}  \sideset{}{'}\sum_{\substack{b_4\\ (b_4,d_1 d_2 d_3e_1e_3)=1}}\frac{f_4(b_4)}{\varphi(b_4)} \mu((b_4,e_2 e_4))\varphi((b_4,e_2 e_4))\\
		=&\prod_{p\mid \mathcal{P}}\left(1-\frac{1}{p}\right) \frac{d_1 d_2 d_3 \mathbf{e} \mu(e_2 e_4) }{\varphi(d_1d_2d_3 \mathbf{e})}\sideset{}{'}\sum_{j_4\mid e_2e_4}\mu(j_4)\sideset{}{'}\sum_{\substack{b_4\\ (b_4,d_1d_2d_3\mathbf{e})=1}} \frac{f_4(b_4j_4)}{\varphi(b_4)}.
	\end{align*}
   For $i\in \{2,3,4\}$, we denote by $\varphi_i$ the multiplicative function given on primes $p$ by $p-i$. We only employ this notation for $d\mid \mathcal{P}$, so that $(d,6)=1$ and $d$ is square-free, ensuring $\varphi_i(d)>0$. By an application of Lemma~\ref{lem_sie1}, the relevant $d_3$ sum now is
	\begin{align*}
		&\prod_{p<P_1}\left(1-\frac{1}{p}\right)^{-1}\sideset{}{'}\sum_{\substack{d_3 \\ (d_3,b_4 d_1d_2  e_2 e_4)}}\frac{\lambda_3(d_3e_1e_3)}{\varphi(d_3)}\\
		=&\prod_{p\mid \mathcal{P}}\left(1-\frac{1}{\varphi(p)}\right) \frac{\varphi(d_1 d_2 b_4 \mathbf{e})\mu(e_1 e_3) }{\varphi_2(d_1d_2b_4\mathbf{e})} \sideset{}{'}\sum_{j_3\mid e_1e_3}\mu(j_3)\sideset{}{'}\sum_{\substack{b_3\\ (b_3,d_1 d_2 b_4 \mathbf{e})=1}}\frac{f_3(b_3j_3)}{\varphi_2(b_3)}.
		\end{align*}
	The $d_2$ sum becomes 
	\begin{align*}
		& \prod_{p<P_1}\left(1-\frac{1}{p}\right)^{-1}\sideset{}{'}\sum_{\substack{d_2 \\ (d_2,d_1b_3 b_4 e_1e_2)=1}} \frac{\lambda_2(d_2e_3e_4)}{\varphi_2(d_2)} \\
		=&\prod_{p\mid \mathcal{P}}\left(1-\frac{1}{\varphi_2(p)}\right) \frac{\varphi_2(d_1b_3b_4e_1e_2e_3e_4)\mu(e_3 e_4) }{\varphi_3(d_1b_3b_4e_1e_2e_3e_4)} \sideset{}{'}\sum_{j_2\mid e_3e_4}\mu(j_2) \sideset{}{'}\sum_{\substack{b_2\mid \mathcal{P}\\ (b_2,d_1 b_3 b_4\mathbf{e})=1}}\frac{f_2(b_2j_2)}{\varphi_3(b_2)}.
	\end{align*}
	Finally, for the $d_1$ sum we get 
	\begin{align*}
		&\prod_{p<P_1}\left(1-\frac{1}{p}\right)^{-1}\sideset{}{'}\sum_{\substack{d_1\\ (d_1,b_2 b_3 b_4 e_3e_4)=1}} \frac{\lambda_1(d_1e_1e_2)}{\varphi_3(d_1)}\\
		=&\prod_{p\mid \mathcal{P}}\left(1-\frac{1}{\varphi_3(p)}\right)\frac{\varphi_3(b_2b_3b_4\mathbf{e})\mu(e_1 e_2) }{\varphi_4(b_2b_3b_4\mathbf{e})} \sideset{}{'}\sum_{j_1\mid e_1e_2}\mu(j_1)\sideset{}{'}\sum_{\substack{b_1\\ (b_1,b_2 b_3 b_4\mathbf{e})=1}}\frac{f_1(b_1j_1)}{\varphi_4(b_1)}.
		\end{align*}
	Note that for $p\geq 5$
	\begin{align*}
		\left(1-\frac{1}{p}\right)    \left(1-\frac{1}{\varphi(p)}\right)\left(1-\frac{1}{\varphi_2(p)}\right)\left(1-\frac{1}{\varphi_3(p)}\right)=1-\frac{4}{p}.
	\end{align*}
	Combining the four evaluations and recalling the normalisation of the $f_i$, we get for the term of interest in~\eqref{eq:sievemainsum} the identity
	\begin{align}
    \nonumber & \frac{3^4}{6} X\1_{m\equiv 4 \pmod 6} \frac{1}{\prod_{p<P_1}(1-1/p)^4}  \sum_{\substack{d_1,d_2,d_3,d_4\mid\mathcal{P}\\ (d_1,d_2,d_3,d_4)\in \mathcal{D}_m}} \frac{\prod_{i=1}^4\lambda_{i}(d_i)}{[d_1,d_2,d_3,d_4]}\\
		\nonumber=& \frac{3^4}{6} X\1_{m\equiv 4 \pmod 6} \prod_{p\mid \mathcal{P}}\left(1-\frac{4}{p}\right)\sideset{}{'}\sum_{\substack{e_1\mid m\\ e_2, e_3\mid  m+2 \\ e_4\mid m+4 \\ e_i \text{ coprime}}}\frac{1}{\varphi_4(\mathbf{e})}   \\
		& \times \sideset{}{'}\sum_{\substack{j_1\mid e_1e_2\\ j_2\mid e_3e_4\\ j_3\mid e_1e_3\\ j_4\mid e_2 e_4}}\sideset{}{'}\sum_{\substack{b_i\\ (b_i,b_j)=1,\,\forall i\neq j\\ (b_i,\mathbf{e})=1}} \mu(j_1)\mu(j_2)\mu(j_3)\mu(j_4)\frac{f_1(b_1j_1)f_2(b_2j_2)f_3(b_3j_3)f_4(b_4j_4)}{\varphi_4(b_1b_2b_3b_4)}. \label{eq:ebjsum}
	\end{align}
	Write $\mathbf{b}=b_1b_2b_3b_4$ and $\mathbf{j}=j_1j_2j_3j_4$. The term with $\mathbf{b}\mathbf{j}=1$ in~\eqref{eq:ebjsum} is the main term. Since $\mathcal{P}$ is coprime to the primes less than $5$, we have
    \begin{align}\label{eq:sums}
        \sideset{}{'}\sum_{\substack{e_1\mid m\\ e_2, e_3\mid  m+2 \\ e_4\mid m+4 \\ e_i \text{ coprime}}}\frac{1}{\varphi_4(\mathbf{e})}&= \sideset{}{'}\sum_{\substack{e_1\mid m}}\frac{1}{\varphi_4(e_1)}\sideset{}{'}\sum_{\substack{e_2\mid m+2}}\frac{1}{\varphi_4(e_2)}\sideset{}{'}\sum_{\substack{e_3\mid m+2}}\frac{1}{\varphi_4(e_3)}\sideset{}{'}\sum_{\substack{e_4\mid m+4}}\frac{1}{\varphi_4(e_4)}\\
        &=\prod_{\substack{5\leq p< P_1\\ p\mid m(m+4)}}\left(1+\frac{1}{p-4}\right)\prod_{\substack{5\leq p< P_1\\ p\mid m+2}}\left(1+\frac{2}{p-4}\right).\nonumber
    \end{align}
    Observe that the sums on the right-hand side of~\eqref{eq:sums} are multiplicative in $m$, $m+2$, $m+4$, respectively, and recall that $\mathcal{P}=\prod_{5\leq p < P_1}p$ is square-free. Since $3^4/6=27/2$ and $f_i(1)=\prod_{p\mid \mathcal{P}}(1-1/p)^{-1}$, it follows that terms with  $\mathbf{b}\mathbf{j}=1$ in~\eqref{eq:ebjsum} contribute
	\begin{align*}
    &X \1_{m\equiv 4\pmod 6}\frac{27}{2}\prod_{5\leq p < P_1}\left(1-\frac{4}{p}\right)\prod_{\substack{5\leq p< P_1\\ p\mid m(m+4)}}\left(1+\frac{1}{p-4}\right)\left(1-\frac{1}{p}\right)^{-4}\prod_{\substack{5\leq p < P_1\\ p\mid m+2}}\left(1+\frac{2}{p-4}\right)\\
    \eqqcolon& \mathfrak{S}(m;P_1),
	\end{align*}
    say.
    We can complete the product by introducing an acceptable error term. Indeed, the tail product is
	\begin{align*}
		&\prod_{p>P_1}\left(1-\frac{6p^2-4p+1}{(p-1)^4}\right)\prod_{\substack{p>P_1\\ p\mid m(m+4)}}\left(1+\frac{1}{p-4}\right)\prod_{\substack{p>P_1\\ p\mid m+2}}\left(1+\frac{2}{p-4}\right)\\
		=&   \prod_{p>P_1}\left(1+O\left(\frac{1}{p^2}\right) \right)\prod_{\substack{p>P_1\\ p\mid m(m+2)(m+4)}}\left(1+O\left(\frac{1}{p}\right) \right)\\
		=&\exp\left(O\left(\sum_{p>P_1}\frac{1}{p^2}+\sum_{\substack{p>P_1\\ p\mid m(m+2)(m+4)}}\frac{1}{p}\right)\right)\\
		=&\exp\left(O\left(\frac{1}{P_1}+\frac{\log m}{P_1 \log P_1}\right)\right)\\
		=&1+O\left(\frac{\log m}{P_1}\right).
	\end{align*}
 Thus $\mathfrak{S}(m;P_1)=\mathfrak{S}(m)\left(1+O\bigl(\frac{\log m}{P_1}\bigr)\right)$, where the completed singular series is as in~\eqref{eq:singserdef}.

    It remains to estimate the contribution of the remaining terms with  $\mathbf{b}\mathbf{j}\neq 1$ in~\eqref{eq:ebjsum}. To do so, we first apply the triangle inequality and then split up the $j_i, 1\leq i \leq 4$ into two components $l_i, 1\leq i \leq 8$ with the properties 
    \begin{align*}
        j_1&=l_1l_2 \text{ with } l_1\mid e_1, l_2\mid e_2,\\
        j_2&=l_3l_4 \text{ with } l_3\mid e_3, l_4\mid e_4,\\
        j_3&=l_5l_6 \text{ with } l_5\mid e_1, l_6\mid e_3,\\
        j_4&=l_7l_8 \text{ with } l_7\mid e_2, l_8\mid e_4.
    \end{align*}
    We can then upper bound the $\mathbf{e}$ sum with the additional $l_i$ conditions by
	\begin{align*}
		&\sum_{\substack{e_1\mid m, l_1l_5\mid e_1 \\ e_2, e_3\mid  m+2, l_2l_7\mid e_2, l_3l_6\mid e_3\\ e_4\mid m+4, l_4,l_8\mid e_4 \\ e_i\mid \mathcal{P}, e_i \text{ coprime}}}\frac{1}{\varphi_4(\mathbf{e})} \\
		\leq& \prod_{\substack{p\mid \mathcal{P}\\ p\mid m(m+4)}}\left(1+\frac{1}{p-4}\right)\prod_{\substack{p\mid \mathcal{P}\\ p\mid m+2}}\left(1+\frac{2}{p-4}\right) \frac{\prod_{i=1}^8\prod_{p\mid l_i}\left(1+O(1)/p\right)}{[l_1,l_5][l_2,l_7][l_3,l_6][l_4,l_8]}.
	\end{align*}
    This reduces the required estimate to showing that 
	\begin{align}\label{eq_L4.4_1}
	   \sum_{\substack{l_i, b_i\mid \mathcal{P}\\ \mathbf{bl}\neq 1}}\frac{\left|f_1(b_1l_1l_2)f_2(b_2l_3l_4)f_3(b_3l_5l_6)f_4(b_4l_7l_8)\right|\prod_{p\mid \mathbf{b}\mathbf{l}}\left(1+O(1)/p\right)}{\mathbf{b} [l_1,l_5][l_2,l_7][l_3,l_6][l_4,l_8]}\ll e^{-\frac{\log D_1}{10 \log P_1}}.
	\end{align}
	Write $(l_1,l_5)=d_1, (l_2,l_7)=d_2, (l_3,l_6)=d_3, (l_4,l_8)=d_4$, $\mathbf{d}=d_1d_2d_3d_4$, and $\mathbf{l}=\prod_{i=1}^8l_i$. We can then recombine the remaining variables and estimate the left-hand side of ~\eqref{eq_L4.4_1} by
	\begin{align*}
		\leq \sum_{\substack{d_i, b_i\mid \mathcal{P}\\ \mathbf{bd}\neq 1}} \frac{\left|f_1(b_1d_1 d_2)f_2(b_2d_3d_4)f_3(b_3d_1d_3)f_4(b_4d_2d_4)\right|\tau(\mathbf{b})^4\prod_{p\mid \mathbf{bd}\mathbf{l}}\left(1+O(1)/p\right)}{\mathbf{b} d_1d_2d_3d_4}.
	\end{align*}
    Note that for $n\mid \mathcal{P}$ with $n>1$, we have $r_{P_1}(n)=0$ and so Lemma~\ref{lem_fundlem} gives an upper bound for $|f(n)|$ in that case. Thus, writing $s=\frac{\log D_1}{\log P_1}$ and applying Lemma~\ref{lem_fundlem}, we upper bound the previous display by
	\begin{align*}
		\leq \sum_{r_1, r_2, r_3, r_4 > (s-\beta-1)/2} 4^{-r_1-r_2-r_3-r_4} \sum_{}\frac{\tau(\mathbf{b})^6 \tau(\mathbf{d})^2\prod_{p\mid \mathbf{bd}\mathbf{l}}\left(1+O(1)/p\right)}{\mathbf{db}},
	\end{align*}
	where the summation ranges over $b_i, d_i$ with
	\begin{align*}
		b_1d_1d_2&\mid \mathcal{P},&
		b_2d_3d_4&\mid \mathcal{P}, \\
		b_3d_1d_3&\mid \mathcal{P},&
		b_4d_2d_4&\mid \mathcal{P},
        \end{align*}
        and
        \begin{align*}
		(b_1d_1d_2,\mathcal{P}_{r_1})=
		(b_2d_3d_4,\mathcal{P}_{r_2})=
		(b_3d_1d_3,\mathcal{P}_{r_3})=
		(b_4d_2d_4,\mathcal{P}_{r_4})&=1,
	\end{align*}
    and where we wrote $\mathcal{P}_r$ for the product of primes between $5$ and $P_1^{\bigl(\frac{\beta-1}{\beta+1}\bigr)^r}$. Dropping in each of these conditions one of the $d_i$ in a suitable manner, we get the upper bound
	\begin{align}\label{eq_rsum}
		\leq \left(\sum_{r> (s-\beta-1)/2} 4^{-r}\sum_{\substack{b\mid \mathcal{P}\\ (b,\mathcal{P}_r)=1}} \frac{\tau^7(b) \prod_{p\mid b}\bigl(1+O(1)/p\bigr) }{b}\right)^4.
	\end{align}
	By Mertens' theorem we have
	\begin{align*}
		\sum_{\substack{b\mid \mathcal{P}\\ (b,\mathcal{P}_r)=1}}\frac{\tau(b)^{7}\prod_{p\mid b}\bigl(1+O(1)/p\bigr)}{b}&\ll \prod_{P_1^{\bigl(\frac{\beta-1}{\beta+1}\bigr)^r}<p< P_1}\Bigl(1+\frac{2^7}{p}\Bigr)\\
		&\ll \Bigl(\frac{\log P_1}{\log P_1^{\bigl(\frac{\beta-1}{\beta+1}\bigr)^r}} \Bigr)^{2^{7}}\\
		&= \Bigl(\frac{\beta+1}{\beta-1} \Bigr)^{128 r}.
	\end{align*}
	Recall $\beta=200$, whence we have $\Bigl(\frac{\beta+1}{\beta-1} \Bigr)^{128}/4\leq 0.9$ and the tail sum over $r$ in~\eqref{eq_rsum} is bounded by $O(e^{-s/10})$. This shows~\eqref{eq_L4.4_1} and completes the proof.
\end{proof}
	
As an immediate corollary, we get that upper and lower bound pre-sieves are approximately the same for our additive problem.

\begin{corollary}\label{cor_upperlowersieve}
	Let $f,g$ be arithmetic functions, satisfying $|f-g|(n)\leq \bigl(\Omega-\omega \bigr)(n)\Omega(n+2)$ with $\omega, \Omega$ being pre-sieves as in Definition~\ref{def:presieve}. Then, for $\epsilon \gg  e^{-\frac{\log D_1}{10\log P_1}}$, we have
\begin{align*}
	f\approx_\epsilon g.
\end{align*}
The same holds with the roles of $n$ and $n+2$ reversed.
\end{corollary}

If the exceptional zero exists, we expect a different main term for some $m$. To accommodate this, we now prove a variant of Proposition~\ref{prop_presieveasmyp} that includes restrictions on $n$ and $m-n$ into certain residue classes depending on a given primitive quadratic character.

\begin{proposition}\label{prop_fundlem_exceptional}
 For $1\leq i \leq 4$, let the functions $f_i\in \{ \omega, \Omega \}$ be pre-sieves as in Definition~\ref{def:presieve}. Let  $\chi$ be a primitive quadratic character to the modulus $r
 \geq 3$. There exist functions $\sigma_1, \sigma_2\colon \mathbb{N}\to [-1,1]$ (depending on $\chi$) --- they are real since $\chi$ is real-valued --- with $|\sigma_1|,|\sigma_2|\leq 1$ such that for any $1\leq X\leq m$, we have
    \begin{align}
        \nonumber  &\sum_{n\leq X} \1_{\chi(n)=\chi(m-n)=-1}\1_{(n+2,r)=(m-n+2,r)=1} f_1(n)f_2(n+2) f_3(m-n)f_4(m-n+2)\\
		\label{eq_fundex_1}=& X \mathfrak{S}(m)\left(\frac{1+\sigma_1(m)}{4}+O_{\epsilon}(r^{-1+\epsilon})\right)\left(1+ O\left(e^{- \frac{\log D_1}{10 \log P_1}}+\frac{\log m}{P_1}\right)\right)+O(rD_1^4(\log P_1)^4) 
        \end{align}
        and
        \begin{align}
		\nonumber  &\sum_{n\leq X} \1_{\chi(n)=\chi(n+2)=1}\1_{\chi(m-n)=-1} \1_{(m-n+2,r)=1} f_1(n)f_2(n+2) f_3(m-n)f_4(m-n+2)\\
		\label{eq_fundex_2} =& X\mathfrak{S}(m)\left(\frac{1-\sigma_1(m)-\sigma_2(m)}{8}+O_{\epsilon}(r^{-1/2+\epsilon})\right)\bigl(1+ O(e^{- \frac{\log D_1}{10 \log P_1}}+\frac{\log m}{P_1})\bigr)+O(rD_1^4(\log P_1)^4).
	\end{align}
\end{proposition}
\begin{proof}
	Splitting into congruence classes modulo $r$, we have to consider
	\begin{align*}
	   \sum_{b(r)}\mathfrak{a}(b)\sum_{\substack{n\leq X \\ n\equiv b(r)}} f_1(n)f_2(n+2)f_3(m-n)f_4(m-n+2),
	\end{align*}
	where $\mathfrak{a}\in\{\mathfrak{a}_1,\mathfrak{a}_2\}$ with
	\begin{align*}
		\mathfrak{a}_1(b)=\1_{\chi(b)=\chi(m-b)=-1}\1_{(b+2,r)=(m-b+2,r)=1}
	\end{align*}
	and
	\begin{align*}
	   \mathfrak{a}_2(b)=\1_{\chi(b)=\chi(b+2)=1}\1_{\chi(m-b)=-1} \1_{(m-b+2,r)=1}.
	\end{align*}
     We treat the sum over $n$ just as in the proof of Proposition~\ref{prop_presieveasmyp}, the only difference being that fixing $n$ into a residue class modulo $r$ means that primes dividing $r$ are not part of the calculations. Indeed, opening the sieves we now have
	\begin{align*}
		&\sum_{\substack{n\leq X \\ n\equiv b(r)}} f_1(n)f_2(n+2)f_3(m-n)f_4(m-n+2)
		\\=&  \frac{3^4}{\prod_{p<P_1}(1-1/p)^4} \1_{m\equiv 4\pmod 6}  \sum_{d_1,d_2,d_3,d_4\mid \mathcal{P}}  \prod_{i=1}^4\lambda_{i}(d_i) \sum_{\substack{n\leq X\\ n\equiv b(r)\\ d_1\mid n \\ d_2\mid n+2 \\ d_3\mid m-n \\ d_4\mid  m-n+2}}  \1_{n\equiv 5(6)}.
 	\end{align*}
    Apart from the congruence condition modulo $r$, this is just as in~\eqref{eq_sieveopen}. By considering the support of $\mathfrak{a}$, we can assume that $(b,r)=1$ and thus also $(d_1d_2d_3d_4,r)=1$. Furthermore, the arising condition that $b,m-b\equiv 5 \pmod{(6,r)}$ is automatic by the support of $\mathfrak{a}$. From this point, we can follow the argument in the proof of Proposition~\ref{prop_presieveasmyp} verbatim, apart from the additional condition that $(d_1d_2d_3d_4,r)=1$ and the additional factor $1/r$ when approximating the sum over $n$. Explicitly, for $(6,r)=1$ the innermost sum equals $\tfrac{X}{6r[d_1,d_2,d_3,d_4]}+O(1)$, the term $O(1)$ contributing $O(rD_1^4(\log P_1)^4)$ after summation over the $d_i$ and over $b\pmod r$. This shows that
	\begin{align*}   
		&\sum_{b(r)}\mathfrak{a}(b)\sum_{\substack{n\leq X \\ n\equiv b(r)}} f_1(n)f_2(n+2)f_3(m-n)f_4(m-n+2)\\
		&=\frac{\sum_{b(r)}\mathfrak{a}(b)}{r}X \,\prod_{p\mid r}\Bigl(1-\tfrac1p\Bigr)^{-4}\mathfrak{S}'(m,r)\left(1+O\Bigl(e^{- \frac{\log D_1}{10 \log P_1}}+\frac{\log m}{P_1}\Bigr) \right) +O(rD_1^4(\log P_1)^4),
	\end{align*}
	where
	\begin{align*}
        \mathfrak{S}'(m,r)=\1_{m\equiv 4\pmod 6}\frac{27}{2}\prod_{\substack{p\geq 5\\ p\nmid r}}\left(1-\frac{6p^2-4p+1}{(p-1)^4}\right)\prod_{\substack{p\geq 5\\ p\mid m(m+4)\\ p\nmid r}}\left(1+\frac{1}{p-4}\right)\prod_{\substack{p\geq 5\\ p\mid m+2\\ p\nmid r}}\left(1+\frac{2}{p-4}\right).
	\end{align*}
	It remains to consider the sum over $b$. We have $\1_{\chi(b)=-1}=\1_{(b,r)=1}(1-\chi(b))/2$ and so in the case of $\mathfrak{a}=\mathfrak{a}_1$
	\begin{align*}
		4\sum_{b(r)}\mathfrak{a}_1(b)=&\sum_{b(r)}\1_{(b,r)=(b+2,r)=(m-b,r)=(m-b+2,r)=1}\\
		&-2 \sum_{b(r)}\chi(b)\1_{(b(b+2)(m-b)(m-b+2),r)=1}\\
		&+\sum_{b(r)}\chi(b)\chi(m-b)\1_{((b+2)(m-b+2),r)=1}\\
		\eqqcolon &S_1(m,r)-2S_2(m,r)+S_3(m,r),
	\end{align*}
    say. Naturally all the $S_i$ are functions of both $m$ and $r$. By the nature of the singular series being a product of local solution densities (or a short direct verification), we see that 
	\begin{align*}
		\frac{S_1(m,r)}{r}\prod_{p\mid r}\Bigl(1-\tfrac1p\Bigr)^{-4}=\frac{\mathfrak{S}(m)}{\mathfrak{S}'(m,r)}.
	\end{align*}
	We set
	\begin{align*}
		\sigma_1(m)\coloneqq\frac{S_3(m,r)}{S_1(m,r)}.
	\end{align*}
	An application of the triangle inequality immediately shows $|\sigma_1|\leq 1$. The sum $S_2$ is multiplicative in $r$. Indeed, for coprime $r=r_1r_2$ the Chinese Remainder Theorem and $\chi=\chi_{r_1}\chi_{r_2}$ factor the sum as $S_2(m,r_1)S_2(m,r_2)$. Using a simple inclusion-exclusion argument and orthogonality of characters gives $|S_2(m,r)|\ll r^{o(1)}.$ This shows~\eqref{eq_fundex_1}.
	
    For the case $\mathfrak{a}=\mathfrak{a}_2$ we observe $\1_{\chi(b)=\chi(b+2)=1}=\1_{(b(b+2),r)=1}(1+\chi(b))(1+\chi(b+2))/4$ so that
	\begin{align*}
		8\sum_{b(r)}\mathfrak{a}_2(b)=S_1-S_3-S_4+S_5-S_6{}+O_\epsilon\bigl(r^{1/2+\epsilon}\bigr),
	\end{align*}
	where $S_1$ and $S_3$ are as above and
	\begin{align*}
		S_4&= \sum_{b(r)}\chi(b+2)\chi(m-b)\1_{(b(m-b+2),r)=1},\\
		S_5&=\sum_{b(r)}\chi(b+2)\1_{(b(m-b)(m-b+2),r)=1},\\
		S_6&=\sum_{b(r)}\chi(b)\chi(b+2)\chi(m-b)\1_{(m-b+2,r)=1}.
	\end{align*}
	Here the error term $O_\epsilon(r^{1/2+\epsilon})$ collects the three single- and double-character sums $\sum_{b(r)}\chi(b)$, $\sum_{b(r)}\chi(b)\chi(b+2)$ and $\sum_{b(r)}\chi(m-b)$ arising in the expansion, each $\ll_\epsilon r^{1/2+\epsilon}$ as for $S_5,S_6$.
	We set
	\begin{align*}
		\sigma_2(m)\coloneqq\frac{S_4(m,r)}{S_1(m,r)}.
	\end{align*}
	The triangle inequality again shows $|\sigma_2(m)|\leq 1$. We have $|S_5(m,r)|\ll r^{o(1)}$ by the same argument as for $S_2$. Finally, since  $\chi$ is primitive and quadratic we have that $r/(4,r)$ is square-free. Thus, using that $S_6$ is multiplicative in $r$, a simple inclusion-exclusion argument, and the Weil bound~\cite[Corollary 11.24]{iw-kow}, we get $|S_6(m,r)|\ll_{\epsilon} r^{1/2+\epsilon}$. This shows~\eqref{eq_fundex_2} and completes the proof.
\end{proof}
	
\begin{remark}
	It is not difficult to calculate $\sigma_1(m)$ and $\sigma_2(m)$ and write them as products of local densities. However, when compared to the related case of $\widetilde{\mathfrak{S}}(m)$ in~\cite{mv}, the treatment is more involved. In particular the primes $3$, $5$ and $7$ need to be handled separately. To simply show the existence of suitable representations of $m$ as in Theorem~\ref{MT1}, the precise evaluation of the functions $\sigma_i$ plays no role. It matters only that both $|\sigma_1|$ and $|\sigma_2|$ are $1$-bounded and that $\sigma_1$ appears with different signs in~\eqref{eq_fundex_1} and~\eqref{eq_fundex_2}.
\end{remark}
	
\subsection{Upper bounds for the sifted sum.}\label{subsec:6.3}

In order to prove upper bounds of the correct order of magnitude for the sifted additive problem~\eqref{eq_siftedsums_main} involving beta sieves and $\Lambda_{R,r}$ as in~\eqref{eq:LRrdef}, we require the following lemma about correlations.

\begin{lemma}[An upper bound for correlations of multiplicative functions]\label{le:henriot}
    Let $A\geq 1$ be fixed. Let $h:\mathbb{N}\to \mathbb{R}_{\geq 0}$ be a multiplicative function that satisfies $h(n)\leq n^{o(1)}$ and $h(p^k)\leq A^k$ for all primes $p$ and all $k\in \mathbb{N}$. Let $\log N\leq P'_1,P'_2,P'_3 \leq N$. Then, for any $m\in [5N/4,7N/4]$ and $m \equiv 4\pmod 6$, we have
	\begin{align}
	\label{eq:henriotapp}	&\sum_{N/2< n\leq N}h(n)(\tau^2r_{P'_1})(n+2)(\tau^2r_{P'_2})(m-n)(\tau^2r_{P'_3})(m-n+2)\\
		\ll& \frac{N}{\log N}\mathfrak{S}(m) \prod_{i=1}^3 \nonumber \left(\frac{\log N}{\log P'_i}\right)^{3} \prod_{p<N}\left(1+\frac{h(p)}{p}\right)\prod_{\substack{p\mid m (m+2)(m+4)\\ p\geq \min\{P'_1,P'_2,P'_3\}}}\Bigl(1+\frac{O(1)}{p}\Bigr).
	\end{align}
    The same bound holds if $h(n)$ swaps places with any of the other weights on $n+2,m-n,m-n+2$.
\end{lemma}

\begin{proof} 
    We want to apply Henriot's bound (\cite[Corollary 1]{henriot}) to the four-variable multiplicative function 
    \begin{align*}
        F(n_1,n_2,n_3,n_4)=h(n_1) \prod_{i=1}^3 \tau^2(n_{i+1}) \1_{(n_{i+1},\mathcal{P}(P'_i))=1},
    \end{align*}
    the tuple of polynomials 
    \begin{align*}
        (Q_1(u),Q_2(u),Q_3(u),Q_4(u))=(u,u+2,m-u,m-u+2),
    \end{align*}
    and with choices $x=y=N/2$. Note that $Q\coloneqq Q_1Q_2Q_3Q_4$ has discriminant $\mathrm{D}=2^4m^2(m+2)^4(m+4)^2$, that $Q$ is primitive since its leading coefficient is $1$, and that the sum of the absolute values of the coefficients of $Q$ is $\ll N^2$.
    
    Moreover, if $\rho_{Q}(n)$ denotes the number of solutions to the congruence $Q(u)\equiv 0\pmod n$, then since $m\asymp N$, we have
    \begin{align*}
        \prod_{4<p\leq N}\left(1-\frac{\rho_{Q}(p)}{p}\right)=&\prod_{\substack{4<p\leq  N\\p\nmid \mathrm{D}}}\left(1-\frac{4}{p}\right)\prod_{\substack{4<p\leq  N\\p\mid m(m+4)}}\left(1-\frac{3}{p}\right)\prod_{\substack{4<p\leq  N\\p\mid m+2}}\left(1-\frac{2}{p}\right)\\
        &\ll \prod_{\substack{4<p\leq  N}}\left(1-\frac{4}{p}\right) \prod_{p\mid m(m+4)}\left(1+\frac{1}{p}\right)\prod_{p\mid m+2}\left(1+\frac{2}{p}\right)\\
        &\ll \frac{\mathfrak{S}(m)}{(\log N)^4}.
    \end{align*}
    We can now apply~\cite[Corollary 1]{henriot}, observing that in Henriot's notation we have $\widetilde{F}=\widetilde{G}=F$, and that we can simplify the main term there with the trivial inequality
    \begin{align*}
        \sum_{n_1n_2n_3n_4\leq N}a_{n_1,1}a_{n_2,2}a_{n_3,3}a_{n_4,4}\leq \prod_{j=1}^4 \sum_{n\leq N}a_{n,j},
    \end{align*} 
    valid for $a_{n,j}\geq 0$.
    Then, recalling Henriot's notation of $\Delta_D$ given in~\cite[eq. (1.3)]{henriot}, and our normalisation of $r_{P'_i}$, we can estimate the left-hand side of~\eqref{eq:henriotapp} by
	\begin{align*}
		\ll& \Delta_\mathrm{D} N \mathfrak{S}(m) \frac{1}{(\log N)^4}\prod_{i=1}^3\prod_{p\leq P'_i}\left(1-\frac{1}{p}\right)^{-1} \sum_{\substack{n\leq  N\\ (n,\mathrm{D})=1}}\frac{h(n)}{n}\prod_{i=1}^3 \sum_{n\leq  N}\frac{\tau^2(n) \1_{(n,\mathcal{P}(P'_i))=1}}{n}.
        \end{align*}
    We upper bound each of the four sums over $n$ by products
    \begin{align*}
        \ll& \Delta_\mathrm{D} N \mathfrak{S}(m) \frac{1}{(\log N)^4} \prod_{\substack{p\leq N\\ p\nmid \mathrm{D}}}\left(1+\frac{h(p)}{p}\right) \prod_{i=1}^3 \left( \prod_{p\leq P'_i}\left(1-\frac{1}{p}\right)^{-1}  \prod_{P'_i<p\leq N} \left(1+\frac{4}{p}\right) \right)\\  
        \ll& \Delta_\mathrm{D} N \mathfrak{S}(m) \frac{1}{(\log N)^4} \prod_{\substack{p\leq N\\ p\nmid \mathrm{D}}} \left(1+\frac{h(p)}{p}\right) \prod_{i=1}^3  \left(\frac{\log P'_i}{\log 2} \left(\frac{\log N}{\log P'_i}\right)^4 \right).
    \end{align*}
    By Mertens' theorem, this is bounded by
    \begin{align}\label{eq:henriotappli}
        	\ll& \Delta_\mathrm{D} N \mathfrak{S}(m)  \frac{1}{\log  N}  \prod_{\substack{p\leq N\\ p\nmid \mathrm{D}}} \left(1+\frac{h(p)}{p}\right) \prod_{i=1}^3 \left(\frac{\log N}{\log P'_i}\right)^3.
    \end{align}
        It remains to consider the $\Delta_\mathrm{D}$ factor. We have
	\begin{align*}
		\Delta_\mathrm{D}&=\prod_{p\mid \mathrm{D}}\left(1+\frac{1}{p}\sum_{\substack{\kappa_1,\kappa_2,\kappa_3,\kappa_4\in \{0,1\}\\ (\kappa_1,\kappa_2,\kappa_3,\kappa_4)\neq (0,0,0,0)}}h(p^{\kappa_1})\prod_{i=1}^3 \tau^2(p^{\kappa_{i+1}})\1_{p\geq P'_i}\right)\\
		&\ll \prod_{p\mid \mathrm{D}}\left(1+\frac{h(p)}{p}\right)\prod_{\substack{p\mid \mathrm{D}\\p\geq \min\{P_1',P_2',P_3'\}}}\Bigl(1+\frac{O(1)}{p}\Bigr).
	\end{align*}
    Plugging this into \eqref{eq:henriotappli} we are done, since we can assume that $h(p)=0$ for $p>N$.
\end{proof}

In the next proposition, we show that this upper bound for sums of multiplicative functions, when combined with Lemma~\ref{lem_fundlem}, allows us to estimate correlations of sieves and the function $H_R$ with absolute values. In particular, this shows that the absolute value of a lower bound sieve is still sieve-like in our additive setup.

\begin{proposition}[Correlations of absolute values of sieve weights]\label{prop:LambdaQ}
    Let $m\geq 10$ be an integer, and let $\varepsilon>1/\log \log m$ and $N^\varepsilon \leq R, P_1\leq N$. For $1\leq i \leq 3$, let the functions $f_i\in \{ \omega, \Omega \}$ be pre-sieves as in Definition~\ref{def:presieve}, and let $P_1^{500}\leq D_1\leq N^{1/100}$. Let $H_R$ be as in Theorem~\ref{thm:F-P-C_intro}. Then for any $m\in [5N/4,7N/4]$ and $m \equiv 4\pmod 6$, we have
	\begin{align*}
		\sum_{N/2<n\leq N}H_R(n) |f_
		1(n+2)||f_2(m-n)||f_3(m-n+2)|\ll \varepsilon^{-10}N\mathfrak{S}(m).
    \end{align*}
    The same bound holds if $H_R(n)$ swaps places with any of the other weights on $n+2,m-n,m-n+2$.
\end{proposition}

\begin{proof}

We start by relating all the summands to multiplicative functions. We recall that 
    \begin{align}
       H_R(n)= (\log R) \tau(n)\int_{\R}\frac{h_{\xi}(n)}{(1+|\xi|)^{10}} \d \xi,
    \end{align}
where
 \begin{align*}
        h_{\xi}(p)=\min\Bigl\{1,10(1+|\xi|)\frac{\log p}{\log R} \Bigr\}.
    \end{align*}
    We next apply Lemma~\ref{lem_fundlem} and recall the notation therein (in particular, $\mathcal{P}_r=\mathcal{P}(P_1^{\bigl(\frac{\beta-1}{\beta+1}\bigr)^{r}})$). Since $|f_i(n)|\le r_{P_1}(n)+|f_i(n)-r_{P_1}(n)|$ and $r_{P_1}(n)=\prod_{p<P_1}\bigl(1-\frac{1}{p}\bigr)^{-1}\1_{(n,\mathcal{P}_0)=1}$ is the $r=0$ term of the sum below, we get the upper bound
    \begin{align*}
         |f_i(n)|\ll \prod_{p<P_1}\left(1-\frac{1}{p}\right)^{-1} \tau(n)^2 \sum_{r\geq 0}4^{-r}\1_{(n,\mathcal{P}_r)=1}.
    \end{align*}
     Write $a(r)\coloneqq((\beta-1)/(\beta+1))^r$ with $\beta=200$. Then by Mertens' theorem, 
     \begin{align*}
         \prod_{p< P_1}\left(1-\frac{1}{p}\right)^{-1} \1_{(n,\mathcal{P}_r)=1} = &\prod_{P_1^{a(r)}\leq p < P_1} \left(1-\frac{1}{p}\right)^{-1}  r_{P_1^{a(r)}}(n)\\
         \ll & a(r)^{-1} r_{P_1^{a(r)}}(n).
     \end{align*}
     Thus, 
    \begin{align}\label{eq:LQ1}
        \nonumber &\sum_{N/2<n\leq N}H_R(n) |f_
		1(n+2)||f_2(m-n)||f_3(m-n+2)|\\
        \ll &(\log R)  \int_{\R}\frac{1}{(1+|\xi|)^{10}} \sum_{r_1,r_2,r_3 \geq 0} 4^{-r_1-r_2-r_3}a(r_1)^{-1}a(r_2)^{-1}a(r_3)^{-1} \mathcal{H}(\xi,r_1,r_2,r_3) \d \xi, 
    \end{align}
    where
    \begin{align*}
          \mathcal{H}(\xi,r_1,r_2,r_3)=\sum_{N/2<n\leq N} \tau(n)h_{\xi}(n) (\tau^2 r_{P_1^{a(r_1)}})(n+2)(\tau^2 r_{P_1^{a(r_2)}})(m-n)(\tau^2 r_{P_1^{a(r_3)}})(m-n+2).
    \end{align*}
    This is amenable for an application of Lemma~\ref{le:henriot} (with $h(p)=2h_{\xi}(p)$). To prepare this application, we estimate
\begin{align*}\begin{split}
	   \left|\prod_{p<N}\left(1+\frac{2 h_{\xi}(p)}{p}\right)\right|\leq& \exp\left(\sum_{p\leq N}\frac{2 h_{\xi}(p)}{p}\right)\\
	   \ll & \exp\Bigl(\sum_{p\leq R^{1/(1+|\xi|)}}\frac{20(1+|\xi|)\log p}{p\log R}+\sum_{R^{1/(1+|\xi|)}<p\leq N}\frac{2}{p}\Bigr)\\
	   \ll& \Bigl(\frac{\log N}{\log R}\Bigr)^2(1+|\xi|)^2,
	   \end{split}
	\end{align*}
    by a trivial bound for $|\xi|>10/\log R$ and Mertens' theorem in the complementary range. Consequently, an application of Lemma~\ref{le:henriot} gives
    \begin{align*}
        &\mathcal{H}(\xi,r_1,r_2,r_3)\\
        \ll&  \Bigl(\frac{\log N}{\log R}\Bigr)^2(1+|\xi|)^2\frac{N}{\log N}\mathfrak{S}(m) \prod_{i=1}^3 \left(\frac{\log N}{\log P_1^{a(r_i)}}\right)^{3} \prod_{\substack{p\mid m (m+2)(m+4)\\ p\geq \min\{P_1^{a(r_1)},P_1^{a(r_2)},P_1^{a(r_3)}\}}}\Bigl(1+\frac{O(1)}{p}\Bigr).
    \end{align*}
    Plugging this into~\eqref{eq:LQ1}, the integral over $\xi$ converges and we get the upper bound
    \begin{align*}
        \ll N \mathfrak{S}(m)\frac{\log N}{\log R} \Bigl(\frac{\log N}{\log P_1}\Bigr)^9 \left(\sum_{r\geq 0}4^{-r}a(r)^{-4}\prod_{\substack{p\mid m (m+2)(m+4)\\ p\geq P_1^{a(r)}}}\Bigl(1+\frac{O(1)}{p}\Bigr) \right)^3.
    \end{align*}
    Since $R,P_1\geq N^{\varepsilon}$, we have $\frac{\log N}{\log R} \Bigl(\frac{\log N}{\log P_1}\Bigr)^9\leq \varepsilon^{-10}$,  so that it only remains to show that the sum over $r$ converges. 
    If $P_1^{a(r)}>\log m$, then
    \begin{align*}
        \prod_{\substack{p\mid m (m+2)(m+4)\\ p\geq P_1^{a(r)}}}\Bigl(1+\frac{O(1)}{p}\Bigr)\ll 1.
    \end{align*}
    On the other hand, we always have 
    \begin{align*}
        \prod_{\substack{p\mid m (m+2)(m+4)\\ p\geq P_1^{a(r)}}}\Bigl(1+\frac{O(1)}{p}\Bigr)\ll \log \log m. 
    \end{align*}
    Also, recalling $\beta=200$, we have
    \begin{align*}
         a(1)^{-4}/4= \left(\frac{\beta+1}{\beta-1}\right)^4/4<1/3.
    \end{align*}
    Thus,
    \begin{align*}
        \sum_{r\geq 0}4^{-r}a(r)^{-4}\prod_{\substack{p\mid m (m+2)(m+4)\\ p\geq P_1^{a(r)}}}\Bigl(1+\frac{O(1)}{p}\Bigr) \leq&  \sum_{r}3^{-r}+(\log \log m)\sum_{\substack{r\geq 0\\ P_1^{a(r)}\leq \log m}}3^{-r}.
    \end{align*}
    If $ P_1^{a(r)}\leq \log m$, then  $r\gg \log \frac{\log P_1}{\log \log m}$,
    so that
    \begin{align*}
        (\log \log m)\sum_{\substack{r\\ P_1^{a(r)}\leq \log m}}3^{-r}\ll \frac{(\log \log m)^2}{\log P_1}\ll  \frac{(\log \log m)^3}{\log  m}\ll 1,
    \end{align*}
    since $P_1\geq N^{\varepsilon}\gg m^{\varepsilon}$ and $\varepsilon\geq (\log \log m)^{-1}$. 
\end{proof}

\section{From primes to Cram\'er's model}\label{sec:PrimestoRough}
	
In this section, we combine Proposition~\ref{prop:LambdaQ} with variants of Gallagher's prime number theorem to show results of the type $\Lambda(n)\omega(n+2)\approx r_{P_0}(n)\omega(n+2) $. 

Since the results depend on the possible existence of an exceptional zero, we first define it and import some well known results about it in Subsection~\ref{subsec:7.1}. Afterwards, in Subsection~\ref{subsec:7.2} we import or show variants of Gallagher's prime number theorem. We apply those in Subsection~\ref{subsec:7.3} to show that all of $\Lambda, \Lambda_{E_3^*}$ (if taking into account the possible exceptional zeros) and the Cram\'er model $r_P$ are Fourier-close to Heath-Brown's model $\Lambda_{R,1}$. These statements are combined in Subsection~\ref{subsec:7.4} to prove Theorem~\ref{thm:gallagher_application}, which contains Theorem~\ref{thm:F-P-C_intro}. From it, we can quickly deduce Propositions~\ref{prop_reducetorough} and~\ref{prop_reducetoroughexc}, which are the key ingredients for proving Theorem~\ref{MT1}.

\subsection{Exceptional zeros}\label{subsec:7.1}

We now define some necessary terminology involving the exceptional character and zero.
	\begin{definition}[Exceptional zero]\label{def:exceptional} We say that a real number $\widetilde{\beta}$ is an \emph{exceptional zero of level $P\geq 3$ and quality $\kappa \in (0,1)$} if there exists a primitive Dirichlet character $\widetilde{\chi}$ of some modulus $\widetilde{r}\leq P$ such that $L(\widetilde{\beta},\widetilde{\chi})=0$ and
 \begin{align}\label{eq:1-clogP}
  \widetilde{\beta} \geq 1-\frac{\kappa}{\log P}.   
 \end{align}
 The character $\widetilde{\chi}$ is called an \emph{exceptional character} and $\widetilde{r}$ is called the exceptional modulus.
	\end{definition}

 By the Landau--Page theorem~\cite[Corollary 11.10]{MV-book}, if there is an exceptional zero of level $P$ and quality $\kappa$, and $P$ is large enough in terms of $\kappa$, then the zero is unique, simple and there is a unique exceptional character, which is a quadratic character. We further have the well known bound (see~\cite[equation (4.1)]{mv})
 \begin{align}\label{eq:exczerolowerbound}
     1-\widetilde{\beta}\gg \frac{1}{\widetilde{r}^{1/2} (\log \widetilde{r})^2}.
 \end{align}
In particular, putting~\eqref{eq:1-clogP} and~\eqref{eq:exczerolowerbound} together we obtain  
\begin{align}\label{eq:rlowerbound}
\widetilde{r}\gg \left(\frac{\log P}{\kappa \log \log P}\right)^2.
\end{align}

    We need to use the fact that we can choose the quality parameter such that there is either no exceptional zero or its modulus is small. The procedure for this is standard; for our purposes, the most suitable is~\cite[Proposition 7.3.]{green-sarkozy}, which we rephrase as follows. 

	\begin{lemma}\label{lem_excmodulussize} Let $R_0=N^{\delta_1^3}$ and $\widetilde{R}=N^{2\delta_1^5}$.
    There exists an absolute constant $\lambda >0$ such that for every $\delta>0$ the global parameter $\delta_1$ may be chosen with $\delta_1\leq\delta$ so that one of the following holds.
   \begin{enumerate}
       \item There are no exceptional zeros of quality $\lambda  \delta_1^2$ and level $R_0^2$.
       \item If $\widetilde{r}$ is the modulus of an exceptional zero of quality $\lambda  \delta_1^2$ and level $R_0^2$, then $\widetilde{r}\leq \widetilde{R}$.  
   \end{enumerate}    

	\end{lemma}

\subsection{Gallagher-type estimates}\label{subsec:7.2}

 Gallagher~\cite{Gal} proved a flexible version of the prime number theorem in arithmetic progressions that allows moduli up to a power of the summation range, at the cost of giving only a weak saving. We now state his result and prove three variants involving first the indicator of the primes without logarithmic weight, then $\Lambda_{E_3^*}$, and finally $r_P(n)$.

\begin{lemma}[Gallagher]\label{le:gallagher}
    Let $2\leq \exp((\log N)^{1/2}) \leq R \leq N$. Let  $0<\kappa<1$. 
    \begin{enumerate}
		\item If there is no exceptional zero of level $R$ and quality $\kappa$, then we have
		  \begin{align*}
		  \sum_{ r\leq R}\,\,\,\asum_{\chi\pmod r} \max_{\substack{I\subset [1,N]\\I \textnormal{ interval}}}\frac{1}{|I|+N/R}\left|\sum_{n\in I}(\Lambda(n)\chi(n)-\1_{r=1})\right|\ll \exp\left(-c \kappa \frac{\log N}{\log R}\right).
	       \end{align*}
		\item If there is an exceptional zero $\widetilde{\beta}$ of level $R$, then if $\widetilde{\chi}$ is the exceptional character, we have
            \begin{align*}
			    &\sum_{ r\leq R}\,\,\,\asum_{\chi\pmod r}\max_{\substack{I\subset [1,N]\\I \textnormal{ interval}}}\frac{1}{|I|+N/R} \left|\sum_{n\in I}(\Lambda(n)\chi(n)-\1_{r=1}+\1_{\chi=\widetilde{\chi}}n^{\widetilde{\beta}-1})\right|\\
                \ll& (1-\widetilde{\beta})(\log N)\exp\left(-c \frac{\log N}{\log R}\right).
		    \end{align*}
	   \end{enumerate}
\end{lemma}
	
\begin{proof}
    This follows from Gallagher's work~\cite{Gal} after noting that the contribution of prime powers to the von Mangoldt function is trivial and keeping track of the dependence on the quality of the zero explicitly.
\end{proof}

We now remove the logarithmic weight of the von Mangoldt function in Lemma~\ref{le:gallagher} with summation by parts.

\begin{lemma}[Gallagher for prime indicator]\label{le:gallagherprime}
 Let $2\leq \exp((\log N)^{1/2}) \leq R \leq N^{1/7}$ and $0<\kappa<1$. We have
    \begin{enumerate}
		\item  If there is no exceptional zero of level $R$ and quality $\kappa$, then we have
		  \begin{align*}
		  \sum_{ r\leq R}\,\,\,\asum_{\chi\pmod r} \max_{\substack{I\subset [2,N]\\I \textnormal{ interval}}}\frac{\log N}{|I|+N/R}\left|\sum_{n\in I}\left(\1_{n\in \mathbb{P}}\chi(n)-\frac{\1_{r=1}}{\log n}\right)\right|\ll \exp\left(-c\kappa \frac{\log N}{\log R}\right).
	       \end{align*}
		\item  If there is an exceptional zero $\widetilde{\beta}$ of level $R$ and quality $\kappa$, then if $\widetilde{\chi}$ is the exceptional character, we have
            \begin{align*}
			    &\sum_{ r\leq R}\,\,\,\asum_{\chi\pmod r}\max_{\substack{I\subset [2,N]\\I \textnormal{ interval}}}\frac{\log N}{|I|+N/R} \left|\sum_{n\in I}\left(\1_{n\in \mathbb{P}}\chi(n)-\frac{\1_{r=1}}{\log n}+\frac{\1_{\chi=\widetilde{\chi}}n^{\widetilde{\beta}-1}}{\log n}\right)\right|\\
                \ll& (1-\widetilde{\beta})(\log N)\exp\left(-c \frac{\log N}{\log R}\right).
		    \end{align*}
	   \end{enumerate}
\end{lemma}

\begin{proof}
Write 
\begin{align*}
    a_n=\1_{n\in I}\left(\1_{n\in \mathbb{P}}\chi(n)-\frac{\1_{r=1}}{\log n}\right)
\end{align*}
in the unexceptional case and 
\begin{align*}
    a_n=\1_{n\in I}\left(\1_{n\in \mathbb{P}}\chi(n)-\frac{\1_{r=1}}{\log n}+\frac{\1_{\chi=\widetilde{\chi}}n^{\widetilde{\beta}-1}}{\log n}\right)
\end{align*}
in the exceptional case. The contribution of $n\leq \sqrt{N}$ is at most $R^3 N^{-1/2} \log N$. By $\exp((\log N)^{1/2}) \leq R \leq N^{1/7}$ this is $\ll N^{-1/14}\log N$, and hence small enough, using additionally~\eqref{eq:exczerolowerbound} in the exceptional case.
In the remaining range $n>\sqrt{N}$, we apply Lemma~\ref{lem:partsum}. Since $I\cap [\sqrt{N},\xi]$ is again an interval, the stated estimate follows from Lemma~\ref{le:gallagher}, after handling prime powers trivially.
\end{proof}

We now show a version of Lemma~\ref{le:gallagher} for the function $\Lambda_{E_3^*}$ (see Definition~\ref{def:lambdaE3}).
\begin{lemma}\label{le:gallagherE3}
   Let $2\leq \exp((\log N)^{4/5}) \leq R \leq N^{1/50}$ and $0<\kappa<1$. 
	\begin{enumerate}
		\item  If there is no exceptional zero of level $R$ and quality $\kappa$, then we have
		  \begin{align*}
			\sum_{1\leq r\leq R}\,\,\,\asum_{\chi\pmod r} \max_{\substack{I\subset [1,N]\\I \textnormal{ interval}}}\frac{1}{|I|+N/R}\left|\sum_{n\in I}(\Lambda_{E_3^{*}}(n)\chi(n)-\1_{r=1})\right|\ll \exp\left(-c \kappa \frac{\log N}{\log R}\right).
		  \end{align*}
			
		  \item  If there is an exceptional zero $\widetilde{\beta}$ of level $R$, then if $\widetilde{\chi}$ is the exceptional character, we have
		      \begin{align}\label{eq:Gallagher2}\begin{split}
				&\sum_{1\leq r\leq R}\,\,\,\asum_{\chi\pmod r}\max_{\substack{I\subset [1,N]\\I \textnormal{ interval}}}\frac{1}{|I|+N/R} \Bigg|\sum_{n\in I}\left(\Lambda_{E_3^{*}}(n)\chi(n)-\1_{r=1}+\1_{\chi=\widetilde{\chi}}n^{\widetilde{\beta}-1}\right)\Bigg|\\
				&\ll  (1-\widetilde{\beta})(\log N)\exp\left(-c \frac{\log N}{\log R}\right).
				\end{split}
			\end{align} 
	\end{enumerate}
\end{lemma}

\begin{proof}
Let us prove part (2); part (1) is similar but somewhat easier. Similar to the proof of the previous Lemma, we apply Lemma~\ref{lem:partsum}, but here to multiply by $c_{E_3^*}(n)$. Every summand is $O(1)$, so for each pair $(r,\chi)$ the terms with $n<N/R^4$ contribute at most $O\bigl((\log N)(N/R^4)/(N/R)\bigr)=O(R^{-3}\log N)$, and $O(R^{-1}\log N)$ in total. Since $\widetilde{r}\leq R$ and $\log R\geq (\log N)^{4/5}$, this is $O\bigl((1-\widetilde{\beta})(\log N)\exp(-c\log N/\log R)\bigr)$ by~\eqref{eq:exczerolowerbound}, once $c$ is small enough. We may therefore restrict to $I\subset [N/R^4,N]$, where~\eqref{eq:c_E} applies. Recalling Definition~\ref{def:lambdaE3}, it then suffices to show for $j\in \{1,2\}$ the bound
\begin{align*} 
    &\sum_{1\leq r\leq R}\,\,\,\asum_{\chi\pmod r}\max_{\substack{I\subset [N/R^4,N]\\I \textnormal{ interval}}}\frac{\log N}{|I|+N/R} \Bigg|\sum_{n\in I}\left(\1_{n\in B_j}\chi(n)-c_{B_j}(n)\1_{r=1}+c_{B_j}(n)\1_{\chi=\widetilde{\chi}}n^{\widetilde{\beta}-1}\right)\Bigg|\\
				&\ll  (1-\widetilde{\beta})(\log N)\exp\left(-c \frac{\log N}{\log R}\right).
\end{align*}
Let $\beta\in[1/2,1]$. Note that for $(t_1,t_2,t_3)\in [K_1,K_1+1)\times [K_2,K_2+1)\times [K_3,K_3+1)$ with $K_i\in [N^{1/10}, N]$ we have 
  \begin{align*}
  \frac{t_3^{\beta-1}}{\log t_3}=&\frac{K_3^{\beta-1}}{\log K_3}+O(N^{-1/10})  \\
   \frac{(t_1t_2t_3)^{\beta-1}}{(\log t_1)(\log t_2)(\log t_3)}=&\frac{(K_1K_2K_3)^{\beta-1}}{(\log K_1)(\log K_2)(\log K_3)}+O(N^{-1/10})  .
  \end{align*}
We use this to swap between integration and summation and get that for any $I\subset [N/R^4,N]$ we have
\begin{align*}
    \sum_{n\in I} c_{B_1}(n) n^{\beta-1}=& \sum_{n\in I} n^{\beta-1} \int_{\substack{1/10\leq t_1\leq 1/3-\delta_1\leq t_2\leq (1-t_1)/2\\ \frac{\log n}{\log N}-t_1-t_2\geq 1/10}}\frac{\d t_1 \d t_2}{t_1t_2\log (n/N^{t_1+t_2})}\\
    =& \int_{t\in I} \int_{\substack{1/10\leq t_1\leq 1/3-\delta_1\leq t_2\leq (1-t_1)/2\\ \frac{\log t}{\log N}-t_1-t_2\geq 1/10}}\frac{t^{\beta-1} \d t_1 \d t_2 \d t}{t_1t_2\log (t/N^{t_1+t_2})}+O(|I|N^{-1/9})\\
    =& \sum_{\substack{n_1 n_2 n_3\in I\\ 1/10\leq \frac{\log n_1}{\log N}\leq 1/3-\delta_1\leq \frac{\log n_2}{\log N}\leq (1-\frac{\log n_1}{\log N})/2 \\ \frac{\log n_3}{\log N}\geq 1/10}} \frac{(n_1n_2n_3)^{\beta-1}}{(\log n_1)(\log n_2)(\log n_3)}+O(|I|N^{-1/9}),
\end{align*}
and similarly for $B_2$. Let $S$ be a polygonal subset of $(1/10,1)^3$ and for brevity denote $\mathcal{L}(n_1,n_2,n_3)=((\log n_1)/(\log N),(\log n_2)/(\log N),(\log n_3)/(\log N))$.  It now suffices to show
\begin{align*}
&\sum_{1\leq r\leq R}\,\,\asum_{\chi\pmod r}\max_{\substack{I\subset [N/R^4,N]\\I \textnormal{ interval}}}\frac{\log N}{|I|+N/R} \Bigg|\sum_{\substack{n_1n_2n_3\in I\\ \mathcal{L}(n_1,n_2,n_3)\in S}}\Bigg(\1_{\mathbb{P}}(n_1)\1_{\mathbb{P}}(n_2)\1_{\mathbb{P}}(n_3)\chi(n_1n_2n_3)\\
&\qquad \qquad \qquad \qquad \qquad \qquad \qquad \qquad  -\frac{\1_{r=1}-\1_{\chi=\widetilde{\chi}}(n_1n_2n_3)^{\widetilde{\beta}-1}}{(\log n_1)(\log n_2)(\log n_3)}\Bigg)\Bigg|\\
&\ll (1-\widetilde{\beta})(\log N)\exp\left(-c \frac{\log N}{\log R}\right).
\end{align*}
When $r=1$ or $\chi=\widetilde{\chi}$, we get the required result, by successively replacing in the sums over $n_i$ the prime indicator (and possibly character) by their expected approximation with Lemma~\ref{le:gallagherprime}, noting that for each fixed $n_1, n_2$ the condition $n_3\in I/(p_1p_2), \mathcal{L}(n_1,n_2,n_3)\in S$ defines a union of intervals. For the remaining characters ($r>1$, $\chi\ne\widetilde{\chi}$) we bound two of the three prime indicators trivially by $1$ and apply Lemma~\ref{le:gallagherprime} to the sum over the third variable, whose range is again a union of intervals. The resulting sum over $r\le R$, $\chi\bmod r$ is directly of the form estimated there, giving the claim by the triangle inequality. Here the lower bound $\exp((\log N)^{4/5})$ is any convenient value exceeding $\exp((\log N)^{1/2})$, ensuring Lemma~\ref{le:gallagherprime} applies on each of the sub-intervals arising above.
\end{proof}

The last Gallagher-type result we need is for $P$-rough numbers. If $P$ is a sufficiently small power of $N$, the exceptional zero plays no role here.

\begin{lemma}\label{le:gallagher_rough}
	Let $N^{1/5}\geq P\geq  R^{10} \geq \exp((\log N)^{1/2})\geq 2$.  There exists an absolute constant $c>0$ such that
	\begin{align}\label{eq:Gallagher5}
		\sum_{2\leq r\leq R}\,\,\,\asum_{\chi\pmod r} \max_{\substack{I\subset [1,N]\\I \textnormal{ interval} \\ }}\frac{1}{|I|+N/R}\left|\sum_{n\in I}r_P(n)\chi(n)\right|\ll \frac{\log N}{\log P}\exp\left(-c\frac{\log P }{\log R}\right)+\exp\left(-c\frac{\log N }{\log P}\right).
	\end{align}
		
\end{lemma}

\begin{proof}
    Let us first handle the possible contribution of the exceptional character $\widetilde{\chi}\pmod{\widetilde{r}}$ of level $R$. Let $\lambda_d^{\pm}$ be the upper and lower linear sieve coefficients with sifting parameter $P$ and level $D=N^{1/2}$. Then 
    \begin{align*}
        \sum_{\substack{d\mid n\\d\leq D}}\lambda_d^{-} \leq  \1_{(n,\prod_{p\leq P}p)=1}\leq \sum_{\substack{d\mid n\\d\leq D}}\lambda_d^{+},  
    \end{align*}
    so for $v\in \{-1,+1\}$ we have
    \begin{align*}
        \sum_{\substack{d\mid n\\d\leq D}}\lambda_d^{-}\1_{\widetilde{\chi}(n)=v} \leq  \1_{(n,\prod_{p\leq P}p)=1}\1_{\widetilde{\chi}(n)=v}\leq \sum_{\substack{d\mid n\\d\leq D}}\lambda_d^{+}\1_{\widetilde{\chi}(n)=v} . 
    \end{align*}
     Hence, for any $v\in \{-1,+1\}$ and any interval $I\subset [1,N]$ with $|I|\geq N^{9/10}$, by the P\'olya--Vinogradov inequality and the fundamental lemma of the linear sieve~\cite[eq. (11.134) and Theorem 11.13]{cribro} we have
    \begin{align*}
        \sum_{n\in I}\1_{(n,\prod_{p\leq P}p)=1}\1_{\widetilde{\chi}(n)=v}&\leq \sum_{d\leq D}\lambda_d^{+}\sum_{\substack{n\in I\\d\mid n}}\1_{\widetilde{\chi}(n)=v}\\
        &=\sum_{d\leq D}\lambda_d^{+}\sum_{\substack{m\in I/d\\(m,\widetilde{r})=1}}\left(\frac{1}{2}+\frac{v}{2}\widetilde{\chi}(m)\right)\\
        &=\frac{\varphi(\widetilde{r})}{\widetilde{r}}\sum_{d\leq D}\lambda_d^{+}\frac{|I|}{d}+O(\widetilde{r}D)+O(D R^{1/2}\log R)\\
        &=\frac{\varphi(\widetilde{r})}{\widetilde{r}}\left(1+O\left(\exp\left(-\frac{1}{2}\frac{\log N}{\log P}\right)\right)\right)\prod_{p\leq P}\left(1-\frac{1}{p}\right)|I|.
    \end{align*}
    By a symmetric  argument, we also have the corresponding lower bound. Hence,  
    \begin{align*}
        \frac{1}{|I|}\left|\sum_{n\in I}\1_{(n,\prod_{p\leq P}p)=1}\widetilde{\chi}(n)\right|=& \frac{1}{|I|}\left|\sum_{n\in I}\1_{(n,\prod_{p\leq P}p)=1}\1_{\widetilde{\chi}(n)=+1}-\sum_{n\in I}\1_{(n,\prod_{p\leq P}p)=1}\1_{\widetilde{\chi}(n)=-1}\right|\\
        \ll &\exp\left(-\frac{1}{2}\frac{\log N}{\log P}\right)\prod_{p\leq P}\left(1-\frac{1}{p}\right).  
    \end{align*}

    This is an admissible error term, so now it suffices to prove~\eqref{eq:Gallagher5} with the additional summation condition $\chi\neq \widetilde{\chi}$. 

    We write
    \begin{align*}
        r_P(n)=\prod_{p< P}\left(1-\frac{1}{p}\right)^{-1} \left(\sum_{1\leq k< (\log N)/(\log P)}\frac{1}{k!}\sum_{\substack{n=p_1\cdots p_k\\p_1,\ldots, p_k\geq P}} 1+O(\1_{\exists p> P\colon\, p^2\mid n} \log N)\right).
    \end{align*}
    Since $P\geq R^{10}$, the contribution of the term involving higher prime powers is negligible. Thus, by the triangle inequality and Mertens' theorem, to show~\eqref{eq:Gallagher5} it suffices to show that
    \begin{align*}
     \sum_{1\leq k \leq \frac{\log N}{\log P}} S_k\ll \frac{\log N}{\log P}\exp\left(-c\frac{\log P }{\log R}\right),   
    \end{align*}
    where
    \begin{align*}
        S_k=\frac{\log P}{k!} \sum_{P\leq p_1,p_2,\ldots, p_{k-1}\leq N}\,\,\sum_{1\leq r\leq R}\quad\asum_{\substack{\chi\pmod r\\\chi\neq \widetilde{\chi}}}\max_{\substack{I\subset [1,N]\\I \textnormal{ interval}\\}}\frac{1}{|I|+N/R} \Bigg|\sum_{\substack{p_{k}\in I/(p_1p_2\cdots p_{k-1})\\ p_k\geq P}}\chi(p_k)\Bigg|.
    \end{align*}
    If $p_1p_2\cdots p_{k-1}>N/P$, then by $I\subset [1,N]$ and $p_k>P$ the innermost sum is empty. We make use of this and insert a factor of $1=\frac{p_1p_2\cdots p_{k-1}}{p_1p_2\cdots p_{k-1}}$ to arrive at
    \begin{align*}
      S_k=  \frac{\log P}{k!} \sum_{\substack{P\leq p_1,p_2,\ldots, p_{k-1}\leq N \\ p_1p_2\cdots p_{k-1}\leq N/P}} \frac{1}{p_1p_2\cdots p_{k-1}}\,\sum_{1\leq r\leq R}\quad\asum_{\substack{\chi\pmod r\\\chi\neq \widetilde{\chi}}}\max_{\substack{I\subset [1,N]\\I \textnormal{ interval}\\}}\frac{p_1p_2\cdots p_{k-1}}{|I|+N/R} \Bigg|\sum_{\substack{p_{k}\in I/(p_1p_2\cdots p_{k-1})\\ p_k\geq P}}\chi(p_k)\Bigg|.
    \end{align*}
     Since we handled the exceptional character separately, we can assume that there exists no exceptional zero of level $R$. Thus Lemma~\ref{le:gallagherprime}, with $N$ replaced by $N/(p_1\cdots p_{k-1})\geq P$, gives the estimate
 \begin{align*}
    S_k\ll \frac{1}{k!}  \sum_{\substack{P< p_1,p_2,\ldots, p_{k-1}\leq N \\ p_1p_2\cdots p_{k-1}\leq  N/P}} \frac{1}{p_1p_2\cdots p_{k-1}} e^{-c \frac{\log P}{\log R}}&\leq  \frac{1}{ k!}   \left(\log \frac{\log N}{\log P}+O(1)\right)^{k-1} \exp\left(-c\frac{\log P }{\log R}\right).
 \end{align*}
 By the Taylor expansion of the exponential function, the sum over $k$ of this is 
 \begin{align*}
    \ll \frac{\log N}{\log P}\exp\left(-c\frac{\log P }{\log R}\right).
 \end{align*}
 This completes the proof.
\end{proof}

The summation windows in the next subsection have length $2N/R^4$, coming from Definition~\ref{def:bdef}, while the moduli occurring there run over $r\leq R^2$ only. We therefore apply Lemmas~\ref{le:gallagher},~\ref{le:gallagherE3} and~\ref{le:gallagher_rough} with $R^4$ in place of $R$, so that the normalising factor $|I|+N/R^4$ matches the window, and restrict the resulting sum over $r$ to $r\leq R^2$. This does not change the level at which we exclude an exceptional zero. Indeed, an exceptional zero of level $R^4$ and quality $\kappa$ of modulus at most $R^2$ is by $\widetilde{\beta}\geq 1-\kappa/\log R^4\geq 1-\kappa/\log R^2$ also one of level $R^2$ and quality $\kappa$, and conversely an exceptional zero of level $R^2$ and quality $\kappa$ is one of level $R^4$ and quality $2\kappa$. If the exceptional modulus exceeds $R^2$, then $\widetilde{\chi}$ does not occur among the primitive characters of modulus $r\leq R^2$ and the term $\1_{\chi=\widetilde{\chi}}$ may be dropped, which by $1-\widetilde{\beta}\leq \kappa/\log R^4$ turns the exceptional bound of part (2) into the unexceptional one.

\subsection{Approximation by Heath-Brown's model}\label{subsec:7.3}

We now show that both primes (or the $E_3^*$ numbers) and the Cram\'er model are Fourier-close to Heath-Brown's model $\Lambda_{R,1}$ with a power saving, if one includes an additional error term that is related to $H_R$. For this purpose, recall that we let $h_{\xi}$  be the multiplicative function supported on square-free integers only and given on the primes by 
 \begin{align*}
        h_{\xi}(p)=\min\Bigl\{1,10(1+|\xi|)\frac{\log p}{\log R} \Bigr\},
    \end{align*}
and we set 
\begin{align*}
    H_R(n)=\tau(n)(\log R) \int_{\mathbb{R}} \frac{h_{\xi}(n)}{(1+|\xi|)^{10}}\d \xi.
\end{align*}

\begin{proposition}[From (almost) primes to Heath-Brown's model]\label{prop:LLR}
    Let $\exp((\log N)^{1/2})\leq R \leq N^{1/200}$ and let $\Lambda^*\in\{\Lambda, \Lambda_{E_3^*}\}$. There exists an arithmetic function $E$  such that the following hold.
	\begin{enumerate}
	   \item If there is no exceptional zero of level $R^2$ and quality $\kappa$, then we have
		\begin{align}\label{eq:LLR1}
		  \|\Lambda^{*}-\Lambda_{R,1}-E\|_{\infty}^{\wedge}\ll N R^{-1/3}
		\end{align}
	    and
	   \begin{align}
             \label{eq:LLR101} |E(n)|&\ll H_R(n)\exp\bigl( -c \kappa \frac{\log N}{\log R} \bigr).
	   \end{align}
		\item If there is an exceptional zero $\widetilde{\beta}$ of level $R^2$ and quality $\kappa$ with exceptional character $\widetilde{\chi}$, then we have
			\begin{align}\label{eq:LLRfouerror}
				\|\Lambda^{*}-\Lambda_{R,1}(1-(\cdot)^{\widetilde{\beta}-1}\widetilde{\chi})+E\|_{\infty}^{\wedge}\ll N (R^{-1/3}+ R^{-1+o(1)}\widetilde{r}^{3/2})
			\end{align}
             and
	       \begin{align}	\label{eq:LLRexcerror}
              |E(n)|&\ll H_R(n) (1-\widetilde{\beta})(\log N)\exp\bigl(-c\frac{\log N}{\log R}\bigr).
			\end{align}
	\end{enumerate}
    
\end{proposition}

\begin{proof}

Let us first consider the unexceptional case. Recall the definition of $b_R$ in Definition~\ref{def:bdef}. 
    By Lemma~\ref{lem:bRextract}, the combinatorial decomposition in Lemma~\ref{lem:combdecomp1} and standard minor arc bounds~\cite[Lemma 13.7, 13.8]{iw-kow}, we have for $\Lambda^*\in \{\Lambda, \Lambda_{E_3^*}\} $ the bound
    \begin{align}\label{eq:roughappr2}
        \|\Lambda^*- \Lambda^* * b_{R}\|^{\wedge}_\infty &\ll N R^{-1/3}.
    \end{align}
    Of course, with slightly weaker $R$ dependence this also follows immediately from Lemma~\ref{lem:minor}.

Recall the definition of $\Lambda_{R,r}$ in~\eqref{eq:LRrdef} and let
\begin{align*}
    E_0(n):=&\sum_{n_1+n_2=n} b_R(n_1) \Lambda(n_2) 1_{p|n_2 \Rightarrow p\leq R^2}\\
    &-\sum_{r\leq R^2}\,\,\,\asum_{\chi\pmod r} \chi(n)\frac{r}{\varphi(r)} \Lambda_{R,r}(n) \frac{\sum_{|n'-n|\leq N/R^4}  \Lambda^*(n')1_{p|n' \Rightarrow p\leq R^2}\overline{\chi}(n')}{2N/R^4}
\end{align*}
 account for the discrepancy of potential prime power divisor less than $R^2$.  By Lemma~\ref{lem:f*b} and the fact that $\Lambda_{R,r}$ vanishes for $r>R^2$, we have for $n\in (N/2,N]$ that
    \begin{align*}
        \Lambda^* *b_{R}(n)&=\sum_{r\leq R^2}\,\,\,\asum_{\chi\pmod r} \chi(n)\frac{r}{\varphi(r)} \Lambda_{R,r}(n) \frac{\sum_{|n'-n|\leq N/R^4}  \Lambda^*(n')\overline{\chi}(n')}{2N/R^4}+E_0(n)\\
        &=\Lambda_{R,1}(n) + E_1(n)+E_0(n),
    \end{align*}
    where
    \begin{align*}
         E_1(n)=\sum_{r\leq R^2}\,\,\,\asum_{\chi\pmod r} \chi(n)\frac{r}{\varphi(r)} \Lambda_{R,r}(n)  \frac{\sum_{|n'-n|\leq N/R^4} (\Lambda^*(n')\overline{\chi}(n')-\1_{r=1})}{2N/R^4}.
    \end{align*}    
  We first bound $E_0$ in $L^1$. Both sums defining $E_0$ collect the non-$R^2$-rough discrepancy that Lemma~\ref{lem:f*b} does not capture, and are supported on $n$ divisible by a prime power $p^k$ with $p\le R^2$. Estimating each sum trivially, using
        \begin{align*}
          \sum_{p\le R^2}\ \sum_{\substack{k\ge1\\ p^k\le N}}\log p\ \ll\ R^2\log N
        \end{align*}
        together with the normalisation of $b_R$ and the bound $\|\Lambda_{R,r}\|_\infty\ll H_R$ of Lemma~\ref{lem:LRupper}, we obtain
        \begin{align*}
          \sum_{n}|E_0(n)|\ \ll\ R^{2+o(1)}\ \ll\ NR^{-1/3},
        \end{align*}
        the last step since $R\le N^{1/200}$. As $E$ enters the proposition only through the Fourier norm $\|\cdot\|_\infty^{\wedge}\le\|\cdot\|_1$ (via~\eqref{eq:LLR1} and the downstream application of Lemma~\ref{lem_fourierimpliesapprox}), this $L^1$ bound makes $E_0$ admissible in the required averaged sense, and it is absorbed into the $NR^{-1/3}$ error of~\eqref{eq:LLR1}. Only $E_1$ needs the pointwise bound~\eqref{eq:LLR101}, so we set $E:=E_1$. By Lemma~\ref{lem:LRupper}, and the unexceptional cases of Lemmas~\ref{le:gallagher} and~\ref{le:gallagherE3}, applied with $R^4$ in place of $R$, we have
        \begin{align*}
            |E_1(n)|\ll& H_R(n)  \sum_{r\leq R^2}\max_{1\leq n\leq N}\,\,\, \asum_{\chi\pmod r}\left|\frac{\sum_{|n'-n|\leq N/R^4} (\Lambda^*(n')\overline{\chi}(n')-\1_{r=1})}{2N/R^4} \right|\\
           \ll & H_R(n)\exp\bigl( -c \kappa \frac{\log N}{\log R} \bigr).
        \end{align*}
    This shows~\eqref{eq:LLR101} (with $E=E_1$)  and completes the proof of the unexceptional case.

 We now consider the exceptional case. To proceed we need to extract the contribution of the exceptional zero to $\Lambda^**b_{R}$. We have
    \begin{align} \label{eq:roughapprLconvb}
        \Lambda^* * b_{R}(n)=\Lambda_{R,1}(n)-\widetilde{\chi}(n)\frac{\widetilde{r}}{\varphi(\widetilde{r})}\Lambda_{R,\widetilde{r}}(n) n^{\widetilde{\beta}-1}+E_2(n)+O\left(\frac{\log N}{R} \right),
    \end{align}
    where we used the estimate
    \begin{align}
        \frac{\sum_{|n'-n|\leq N/R^4 } (n')^{\widetilde{\beta}-1}}{2N/R^4}=n^{\widetilde{\beta}-1}+O\Bigl(\frac{\log N}{R}\Bigr)
    \end{align}
    coming from  the mean value theorem
    and wrote
    \begin{align*}
        E_2(n)=\sum_{r\leq R^2}\,\,\,\asum_{\chi\pmod r}  \chi(n)\frac{r}{\varphi(r)} \Lambda_{R,r}(n) \frac{\sum_{|n'-n|\leq N/R^4 } (\Lambda^*(n')\overline{\chi}(n')-\1_{r=1}+\1_{\chi=\widetilde{\chi}}(n')^{\widetilde{\beta}-1})}{2N/R^4}.
    \end{align*}
    Here, Lemma~\ref{lem:LRupper} and the exceptional cases of Lemmas~\ref{le:gallagher},~\ref{le:gallagherE3}, applied with $R^4$ in place of $R$, show that
    \begin{align*}
         |E_2(n)|\ll H_R(n) (1-\widetilde{\beta}) (\log N)\exp\left(-c \frac{\log N}{\log R}\right).
    \end{align*}
    Thus $E_2$ is admissible for~\eqref{eq:LLRexcerror}. Absorbing $E_0$ into the $NR^{-1/3}$ error as in the unexceptional case and setting $E:=E_2$, we have so far proved the existence of an admissible function $E$ such that
    \begin{align*}
        \| \Lambda^*-\Lambda_{R,1}+(\cdot)^{\widetilde{\beta}-1}\widetilde{\chi}\Lambda_{R,\widetilde{r}}+E\|_{\infty}^{\wedge}\ll N R^{-1/3}.
    \end{align*}
    To show~\eqref{eq:LLRfouerror} and thus complete the proof of the proposition, we need to replace $\Lambda_{R,\widetilde{r}}$ by $\Lambda_{R,1}$. For this, it suffices to show that 
    \begin{align}\label{eq:FFourierbound}
        \|F\|_{\infty}^{\wedge}\ll N R^{-1+o(1)} \widetilde{r}^{3/2},
    \end{align}
    where 
    \begin{align*}
        F(n)=&\Lambda_{R,1}(n)-\widetilde{\chi}(n)\frac{\widetilde{r}}{\varphi(\widetilde{r})}\Lambda_{R,\widetilde{r}}(n) n^{\widetilde{\beta}-1}-(1-\widetilde{\chi}(n)n^{\widetilde{\beta}-1})\Lambda_{R,1}(n)\\
        =&\widetilde{\chi}(n) n^{\widetilde{\beta}-1} \bigl(  \Lambda_{R,1}(n)-\frac{\widetilde{r}}{\varphi(\widetilde{r})}\Lambda_{R,\widetilde{r}}(n) \bigr).
    \end{align*}
     We have $F(n)=0$ unless $(n,\widetilde{r})=1$. Opening the definition of $\Lambda_{R,1}$ and writing $q_1=dq$ where $d\mid \widetilde{r}$ and $(q,\widetilde{r})=1$, we get 
    \begin{align*}
        \Lambda_{R,1}(n)=&\sum_{q_1}\frac{\mu(q_1)c_{q_1}(n)}{\varphi(q_1)} G\left(\frac{\log q_1}{\log R}\right)\\
        =& \sum_{d\mid \widetilde{r}} \sum_{(q,\widetilde{r})=1}\frac{\mu(dq)c_{dq}(n)}{\varphi(dq)} G\left(\frac{\log dq}{\log R}\right)\\
        =&  \sum_{(q,\widetilde{r})=1} \frac{\mu(q)c_{q}(n)}{\varphi(q)}\sum_{d\mid \widetilde{r}} \frac{\mu(d)c_{d}(n)}{\varphi(d)}G\left(\frac{\log dq}{\log R}\right)\\
        =& \sum_{(q,\widetilde{r})=1} \frac{\mu(q)c_{q}(n)}{\varphi(q)}\sum_{d\mid \widetilde{r}} \frac{\mu(d)^2}{\varphi(d)}G\left(\frac{\log dq}{\log R}\right),
    \end{align*}
    since $c_d(n)=\mu(d)$ for $(d,n)=1$. Thus, for $(n,\widetilde{r})=1$, we get
    \begin{align*}
        \Lambda_{R,1}(n)-\frac{\widetilde{r}}{\varphi(\widetilde{r})}\Lambda_{R,\widetilde{r}}(n)&= \sum_{(q,\widetilde{r})=1} \frac{\mu(q)c_{q}(n)}{\varphi(q)}\left(\sum_{d\mid \widetilde{r}} \frac{\mu(d)^2}{\varphi(d)}G\left(\frac{\log dq}{\log R}\right)- \frac{\widetilde{r}}{\varphi(\widetilde{r})} G\left(\frac{\log \widetilde{r}q}{\log R}\right)\right)\\
        &=E_3(n)+E_4(n),
    \end{align*}    
    say, where $E_3(n)$ accounts for the contribution of $q \leq R/\widetilde{r}$ and $E_4(n)$ for the remaining $q$. We claim $E_3(n)=0$. Indeed, recalling that $G(t)=1$ for $0\leq t \leq 1$ we have for $q \leq R/\widetilde{r}$ that
    \begin{align*}
       \sum_{d\mid \widetilde{r}} \frac{\mu(d)^2}{\varphi(d)}G\left(\frac{\log dq}{\log R}\right)- \frac{\widetilde{r}}{\varphi(\widetilde{r})} G\left(\frac{\log \widetilde{r}q}{\log R}\right)&= \sum_{d\mid \widetilde{r}} \frac{\mu(d)^2}{\varphi(d)}-\frac{\widetilde{r}}{\varphi(\widetilde{r})}\\
       &=0.
    \end{align*}
    Thus,~\eqref{eq:FFourierbound} follows, if we can show
    \begin{align}\label{eq:roughapprE_3}
        \|\widetilde{\chi} (\cdot)^{\widetilde{\beta}-1}E_4 \|^{\wedge}_\infty  \ll  N R^{-1+o(1)} \widetilde{r}^{3/2}.
    \end{align}
     By Lemma~\ref{lem_twistsievechar} we have
    \begin{align*}
        \|\widetilde{\chi} (\cdot)^{\widetilde{\beta}-1}E_4 \|^{\wedge}_\infty \ll \sqrt{\widetilde{r}}\|(\cdot)^{\widetilde{\beta}-1}E_4  \|^{\wedge}_\infty.
    \end{align*}    We then estimate the Fourier transform of $(\cdot)^{\widetilde{\beta}-1}E_4$ at $\alpha \in \mathbb{R}$ as
    \begin{align*}
        &\sum_{n\leq N} n^{\widetilde{\beta}-1} E_4(n)e(\alpha n)\\
        =& \sum_{\substack{(q,\widetilde{r})=1\\ q>R/\widetilde{r}}} \frac{\mu(q)}{\varphi(q)}\left(\sum_{d\mid \widetilde{r}} \frac{\mu(d)^2}{\varphi(d)}G\left(\frac{\log dq}{\log R}\right)- \frac{\widetilde{r}}{\varphi(\widetilde{r})} G\left(\frac{\log \widetilde{r}q}{\log R}\right)\right)\,\,\,\asum_{a\pmod q}\,\sum_{n\leq N}n^{\widetilde{\beta}-1}e\left(n(\alpha-a/q)\right) \\
        \ll & \sum_{\substack{(q,\widetilde{r})=1\\R/\widetilde{r}<q\leq R^2}}\frac{R^{o(1)}}{q} \,\,\,\asum_{a\pmod q} \left| \sum_{n\leq N}n^{\widetilde{\beta}-1}e\left(n(\alpha-a/q\right)\right|.
    \end{align*}
    By partial summation (or~\cite[Proposition 5.3]{green-sarkozy}), the sum over $n$ is bounded by $\min\{N,\|\alpha-a/q\|^{-1}\}$. Observe that the appearing fractions $a/q$ are at least $R^{-4}$-spaced. The nearest fraction to $\alpha$ (with $q>R/\widetilde{r}$) we estimate trivially by $\min\{\cdot\}\le N$, contributing $\ll R^{o(1)}N/q\ll R^{o(1)}N\widetilde{r}/R$. The remaining $R^{-4}$-spaced fractions contribute $\ll R^{o(1)}\sum'\|\alpha-a/q\|^{-1}\ll R^{4+o(1)}$. Hence the previous expression is
    \begin{align*}
        {}\ll{} \frac{N R^{o(1)}\widetilde{r}}{R}+R^{4+o(1)}\ll  \frac{N R^{o(1)}\widetilde{r}}{R}.
    \end{align*}
    Thus,~\eqref{eq:roughapprE_3}  follows, and this was enough to complete the proof.
\end{proof}

\begin{proposition}[From Cram\'er to Heath-Brown]\label{prop:rLR}
    Let $\exp((\log N)^{1/2})\leq R^{40}\leq P\leq N^{1/5}$. There exists an arithmetic function $E$  such that the following hold.
    \begin{enumerate}
        \item We have
		\begin{align}\label{eq:rLR1}
		  \|r_P-\Lambda_{R,1}-E\|_{\infty}^{\wedge}\ll N R^{-1/3}
		\end{align}
	    and
	   \begin{align}
             \label{eq:rLR101} |E(n)|&\ll H_R(n) \Bigl(\frac{\log N}{\log P}\exp\bigl(-c\frac{\log P}{\log R} \bigr)+\exp\bigl({-c  \frac{\log N}{\log P}}\bigr)\Bigr).
	   \end{align}
       \item Let $\beta\in (1/2,1)$, and let $\chi$ be a primitive quadratic character of modulus $r\in \mathbb{N}$. Then we have\begin{align}\label{eq:rLRfouerror}
				\|r_P(1-(\cdot)^{\beta-1}\chi)-\Lambda_{R,1}(1-(\cdot)^{\beta-1}\chi)+E\|_{\infty}^{\wedge}\ll N R^{-1/3} r^{1/2}.
			\end{align}
             Here $E$ fulfils again~\eqref{eq:rLR101}, and now additionally
	       \begin{align}	\label{eq:rLRexcerror}
              \1_{\chi(n)=1}|E(n)|&\ll H_R(n) (1-\beta)(\log N)\Bigl(\frac{\log N}{\log P}\exp\left(-c\frac{\log P }{\log R}\right)+\exp\left(-c\frac{\log N }{\log P}\right) \Bigr).
			\end{align}
    \end{enumerate}
\end{proposition}

\begin{proof}
    The proof of~\eqref{eq:rLR1} and~\eqref{eq:rLR101} is identical to the one of Proposition~\ref{prop:LLR} in the unexceptional case, apart from using Lemma~\ref{le:gallagher_rough}, applied with $R^4$ in place of $R$, here.

To prove the statements involving a quadratic character, let 
    $f_1(n)=r_{P}(n)(1-\chi(n)n^{\beta-1})$. We first approximate this by $f_2(n)=(1-\chi(n)n^{\beta-1})(r_P*b_{R})(n)$ and claim that
    \begin{align}\label{eq:roughapproxf1f2}
        \|f_1-f_2\|^{\wedge}_\infty \ll N R^{-1/3} \sqrt{r}.
    \end{align}
    By the triangle inequality
    \begin{align*}
         \|f_1-f_2\|^{\wedge}_\infty\leq \|r_P-r_P*b_R\|^{\wedge}_\infty+ \|(\cdot)^{\beta-1}\chi\cdot (r_P-r_P*b_R)\|^{\wedge}_\infty.
    \end{align*}
    By a standard minor arc bound based on Lemma~\ref{lem:combdecomp1} and~\cite[Lemmas 13.7 and 13.8]{iw-kow}, the first part is acceptably small. To estimate the second term on the right, we  remove the character with Lemma~\ref{lem_twistsievechar} and move the weight  $n^{\beta-1}$ inside the convolution. By Lemma~\ref{le:smoothconvo} with $\psi(n)=n^{\beta-1}$ and the upper bound 
    \begin{align*}
        \|\psi'\1_{[N/3,2N]}\|_\infty\leq N^{-1}
    \end{align*}
    we then have 
    \begin{align*}
        \|(\cdot)^{\beta-1}\chi\cdot (r_P-r_P*b_R)\|^{\wedge}_\infty \ll \sqrt{r} \Bigl(\|r_P (\cdot)^{\beta-1}-r_P(\cdot)^{\beta-1}*b_R\|^{\wedge}_\infty+NR^{-1}\Bigr).
    \end{align*}
    The remaining Fourier norm is sufficiently small by Lemmas~\ref{lem:bRextract} and~\ref{lem:roughexceptionalminorarc}. The bound~\eqref{eq:roughapproxf1f2} follows.
    
    We now consider $f_2$. We rewrite the convolution with multiplicative characters as in~\eqref{eq:f*b} and extract an error term to get
    \begin{align}\label{eq:roughapproxf2exp}
        f_2(n)=(1-\chi(n)n^{\beta-1})\Lambda_{R,1}(n)+E_2(n),
    \end{align}
    where 
    \begin{align*}
        E_2(n)=(1-\chi(n)n^{\beta-1})\sum_{q\leq R^2}\,\,\,\asum_{\psi\pmod {q}}\psi(n)\frac{q}{\varphi(q)} \Lambda_{R,q}(n) \frac{\sum_{|n'-n|\leq N/R^4}( r_{P}(n')\overline{\psi}(n')-\1_{q=1})}{2N/R^4}.
    \end{align*}
    Observe that  \begin{align*}
        \1_{\chi(n)=1}(1-\chi(n)n^{\beta-1})\leq 1-n^{\beta-1}\ll (1-\beta)\log N,
    \end{align*}
    so that an application  of Lemmas~\ref{lem:LRupper} and~\ref{le:gallagher_rough}, applied with $R^4$ in place of $R$, for the terms $2\le q\le R^2$, together with the mean-value estimate for $r_P$ bounding the $q=1$ term by $H_R(n)\exp(-c\log N/\log P)$, shows that
    \begin{align*}
    |E_2(n)|\ll& H_R(n) \Bigl(\frac{\log N}{\log P}\exp\left(-c\frac{\log P }{\log R}\right)+\exp\left(-c\frac{\log N }{\log P}\right)\Bigr)\\
         \1_{\chi(n)=1}|E_2(n)|\ll& H_R(n) (1-\beta)\log N\Bigl(\frac{\log N}{\log P}\exp\left(-c\frac{\log P }{\log R}\right)+\exp\left(-c\frac{\log N }{\log P}\right) \Bigr).
    \end{align*}
    This implies that $E_2$ satisfies both~\eqref{eq:rLR101} and~\eqref{eq:rLRexcerror}.

\end{proof}

\subsection{Approximation by Cram\'er's model}\label{subsec:7.4}

We now combine Propositions~\ref{prop:LLR} and~\ref{prop:rLR} to go from $\Lambda^*$ to $r_P$ in Fourier space. This is the final ingredient we need to achieve the goal of this section, which is the proof of Propositions~\ref{prop_reducetorough} and~\ref{prop_reducetoroughexc}. Note that Theorem~\ref{thm:F-P-C_intro} is contained as a special case.
\begin{theorem}[Fourier-approximating primes by Cram\'er]\label{thm:gallagher_application}
    Let $\exp((\log N)^{1/2})\leq R^{40}\leq P\leq N^{1/5}$, and
    let $\Lambda^*\in \{\Lambda, \Lambda_{E_3^*}\}$. There exists an arithmetic function $E$  such that the following hold.
	\begin{enumerate}
	   \item If there is no exceptional zero of level $R^2$ and quality $\kappa$, then we have
		\begin{align}\label{eq:roughappr1}
		  \|\Lambda^{*}-r_{P}-E\|_{\infty}^{\wedge}\ll N R^{-1/3}
		\end{align}
	    and
	   \begin{align}
             \label{eq:roughappr101} |E(n)|&\ll H_R(n) \Bigl(\frac{\log N}{\log P}\exp\bigl(-c\frac{\log P}{\log R} \bigr)+\exp\bigl( -c \kappa \frac{\log N}{\log R} -c \frac{\log N}{\log P}\bigr)\Bigr).
	   \end{align}
		\item If there is an exceptional zero $\widetilde{\beta}$ of level $R^2$, then if $\widetilde{\chi}$ is the exceptional character, we have
			\begin{align}\label{eq:roughapprexcfouerror}
				\|\Lambda^{*}-r_P(1-(\cdot)^{\widetilde{\beta}-1}\widetilde{\chi})+E\|_{\infty}^{\wedge}\ll N (R^{-1/3}\widetilde{r}^{1/2}+R^{-1+o(1)}\widetilde{r}^{3/2}).
			\end{align}
             Here $E$ satisfies again~\eqref{eq:roughappr101}, and now additionally  
	       \begin{align}	\label{eq:roughapprexcerror}
              \1_{\widetilde{\chi}(n)=1}|E(n)|&\ll H_R(n) (1-\widetilde{\beta})(\log N)\left( \frac{\log N}{\log P}\exp\bigl(-c\frac{\log P}{\log R} \bigr)+\exp\bigl(-c\frac{\log N}{\log P}\bigr)\right).
			\end{align}
	\end{enumerate}
\end{theorem}
	
\begin{proof}
This is an immediate consequence of Propositions~\ref{prop:LLR} and~\ref{prop:rLR}, and the triangle inequality, part~(2) being obtained by combining~\eqref{eq:LLRfouerror} with~\eqref{eq:rLRfouerror} at $\chi=\widetilde{\chi},\beta=\widetilde{\beta},r=\widetilde{r}$, whose $R^{-1/3}\widetilde{r}^{1/2}$ dominates the $R^{-1/3}$ of~\eqref{eq:LLRfouerror}.

\end{proof}

\begin{prop}[Reduction to Cram\'er model, unexceptional case]\label{prop_reducetorough}  Let $P_0=N^{\delta_1},R_0=N^{\delta_1^3},P_1=N^{\delta_1^4},D_1=N^{\delta_1^3/100},$  and let $\lambda$ be the constant from Lemma~\ref{lem_excmodulussize}. Let $f\in\{\omega,\Omega\}$ be a pre-sieve as in Definition~\ref{def:presieve}.
	Assume there is no exceptional zero of level $R_0^2$ and quality $\lambda \delta_1^2$. Then, for any $\epsilon \gg e^{-c \delta_1^{-1/2}}$, we have
		\begin{align*}
			\Lambda(n)f(n+2)&\approx_{\epsilon} r_{P_0}(n) f(n+2),\\
			f(n)\Lambda_{E_3^*}(n+2)&\approx_{\epsilon} f(n)r_{P_0}(n+2).
		\end{align*}
\end{prop}

\begin{proof}
    We only consider the case $\Lambda(n)f(n+2)$, the other case being similar. We apply Theorem~\ref{thm:gallagher_application} with $P=P_0, R=R_0$ and let $E(n)$ be as in~\eqref{eq:roughappr1}. Discarding the sieve weights with Lemma~\ref{lem_twistsievechar}, we get
    \begin{align*}
        \|f^+(\Lambda - r_{P_0}-E)\|^{\wedge}_\infty &\ll D_1  N R_0^{-1/3}\ll N^{1-\delta_1^3/3+\delta_1^3/100}.
    \end{align*}
    By Lemma~\ref{lem_fourierimpliesapprox}  we obtain
    \begin{align*}
        \Lambda(n)f(n+2)&\approx_{\epsilon} r_{P_0}(n) f(n+2)+E(n)f(n+2)
    \end{align*}
    as long as $\epsilon\gg N^{-(\delta_1/10)^{4}}$, which is much better than required. 

    By~\eqref{eq:roughappr101} and the choice of parameters, 
    \begin{align*}
            |E(n)|&\ll H_{R_0}(n) \Bigl(\delta_1^{-1}e^{-c \delta_1^{-2}}+e^{-c \delta_1^2 \delta_1^{-3}-c \delta_1^{-1}}\Bigr)\ll H_{R_0}(n) e^{-c\delta_1^{-1}}.
	   \end{align*}
    Then, an application of Proposition~\ref{prop:LambdaQ} gives us
    \begin{align*}
        \sum_{N/2<n\leq N} |E(n)| |f(n+2)| \Omega(m-n)\Omega(m-n+2)&\ll \delta_1^{-40}e^{-c\delta_1^{-1}} N \mathfrak{S}(m)\\
          &\ll e^{-c \delta_1^{-1/2}}N\mathfrak{S}(m),
    \end{align*}
    if $\delta_1$ is small enough in absolute terms. This implies the proposition, after recalling that in Definition~\ref{def_approx} the companion function $g(n)$ is majorised in absolute value by $\Omega(n)\Omega(n+2)$.
\end{proof}

\begin{prop}[Reduction to Cram\'er model, exceptional case]\label{prop_reducetoroughexc} 
     Let $P_0=N^{\delta_1},R_0=N^{\delta_1^3},P_1=N^{\delta_1^4},D_1=N^{\delta_1^3/100},$ $\lambda$ the constant from Lemma~\ref{lem_excmodulussize} and let $f\in\{\omega,\Omega\}$ be a pre-sieve as in Definition~\ref{def:presieve}. Assume that $\widetilde{\chi}$ is the exceptional character with exceptional zero $\widetilde{\beta}$ of quality $\lambda \delta_1^2$ and level $R_0^2$. Then for any $\epsilon>e^{-c \delta_1^{-1/2}}$, we have
    \begin{align*}
	   \1_{\widetilde{\chi}(n)=-1}\Lambda(n)f(n+2)&\approx_{\epsilon}  \1_{\widetilde{\chi}(n)=-1}(1+n^{\widetilde{\beta}-1}) r_{P_0}(n)f(n+ 2),\\
	   \1_{\widetilde{\chi}(n)=-1}f(n)\Lambda_{E_3^*}(n+2)&\approx_{\epsilon} c_{E_3^*}  \1_{\widetilde{\chi}(n)=-1}f(n)(1-\widetilde{\chi}(n+2)n^{\widetilde{\beta}-1})r_{P_0}(n+2).
	\end{align*}
	Further, for $\widetilde{\epsilon}>(1-\widetilde{\beta})(\log N)e^{-c \delta_1^{-1/2}}$, it holds that
	\begin{align*}
		\1_{\widetilde{\chi}(n)=\widetilde{\chi}(n+2)=1}\Lambda(n)f(n+2)&\approx_{\widetilde{\epsilon}} \1_{\widetilde{\chi}(n)=\widetilde{\chi}(n+2)=1}(1-n^{\widetilde{\beta}-1} )r_{P_0}(n)f(n+2),\\
		\1_{\widetilde{\chi}(n)=\widetilde{\chi}(n+2)=1}f(n)\Lambda_{E_3^*}(n+2)&\approx_{\widetilde{\epsilon}} \1_{\widetilde{\chi}(n)=\widetilde{\chi}(n+2)=1}(1-n^{\widetilde{\beta}-1})f(n) r_{P_0}(n+2) .
	\end{align*}
\end{prop}

\begin{proof}
    We again only handle the case of $\Lambda(n)f(n+2)$, the others being similar (noting that $(n+2)^{\widetilde{\beta}-1}$ can be replaced by $n^{\widetilde{\beta}-1}$ with respect to $\approx$ by Taylor expansion). We start with the observation that $\1_{\widetilde{\chi}(n)=\pm 1}=\frac{1}{2}(1\pm \widetilde{\chi}(n))$ that holds for $(n,\widetilde{r})=1$ and write $\mathfrak{c}(n)$ for $\1_{\widetilde{\chi}(n)=-1}$ or $\1_{\widetilde{\chi}(n)=\widetilde{\chi}(n+2)=1}$ and $\mathfrak{d}(n)=1-\widetilde{\chi}(n)n^{\widetilde{\beta}-1}$. Note that higher prime powers can be again discarded easily. 
    
    By Lemma~\ref{lem_excmodulussize} we can assume that the exceptional modulus is bounded by $\widetilde{R}=N^{2\delta_1^5}$. We use this bound to apply the exceptional case of Theorem~\ref{thm:gallagher_application} with $P=P_0=N^{\delta_1}, R=R_0=N^{\delta_1^3}$ together with Lemma~\ref{lem_twistsievechar} and Lemma~\ref{lem_excmodulussize} to get 
    \begin{align*}
         \|\mathfrak{c}f^+(\Lambda - r_{P_0}\mathfrak{d}+E)\|^{\wedge}_\infty &\ll \widetilde{R}^2 D_1  N R_0^{-1/3} =  N^{1+4 \delta_1^5+\delta_1^3/100-\delta_1^3/3} \ll N^{1-\delta_1^3/4},
    \end{align*}
    where $E(n)$ can be estimated by~\eqref{eq:roughapprexcerror}. By Lemma~\ref{lem_fourierimpliesapprox}, we have
    \begin{align*}
        \mathfrak{c}(n)\Lambda(n)f(n+2)\approx_{\epsilon'} \mathfrak{c}(n) r_{P_0}
        (n)\mathfrak{d}(n)+\mathfrak{c}(n)E(n)f(n+2).
    \end{align*}
    for any $\epsilon'>N^{-\delta_1^{4}}$. This is acceptable, as both $\epsilon$ and $\widetilde{\epsilon}$ are larger than the lower bound for $\epsilon'$ by~\eqref{eq:exczerolowerbound}.
    
    In the case $\mathfrak{c}(n)=\1_{\widetilde{\chi}(n)=-1}$ we have
    \begin{align*}
       \mathfrak{c}(n)\mathfrak{d}(n)=\1_{\widetilde{\chi}(n)=-1}(1+n^{\widetilde{\beta}-1}),
    \end{align*}
    and we trivially bound
    \begin{align*}
        |\1_{\widetilde{\chi}(n)=-1}E(n)|\leq |E(n)|.    
    \end{align*}
    So, just as in the unexceptional case (see proof of Proposition~\ref{prop_reducetorough}),  we get
    \begin{align*}
        \mathfrak{c}(n)E(n)f(n+2)\approx_{\epsilon} 0
    \end{align*}
    in the required range of $\epsilon$. 

    In the case $\mathfrak{c}(n)=\1_{\widetilde{\chi}(n)=\widetilde{\chi}(n+2)=1}$ we have
    \begin{align*}
      \mathfrak{c}(n)\mathfrak{d}(n)=1-n^{\widetilde{\beta}-1},
    \end{align*}
    and we can make use of the improved error in~\eqref{eq:roughapprexcerror} to get
    \begin{align*}
        \1_{\widetilde{\chi}(n)=\widetilde{\chi}(n+2)=1}E(n)f(n+2)\approx_{\widetilde{\epsilon} }0.
    \end{align*}

\end{proof}

	\section{Proof of Theorem~\ref{MT1}}\label{sec:final}
	
	We are now ready to prove Theorem~\ref{MT1}. We again consider the unexceptional case and exceptional case separately. For the proof it is convenient to extend the $\approx_\epsilon$ notation, given by Definition~\ref{def_approx}, to allow for lower bounding.
\begin{definition}\label{def:greaterapprox}
    Let $\epsilon>0$ and $f,g$ be finitely supported arithmetic functions. We write $f \gtrapprox_{\epsilon} g$ if there exists a sequence of functions $f_i$, $0\leq i \leq J$ with $f_0=f, f_J=g$ and either $f_i \approx_\epsilon f_{i+1}$ or $f_i\geq f_{i+1}$.
\end{definition}
	
	\subsection{Unexceptional Case}
	We now prove Theorem~\ref{MT1} in the unexceptional case. Before we do this, we include the following remark that explains some of the choices we make. 

\begin{remark}\label{rem:roughnessplit}
    It is at this point that we use the hierarchy of parameters in Definition~\ref{def:para}. We now summarise the reasoning and in particular explain why the proof below requires us to work with two roughness scales $P_0, P_1$, instead of the one scale as in the sketch in the introduction.

      Proposition~\ref{prop_reducetorough} applies Theorem~\ref{thm:gallagher_application} with $P=P_0$ and $R=R_0$. Its error term~\eqref{eq:roughappr101} is admissible only when both $\frac{\log N}{\log P_0}$ and $\frac{\log P_0}{\log R_0}$ are large, which requires $N\ggg P_0\ggg R_0$.

       Next, Theorem~\ref{thm:gallagher_application} produces in~\eqref{eq:roughappr1} a saving of $R_0^{-1/3}$, which needs to be sufficient to compensate the loss of $D_1$ from discarding the sieve weights with Lemma~\ref{lem_twistsievechar}. This forces $D_1\ll R_0^{1/3}$.

       The application of the fundamental lemma type bound of Lemma~\ref{lem_fundlem} in Proposition~\ref{prop_presieveasmyp} requires $D_1\ggg P_1$.

       Finally, the major arc estimate of Lemma~\ref{lem:majorarcbound}, which is the main input to Proposition~\ref{prop_MSremoval}, is of size $N^{1+o(1)}(P_1^{-1/2}R_1^{1/2}+R_1^{-1})$ and saves only for $R_1$ a power of $N$ with $P_1\gg R_1$.

       In particular, this explains the difference in size of $P_1$ and $P_0$, and why the direct application of our results will, in the language of Proposition~\ref{le_chensieve}, produce
       \begin{align*}
          g_1(n)&\approx_\epsilon \mathcal{V}(\omega_{\textup{M}})\, r_{P_0}(n)r_{P_1}(n+2),\\
          g_2(n)&\approx_\epsilon (3/5+\delta_1)c_{E_3^*}\mathcal{V}(\Omega_{\textup{M}})\, r_{P_1}(n)r_{P_0}(n+2),
       \end{align*}
       the scale $P_0$ appearing where $\Lambda$, respectively $\Lambda_{E_3^*}$, stood. Thus, to create identical terms apart from constants, we need to include an additional step in which we descend from $r_{P_0}$ to $r_{P_1}$. We do so by writing $r_{P_0}=r_{P_1}r_{[P_1,P_0)}$ and inserting for $r_{[P_1,P_0)}$ the intermediate sieve of Lemma~\ref{lem:P_1P_0majo}, which we then remove with Proposition~\ref{prop_MSremoval}. This is possible because at this point of the argument we can freely switch between $P_1$-rough numbers and lower and upper bound pre-sieves with Corollary~\ref{cor_upperlowersieve}, recalling $\omega\leq r_{P_1}\leq \Omega$.
\end{remark}
     We are now able to show the following lemma from which the unexceptional case of Theorem~\ref{MT1} will follow quickly. 
	\begin{lemma}\label{lem_finalunex}
    Let $P_1=N^{\delta_1^4},R_0=N^{\delta_1^3}$ and $\lambda$ the constant from Lemma~\ref{lem_excmodulussize}.
		Assume that there exists no exceptional zero of level $R_0^2$ and quality $\lambda\delta_1^2$. There exists $c_1>0$ such that for $\epsilon = e^{-c \delta_1^{-1/2}}$ we have 
			\begin{align}
				\label{eq_lem_finalunex_1} \Lambda(n)\Lambda_2(n+2)&\gtrapprox_{\epsilon} c_1  r_{P_1}(n)r_{P_1}(n+2).
			\end{align}
		\end{lemma}

		\begin{proof}[Proof of Lemma~\ref{lem_finalunex}]
        We fix again the choices $P_0=N^{\delta_1},D_1=N^{\delta_1^3/100}$ and start by applying Proposition~\ref{le_chensieve} which gives us the existence of main-sieves $\omega_{\textup{M}},\Omega_{\textup{M}},\Omega'_{\textup{M}}$ as in Definition~\ref{def:mainsieve} such that
        \begin{align}\label{eq:finalunex1}
            \Lambda(n)\Lambda_2(n+2) \gtrapprox_{\epsilon} g_1(n)-g_2(n)+g_3(n),
        \end{align}
        with 
        \begin{align*}
            g_1(n)&=\Lambda(n)\Omega(n+2)\omega_{\textnormal{M}}(n+2)\\
    g_2(n)&=(3/5+\delta_1)c_{E_3^*}\Omega(n)\Omega_{\textnormal{M}}(n)\Lambda_{E_3^*}(n+2)\\
    g_3(n)&=\Lambda(n)\bigl(\omega -\Omega\bigr)(n+2)\Omega'_{\textnormal{M}}(n+2).
        \end{align*}
        We define the following constants that are related to our main-sieves: 
		\begin{align*}
        \mathcal{V}(\omega_{\textup{M}})&\coloneqq\prod_{P_1\leq p < N^{1/10}}\left(1-\frac{1}{p}\right)^{-1} \sum_{d}\frac{\lambda^\omega_{\textup{M}}(d)}{\varphi(d)}\\
        \mathcal{V}(\Omega_{\textup{M}})&\coloneqq\prod_{P_1\leq p < N^{1/10}}\left(1-\frac{1}{p}\right)^{-1} \sum_{d}\frac{\lambda^\Omega_{\textup{M}}(d)}{\varphi(d)}\\
        \mathcal{V}(\Omega'_{\textup{M}})&\coloneqq\prod_{P_1\leq p < N^{1/10}}\left(1-\frac{1}{p}\right)^{-1} \sum_{d}\frac{\lambda^{\Omega'}_{\textup{M}}(d)}{\varphi(d)}.
    \end{align*}
    Here each $\mathcal{V}(\cdot)$ carries the same normalising product $\prod_{P_1\le p<N^{1/10}}(1-1/p)^{-1}$ emitted by Proposition~\ref{prop_MSremoval}, even though $\Omega_{\textup{M}}$ is constructed on the larger range $[P_1,N^{1/6})$ --- this deliberate normalisation is the source of the constant $3/5$ in $g_2$.
       The next task is to remove the main-sieve. Proposition~\ref{prop_MSremoval} together with the fact that $\epsilon=e^{-c \delta_1^{-1/2}}\geq N^{-(\delta_1/10)^4}$ gives us
        \begin{align*}
            g_1(n)&\approx_\epsilon  \mathcal{V}(\omega_{\textup{M}}) \Lambda(n)\Omega(n+2) \\
            g_2(n)&\approx_\epsilon \mathcal{V}(\Omega_{\textup{M}})(3/5+\delta_1)c_{E_3^*}\Omega(n)\Lambda_{E_3^*}(n+2)\\
            g_3(n)&\approx_\epsilon \mathcal{V}(\Omega'_{\textup{M}}) \Lambda(n)\bigl(\omega -\Omega\bigr)(n+2).
        \end{align*}

    We consider the contribution of each $g_i$ separately now. We start with $g_3$ and show that it is negligible. Indeed, using the facts that $\Omega\geq \omega $, that $\mathcal{V}(\Omega'_{\textup{M}})$ is bounded by~\eqref{eq:chentrivial_numerics}, and applying Corollary~\ref{cor_upperlowersieve}, we get
    \begin{align*}
        0&\geq \mathcal{V}(\Omega'_{\textup{M}}) \Lambda(n)\bigl(\omega -\Omega\bigr)(n+2)  \\
        & \gg \mathcal{V}(\Omega'_{\textup{M}}) \delta_1^{-4} \Omega(n)\bigl(\omega -\Omega\bigr)(n+2)\\
        &\approx_{\epsilon}0,
    \end{align*}
    as long as $\epsilon\gg \delta_1^{-4} e^{-\delta_1^{-1}/10} $, which is acceptable. Thus,
    \begin{align}\label{eq_finalproof_2}
        g_3(n)\approx_{\epsilon}0.
    \end{align}

 For $g_1$, an application of Proposition~\ref{prop_reducetorough} shows
    \begin{align*}
        g_1(n)\approx_\epsilon  \mathcal{V}(\omega_{\textup{M}}) r_{P_0}(n)\Omega(n+2).
    \end{align*}
    By definition $r_{P_0}=r_{P_1}r_{[P_1,P_0)}$ and we can use Lemma~\ref{lem:P_1P_0majo} to lower bound
    \begin{align*}
        r_{P_0}\geq r_{P_1} \omega_{[P_1,P_0)}.
    \end{align*}
    Since $\Omega$ is non-negative, we obtain
    \begin{align*}
        g_1(n)\gtrapprox_{\epsilon} \mathcal{V}(\omega_{\textup{M}}) r_{P_1} \omega_{[P_1,P_0)}(n)\Omega(n+2).
    \end{align*}
    We now remove the $[P_1,P_0)$ sieve component via Proposition~\ref{prop_MSremoval}. Together with~\eqref{eq:P1P0majo1} of Lemma~\ref{lem:P_1P_0majo} and  two applications of Corollary~\ref{cor_upperlowersieve} (recalling $\omega\leq r_{P_1}\leq \Omega$), we obtain
    \begin{align}
        \nonumber g_1(n)&\gtrapprox_{\epsilon}  \mathcal{V}(\omega_{\textup{M}}) r_{P_1} \omega_{[P_1,P_0)}(n)\Omega(n+2)\\
       \label{eq:finalunex2} &\approx_{\epsilon} \mathcal{V}(\omega_{\textup{M}}) r_{P_1}(n)r_{P_1}(n+2).
    \end{align}

For $g_2$ we similarly apply first Proposition~\ref{prop_reducetorough} to get
\begin{align*}
    g_2(n)\approx_\epsilon \mathcal{V}(\Omega_{\textup{M}})(3/5+\delta_1)c_{E_3^*}\Omega(n) r_{P_0}(n+2).
\end{align*}
We can then follow the same steps as for $g_1$, with the exception of using an upper bound sieve in the range $[P_1,P_0)$, since $g_2$ appears with negative sign. This shows
\begin{align}
      \label{eq:finalunex3}   -g_2(n)\gtrapprox_{\epsilon}  -\mathcal{V}(\Omega_{\textup{M}})(3/5+\delta_1)c_{E_3^*} r_{P_1}(n)r_{P_1}(n+2).
\end{align}

    Combining~\eqref{eq:finalunex1},~\eqref{eq_finalproof_2},~\eqref{eq:finalunex2}, and~\eqref{eq:finalunex3}, we have shown that
    \begin{align*}
        \Lambda(n)\Lambda_2(n+2) \gtrapprox_{\epsilon} \bigg(\mathcal{V}(\omega_{\textup{M}})-\mathcal{V}(\Omega_{\textup{M}})(3/5+\delta_1)c_{E_3^{*}}\bigg) r_{P_1}(n)r_{P_1}(n+2).
    \end{align*}
    By~\eqref{eq:chenlower_numerics}, the leading constant is lower bounded away from $0$ if $\delta_1$ is small enough in absolute terms. 
    \end{proof}
        \begin{proof}[Proof of Theorem~\ref{MT1} in the unexceptional case]
			Applying a simple upper bound sieve, we see that
			\begin{align*}
				\Lambda(n)\Lambda_2(n+2)&\leq \delta_1^{-8}r_{P_1}(n)r_{P_1}(n+2)\\
				&\leq \delta_1^{-8}\Omega(n)\Omega(n+2).
			\end{align*}
			Thus, by the non-negativity of $\Lambda(n)\Lambda_2(n+2)$, the definition of $\gtrapprox_{\epsilon}$, and Lemma~\ref{lem_finalunex}, there exist a constant $c_1>0$ and a function $E_1$ fulfilling
			\begin{align*}
				\Lambda \Lambda_2^+*\Lambda \Lambda_2^+(m)&\geq  c_1 \Lambda \Lambda_2^+* r_{P_1} r^+_{P_1}(m)+\Lambda \Lambda_2^+*E_1(m),
			\end{align*}
			and such that for all natural numbers $m\in [5N/4,7N/4]$ with at most $N^{1-(\delta_1/10)^{4}}$ exceptions we have
			\begin{align*}
				|\Lambda \Lambda_2^+*E_1(m)|\ll \delta_1^{-8} e^{-c\delta_1^{-1/2}}  m \mathfrak{S}(m).
			\end{align*}
			Similarly, using the non-negativity of $r_{P_1}(n)r_{P_1}(n+2)$, there exists a function $E_2$ whose contribution is bounded in the same manner and such that
			\begin{align*}
				c_1 \Lambda \Lambda_2^+* r_{P_1} r^+_{P_1}(m)&\geq c_1^2 r_{P_1} r^+_{P_1}* r_{P_1} r^+_{P_1}(m)+E_2* r_{P_1} r^+_{P_1}(m).
			\end{align*}
			  Finally, using upper and lower bound pre-sieves in conjunction with Corollary~\ref{cor_upperlowersieve} and Proposition~\ref{prop_presieveasmyp}, we get the asymptotics 
			\begin{align*}
				r_{P_1} r^+_{P_1}* r_{P_1} r^+_{P_1}(m)=m \mathfrak{S}(m)+O\left(m \mathfrak{S}(m) e^{-\delta_1^{-1}/10}\right).
			\end{align*}
			If $\delta_1$ is sufficiently small in absolute terms, the factors $e^{-\delta_1^{-1}/10}$ and $\delta_1^{-8} e^{-c\delta_1^{-1/2}}$ are small. We conclude from our chain of inequalities and error term estimates that
			\begin{align*}
				\Lambda \Lambda_2^+*\Lambda \Lambda_2^+(m)\geq \frac{c_1^2}{2} m \mathfrak{S}(m)
			\end{align*}
			for all natural numbers $m\in [5N/4,7N/4]$ with $\ll N^{1-(\delta_1/10)^{4}}$ exceptions. Since $\mathfrak{S}(m)\gg 1$ for all $m\equiv 4\pmod{6}$, this proves Theorem~\ref{MT1} in the unexceptional case, after sorting $m$ into intervals of the form $[5N/4,7N/4]$. 
		\end{proof}
		
			\subsection{Exceptional Case}
			The exceptional case of Theorem~\ref{MT1} works in broad terms similarly to the unexceptional case, apart from two additional technical complications described in the introduction, see Subsection~\ref{subsec:introexc}.

			\begin{lemma}\label{lem_finalex}
             Let $P_1=N^{\delta_1^4},R_0=N^{\delta_1^3}$ and $\lambda$ the constant from Lemma~\ref{lem_excmodulussize}.
				Assume that there exists an exceptional character $\widetilde{\chi}$ with exceptional zero $\widetilde{\beta}$ of level $R_0^2$ and quality $\lambda\delta_1^2$. There exist $c_1>0$ such that the following holds. Let 
				\begin{align*}
					\epsilon = e^{-c \delta_1^{-1/2}},\quad 
					\widetilde{\epsilon}=\epsilon (1-\widetilde{\beta}) \log N.
				\end{align*}
				We have
				\begin{align}
					\1_{\widetilde{\chi}(n)=-1} \label{eq_lem_finalex_1} \Lambda(n)\Lambda_2(n+2)&\gtrapprox_{\epsilon} c_1 \1_{\widetilde{\chi}(n)=-1} r_{P_1}(n)r_{P_1}(n+2)\\
					\1_{\widetilde{\chi}(n)=\widetilde{\chi}(n+2)=1}\Lambda(n)\Lambda_2(n+2)&\gtrapprox_{\widetilde{\epsilon}} c_1 (1-\widetilde{\beta}) (\log N) \1_{\widetilde{\chi}(n)=\widetilde{\chi}(n+2)=1} r_{P_1}(n)r_{P_1}(n+2). \label{eq_lem_finalex_2} 
				\end{align}
			\end{lemma}

			\begin{proof}[Proof of Lemma~\ref{lem_finalex}]
                Let $\mathfrak{c}(n)=  \1_{\widetilde{\chi}(n)=-1} $ or $\mathfrak{c}(n)=\1_{\widetilde{\chi}(n)=\widetilde{\chi}(n+2)=1}$ for the two cases we need to consider. We start as in the proof of Lemma~\ref{lem_finalunex} and apply Proposition~\ref{le_chensieve} to lower bound
                \begin{align*}
                    \mathfrak{c}(n)\Lambda(n)\Lambda_2(n+2)\geq  \mathfrak{c}(n)\bigl(g_1(n)-g_2(n)+g_3(n)\bigr).
                \end{align*}
                By Lemma~\ref{lem_excmodulussize}, after possibly decreasing the size of $\delta_1$, we can assume that $\cond(\widetilde{\chi})\leq \widetilde{R}$, and recalling~\eqref{eq:exczerolowerbound} we consequently have $\widetilde{\epsilon}>N^{-\delta_1^{4}}$. This allows us to remove the main-sieves using the second statement of Proposition~\ref{prop_MSremoval}. We get
                \begin{align*}
                \mathfrak{c}(n)g_1(n)&\approx_{\widetilde{\epsilon}} \mathfrak{c}(n) \mathcal{V}(\omega_{\textup{M}}) \Lambda(n)\Omega(n+2) \\
                \mathfrak{c}(n)g_2(n)&\approx_{\widetilde{\epsilon}} \mathfrak{c}(n)\mathcal{V}(\Omega_{\textup{M}})(3/5+\delta_1)c_{E_3^*}\Omega(n)\Lambda_{E_3^*}(n+2)\\
               \mathfrak{c}(n) g_3(n)&\approx_{\widetilde{\epsilon}} \mathfrak{c}(n) \mathcal{V}(\Omega'_{\textup{M}}) \Lambda(n)\bigl(\omega -\Omega\bigr)(n+2).
              \end{align*}
              By the same argument as in the unexceptional case,
              \begin{align*}
                 \mathfrak{c}(n) g_3(n)\approx_\epsilon 0.
              \end{align*}
              To improve this in the case $\mathfrak{c}(n)=\1_{\widetilde{\chi}(n)=\widetilde{\chi}(n+2)=1}$, we first apply Proposition~\ref{prop_reducetoroughexc} to replace the $\Lambda$ component. Together with Corollary~\ref{cor_upperlowersieve} and the bound $1-n^{\widetilde{\beta}-1}\ll (1-\widetilde{\beta})\log n$, this yields 
              \begin{align*}
                  &\1_{\widetilde{\chi}(n)=\widetilde{\chi}(n+2)=1} \mathcal{V}(\Omega'_{\textup{M}}) \Lambda(n)\bigl(\omega -\Omega\bigr)(n+2)\\
                  \approx_{\widetilde{\epsilon}}& \1_{\widetilde{\chi}(n)=\widetilde{\chi}(n+2)=1} \mathcal{V}(\Omega'_{\textup{M}}) (1-n^{\widetilde{\beta}-1})r_{P_0}(n)\bigl(\omega -\Omega\bigr)(n+2)\\
                  \approx_{\widetilde{\epsilon}}&0.
              \end{align*}

              Apart from replacing Proposition~\ref{prop_reducetorough} by Proposition~\ref{prop_reducetoroughexc}, we can follow the same steps as in the unexceptional case to get
              \begin{align*}
                  \1_{\widetilde{\chi}(n)=-1}g_1(n)&\gtrapprox_{\epsilon} \mathcal{V}(\omega_{\textup{M}}) (1+n^{\widetilde{\beta}-1})r_{P_1}(n)r_{P_1}(n+2)
              \end{align*}
              and
              \begin{align*}
                  -\1_{\widetilde{\chi}(n)=-1}g_2(n)&\gtrapprox_\epsilon  -\mathcal{V}(\Omega_{\textup{M}})(3/5+\delta_1)c_{E_3^*} (1-\widetilde{\chi}(n+2)n^{\widetilde{\beta}-1})r_{P_1}(n)r_{P_1}(n+2)\\
                  &\geq - \mathcal{V}(\Omega_{\textup{M}})(3/5+\delta_1)c_{E_3^*} (1+n^{\widetilde{\beta}-1})r_{P_1}(n)r_{P_1}(n+2).
              \end{align*}
              Combining the terms and discarding the non-negative $n^{\widetilde{\beta}-1}$ term,~\eqref{eq_lem_finalex_1} follows. The remaining case~\eqref{eq_lem_finalex_2} is proved in the same way: we either immediately get $\approx_{\widetilde{\epsilon}}$ from Proposition~\ref{prop_reducetoroughexc} or have a factor of $(1-\widetilde{\beta})(\log N) $ that improves the approximation correspondingly.
              
              \end{proof}

\begin{proof}[Proof of Theorem~\ref{MT1} in the exceptional case.]

				We partition the relevant $m$ according to whether $\sigma_1(m)\ge -1/2$ or $\sigma_1(m)<-1/2$, and bound the exceptional set within each class separately. As $0\leq \1_{\widetilde{\chi}(n)=-1} \leq 1$,
				\begin{align*}
					\Lambda \Lambda_2^+*\Lambda \Lambda_2^+(m)&\geq \1_{\widetilde{\chi}(\cdot)=-1}\Lambda \Lambda_2^+*\1_{\widetilde{\chi}(\cdot)=-1}\Lambda \Lambda_2^+(m).
				\end{align*}
				We use~\eqref{eq_lem_finalex_1} of Lemma~\ref{lem_finalex} and follow the same strategy as in the unexceptional case to get
				\begin{align*}
					\1_{\widetilde{\chi}(\cdot)=-1}\Lambda \Lambda_2^+*\1_{\widetilde{\chi}(\cdot)=-1}\Lambda \Lambda_2^+(m)\geq c_1^2 \1_{\widetilde{\chi}(\cdot)=-1}r_{P_1} r^+_{P_1}*\1_{\widetilde{\chi}(\cdot)=-1} r_{P_1} r^+_{P_1}(m)+\widetilde{E}(m),
				\end{align*}
				where
				\begin{align*}
					|\widetilde{E}(m)|\ll \delta_1^{-8} e^{-c \delta_1^{-1/2}} m \mathfrak{S}(m).
				\end{align*}
				for all natural numbers $m\in [5N/4,7N/4]$ with at most $O(N^{1-(\delta_1/10)^4})$ exceptions. By the first case of Proposition~\ref{prop_fundlem_exceptional} and the lower bound of the exceptional modulus~\eqref{eq:rlowerbound} we get
				\begin{align*}
					\1_{\widetilde{\chi}(\cdot)=-1}r_{P_1} r^+_{P_1}*\1_{\widetilde{\chi}(\cdot)=-1} r_{P_1} r^+_{P_1}(m)=m \mathfrak{S}(m)\left(\frac{1+\sigma_1(m)}{4}+O(\delta_1^{-5}(\log N)^{-1})\right)\bigl(1+ O(e^{-c \delta_1^{-1}})\bigr).
				\end{align*}
                    This shows that
				\begin{align*}
					\Lambda \Lambda_2^+*\Lambda \Lambda_2^+(m)\geq \frac{1}{10} m \mathfrak{S}(m)
				\end{align*}
				for all natural numbers $m\in [5N/4,7N/4]$ with 
				\begin{align*}
					\sigma_1(m)\geq -1/2,
				\end{align*}
				apart from
				$O(N^{1-(\delta_1/10)^4})$ exceptions.  It remains to prove Theorem~\ref{MT1} under the assumption that
				\begin{align}\label{eq_MT1exceptional_assumptions}
					\sigma_1(m)< -1/2.
				\end{align}
				
				Multiplying now one of the functions with $\1_{\widetilde{\chi}(\cdot)=\widetilde{\chi}(\cdot+2)=1}$, we have
				\begin{align*}
					\Lambda \Lambda_2^+*\Lambda \Lambda_2^+(m)&\geq \1_{\widetilde{\chi}(\cdot)=-1}\Lambda \Lambda_2^+*\1_{\widetilde{\chi}(\cdot)=\widetilde{\chi}(\cdot+2)=1}\Lambda \Lambda_2^+(m).
				\end{align*}
				By~\eqref{eq_lem_finalex_2} there exists an arithmetic function $\widetilde{E}_1$ with
				\begin{align*}
					&\1_{\widetilde{\chi}(\cdot)=-1}\Lambda \Lambda_2^+*\1_{\widetilde{\chi}(\cdot)=\widetilde{\chi}(\cdot+2)=1}\Lambda \Lambda_2^+(m)\\
					\geq& c_1 (1-\widetilde{\beta})(\log N)\1_{\widetilde{\chi}(\cdot)=-1}\Lambda \Lambda_2^+*\1_{\widetilde{\chi}(\cdot)=\widetilde{\chi}(\cdot+2)=1}r_{P_1}r_{P_1}^+(m)+O(\Lambda \Lambda_2^+*\widetilde{E}_1(m))
				\end{align*}
				such that
				\begin{align*}
					|\Lambda \Lambda_2^+*\widetilde{E}_1(m)|\ll \delta_1^{-8} (1-\widetilde{\beta})(\log N) e^{-c \delta_1^{-1/2}} m\mathfrak{S}(m)
				\end{align*}
				for all natural numbers $m\in [5N/4,7N/4]$ with at most $O(N^{1-(\delta_1/10)^4})$ exceptions. 
				Using the non-negativity of $r_{P_1}(n)r_{P_1}(n+2)$ and~\eqref{eq_lem_finalex_1}, there exists an arithmetic function $\widetilde{E}_2$ such that
				\begin{align*}
					&c_1 (1-\widetilde{\beta})(\log N)\1_{\widetilde{\chi}(\cdot)=-1}\Lambda \Lambda_2^+*\1_{\widetilde{\chi}(\cdot)=\widetilde{\chi}(\cdot+2)=1}r_{P_1}r_{P_1}^+(m)\\
					\geq& c_1^2 (1-\widetilde{\beta})(\log N) \1_{\widetilde{\chi}(\cdot)=-1}r_{P_1} r^+_{P_1}* \1_{\widetilde{\chi}(\cdot)=\widetilde{\chi}(\cdot+2)=1}r_{P_1} r^+_{P_1}(m)\\
					&+(1-\widetilde{\beta})(\log N)\widetilde{E}_2* r_{P_1} r^+_{P_1}(m)
				\end{align*}
				 with
				\begin{align*}
					|(1-\widetilde{\beta})(\log N)\widetilde{E}_2* r_{P_1} r^+_{P_1}(m)|\ll (1-\widetilde{\beta})(\log N) e^{-c \delta_1^{-1/2}} m\mathfrak{S}(m),
				\end{align*}
				outside of an acceptable exceptional set of natural numbers $m$. By the second case of Proposition~\ref{prop_fundlem_exceptional},
				\begin{align*}
					&c_1^2 (1-\widetilde{\beta})(\log N) \1_{\widetilde{\chi}(\cdot)=-1}r_{P_1} r^+_{P_1}* \1_{\widetilde{\chi}(\cdot)=\widetilde{\chi}(\cdot+2)=1}r_{P_1} r^+_{P_1}(m)\\
					=&c_1^2 (1-\widetilde{\beta})(\log N)m\mathfrak{S}(m)\Bigl( \frac{1-\sigma_1(m)-\sigma_2(m)}{8}+O(\delta_1^{-3}(\log N)^{-1/3})\Bigr)\bigl(1+ O(e^{-c \delta_1^{-1}})\bigr).
				\end{align*}
				Since we can assume~\eqref{eq_MT1exceptional_assumptions} and since  $|\sigma_2(m)|\leq 1$, we have
				\begin{align*}
					\frac{1-\sigma_1(m)-\sigma_2(m)}{8}+O(\delta_1^{-3}(\log N)^{-1/3})\geq 1/16+O(\delta_1^{-3}(\log N)^{-1/3}).
				\end{align*}
				If $\delta_1$ is sufficiently small in absolute terms, together with~\eqref{eq:exczerolowerbound} and the fact that $\mathfrak{S}(m)\gg 1$ for $m\equiv 4 \pmod 6$, this shows that
				\begin{align*}
					\Lambda \Lambda_2^+*\Lambda \Lambda_2^+(m)\geq \frac{1}{20} (1-\widetilde{\beta})(\log N)m \mathfrak{S}(m)\gg N^{1-\delta_1^6}
				\end{align*}
				for all natural numbers $m\in [5N/4,7N/4], m\equiv 4\pmod 6$ that fulfil~\eqref{eq_MT1exceptional_assumptions}, apart from $O(N^{1-(\delta_1/10)^4})$ exceptions.  
                \end{proof}

\bibliographystyle{plain}
\bibliography{refs.bib}{}

		\end{document}